\documentclass[12pt]{amsart}
\usepackage[pdftex]{graphicx} 
\usepackage[utf8]{inputenc}
\usepackage{amsmath}
\usepackage{amsxtra}
\usepackage{amscd}
\usepackage{amsthm}
\usepackage{amsfonts}
\usepackage{amssymb}
\usepackage[mathscr]{eucal}
\usepackage{physics}
\usepackage{hyperref}
\hypersetup{
    colorlinks = true,
    allcolors = blue
}
\usepackage{mathtools}
\usepackage{tikz}
\usepackage{stmaryrd}
\usepackage{tikz-cd}
\usepackage{genyoungtabtikz}
\usepackage{multicol}
\usepackage{float}
\usepackage{exscale}

\newcommand{\s}{\mathbf{s}}
\newcommand{\E}{{\mathcal E}}
\newcommand{\Z}{{\mathbb{Z}}}

\newcommand{\C}{{\mathbb{C}}}
\newcommand{\Cx}{{\mathbb{C^\times}}}

\newcommand{\Es}{\E_{\s}}
\newcommand{\slhs}{ \widehat{\sl}_{\s}}
\newcommand{\slh}{\widehat{\mathfrak{sl}}}
\renewcommand{\sl}{\mathfrak{sl}}
\newcommand{\gl}{\mathfrak{gl}}

\newcommand{\Smn}{\mathcal{S}_{m|n}}

\newcommand{\Sym}{\mathrm{Sym}}

\newcommand{\lb}[1]{\llbracket #1 \rrbracket}

\newcommand{\Us}{U_q\widehat{\mathfrak{sl}}_\s}

\newcommand{\bs}{\boldsymbol}
\newcommand{\mc}{\mathcal}
\newcommand{\la}{\lambda}
\newcommand{\La}{\Lambda}

\numberwithin{equation}{section}
\newtheorem{thm}{Theorem}[section]
\newtheorem{prop}[thm]{Proposition}
\newtheorem{lem}[thm]{Lemma}

\newtheorem{dfn}[thm]{Definition}

\begin{document}

\title{Representations of quantum toroidal superalgebras and plane $\s$-partitions}

\author{Luan Bezerra and Evgeny Mukhin}

\address{LB: Shenzhen International Center for Mathematics, Southern University of Science and
Technology, Shenzhen, China}\email{bezerra.luan@gmail.com}

\address{EM: Department of Mathematics,
	Indiana University Indianapolis,
	402 N. Blackford St., LD 270,
	Indianapolis, IN 46202, USA}\email{emukhin@iupui.edu}

\begin{abstract}
We construct Fock and MacMahon modules for the quantum toroidal superalgebra $\Es$ associated with the Lie superalgebra $\gl_{m|n}$ and parity $\s$. The bases of the Fock and MacMahon modules are labeled by super-analogs of partitions and plane partitions with various boundary conditions, while the action of generators of $\Es$ is given by Pieri type formulas. We study the corresponding characters.
\end{abstract}

\maketitle

\section{Introduction}
Representations of quantum toroidal algebras occur in many areas of mathematics and mathematical physics. 
In string theory, modules of quantum toroidal algebras provide an algebraic viewpoint of brane webs \cite{AFS}, which can be seen as the toric diagrams of Calabi-Yau threefolds \cite{LV}. The so-called MacMahon modules describe BPS states in terms of crystal melting, and their characters are used to compute Donaldson-Thomas invariants, \cite{LY}, \cite{ORV}, \cite{OY},  \cite{Su}. The quantum toroidal algebra action on the cohomology of instantons moduli space plays a key role in the AGT correspondence  \cite{MO}. See \cite{LF} for a review on AGT correspondence.
In the simplest examples, the Fock modules of quantum toroidal algebras are related to equivariant K-theory of Hilbert schemes \cite{FT}, \cite{SV}, Laumon spaces \cite{FFNR}, and to Macdonald theory \cite{FHHSY}, \cite{FFJMM}.

In this paper, we study the representation theory of quantum toroidal superalgebras $\Es$ associated with the Lie superalgebra $\gl_{m|n}$ and parity $\s$, where $m\neq n$. These algebras have been already observed in the context of the BPS states, see \cite[Section 8.3]{LY}, \cite{Z}. While similar results are expected in the case of $m= n$, the appropriate definition of $\Es$ and the proof of it's properties for $m=n$ is missing at the moment. We plan to address this issue in a separate text.

The superalgebras $\Es$ (with $m\neq n$) for an arbitrary parity $\s=(s_1,\dots,s_{m+n})$, $s_i\in\{-1,1\}$, were introduced in \cite{BM2}. Fock representations of $\Es$ in the standard parity are given in \cite{BM1}. In contrast to \cite{BM1}, which deals with ``perpendicular" generators acting on the Fock modules by vertex operators, the present paper studies the modules via ``horizontal" generators, which act by Pieri type formulas. The ``horizontal" formulation is given in terms of combinatorics of partitions, plane partitions, and their supersymmetric analogs. This is the form used in the multiple applications mentioned above.

Logically, we follow the path of \cite{FJMM}, where the representations of quantum toroidal algebras were constructed in the even case ($n=0$). Namely, we start with modules of level zero (vector and covector representations), which are easy to construct explicitly, then we use the semi-infinite wedge construction to obtain the level one (or minus one) Fock modules, and then we do the inductive limit again to build the MacMahon modules of general level. However, we had to make several important improvements to the constructions of \cite{FJMM} to cover the supersymmetric case. In particular, in the even case, our method is applicable for a larger class of modules (covering a family of modules of level $-1$ and a family of non-integrable modules of level $1$).

In the construction of the Fock space in the even case, the partitions representing the basis of the module are constructed by rows, and each row corresponds to a basis element in the vector representation. In contrast, we represent partitions using Frobenius coordinates where rows correspond to vectors in vector representations while columns correspond to vectors in covector representations. An example of such a picture is given in Figure \ref{partition pic}. 

\begin{figure}[H]
\centering
\begin{tikzpicture}
\node at (-40pt,7pt) {$\ket{(8,5,5,4,4,2),(7,2,1,1)}$:};
\tgyoung(40pt,0cm,01234012,401234,3401234,2340123,1::;40123,0::::;01,4,3)
{\Yfillcolour{gray}
\Yfillopacity{.3}
\tgyoung(40pt,0pt,_8,:_5,::_5,:::_4,::::_4,:::::_2)}
{\Yfillcolour{blue}
\Yfillopacity{.35}
\tgyoung(40pt,0pt,:,|7:,:|2:,::;:,:::;)}
\end{tikzpicture}
\caption{Young diagram depicting the $\s$-partition corresponding to the vector $\ket{(8,5,5,4,4,2),(7,2,1,1)}$. }
\end{figure}\label{partition pic}
Here the gray rows form the partition $(8,5,5,4,4,2)$ corresponding to a vector in the $6$-th skew-symmetric power of the vector representation, and the blue columns form the partition $(7,2,1,1)$ corresponding to a vector in the $4$-th skew-symmetric power of the covector representation. Numbers in the boxes are called colors and give a grading. Naturally, $v\wedge v$ (or repetition in the partition) is allowed if and only if $v$ is odd. In our example, we have $\gl_{3,2}$ in the standard parity, which translates to rows ending with colors $3$ and $4$ being odd in the gray area, and columns ending with colors $0$ and $4$ being odd in the blue area. In examples appearing in the $\Es$ modules, the parity of a row ending with a box of color $i$ in the gray area is equal to the parity of a column ending with a box of color $i+1$ in the blue area. In our pictures, the gray and blue areas are not completely symmetric since the diagonal is gray, and the number of the gray rows is never smaller than the number of the blue columns. More symmetrically, one can think that half of each box on the main diagonal is in the gray area, and the other half is in the blue area.  We call such shapes $\s$-partitions. If all parities are even $\s\equiv 1$, then the $\s$-partitions are just usual partitions. See Section \ref{sec: pictures} for more details and pictures.

Each $\s$-partition corresponds to an eigenvector of the Cartan currents $K_i^\pm(z)$ in a Fock module. The creation currents $F_i(z)$ act by adding boxes of color $i$, and the annihilation currents $E_i(z)$ act by removing boxes of colors $i$. All matrix coefficients are explicit.
For example, a matrix coefficient of $F_i(z)$ is always a delta function multiplied by a constant, and the support of the delta function is determined by the coordinates of the box being added.

Altogether, we describe in a similar way a family of $\E_s$-modules which after restriction to $U_q\widehat{\gl}_{m|n}$ are highest weight modules of highest weight  $s_i(r \Lambda_{i-1}-(r+1)\Lambda_i)$, $i\in \hat I$, $r\in\Z$, where $\Lambda_i$ are fundamental weights and $s_i$ are the components of the parity $\s$. In particular, it includes the highest weights $\pm\Lambda_i$. We expect these modules remain irreducible when restricted to $U_q\widehat{\gl}_{m|n}$. We show this is the case for modules corresponding to the highest weights $\Lambda_i$, $i=0,\dots,m$, $(r+1)\Lambda_0-r\Lambda_{-1}$,  $(r+1)\Lambda_m-r\Lambda_{m+1}$, $r\in\Z_{>0}$, in the standard parity  by comparing  the characters. 

\medskip

The construction of MacMahon modules out of Fock spaces is parallel to the even case. We take tensor products of the Fock spaces and appropriately pass to an inductive limit. The basis vectors in a MacMahon module correspond to plane $\s$-partitions. The simplest example is shown in Figure \ref{3d-intro}.

\begin{figure}[H]
\centering
\begin{tikzpicture}

\draw[-] (90pt,13pt)--++(-20:52pt)--++(-90:13pt)--++(-20:39pt)--++(-90:39pt);
\draw[-] (90pt,13pt)++(-140:13pt)--++(-20:52pt)--++(-90:13pt)--++(-20:39pt)--++(-90:39pt);
\draw[-] (90pt,13pt)++(-140:13pt)++(-20:52pt)++(-90:13pt)--++(-90:13pt)--++(-20:39pt);
\draw[-] (90pt,13pt)++(-140:26pt)++(-90:13pt)--++(-20:2.55pt);
\draw[-] (90pt,13pt)++(-140:13pt)++(-20:52pt)++(-90:26pt)--++(-90:13pt)--++(-20:39pt);
\draw[-] (90pt,13pt)++(-140:13pt)++(-20:52pt)++(-90:39pt)++(-20:13pt)--++(-90:12pt)++(-90:1pt)++(-20:2.55pt)--++(-20:23.45pt);
\draw[-] (90pt,13pt)++(-140:26pt)--++(-20:52pt)--++(-90:39pt)--++(-20:26pt)--++(-90:13pt);
\draw[-] (90pt,13pt)++(-140:39pt)++(-20:13pt)--++(-20:13pt)--++(-90:26pt)--++(-20:26pt)--++(-90:13pt)--++(-20:26pt)--++(-90:13pt);
\draw[-] (90pt,13pt)++(-140:26pt)++(-20:26pt)--++(-90:13pt)--++(-20:26pt);
\draw[-] (90pt,13pt)++(-140:26pt)++(-20:26pt)++(-90:13pt)--++(-90:13pt)--++(-20:26pt);
\draw[-] (90pt,13pt)++(-140:39pt)++(-90:13pt)--++(-20:26pt);
\draw[-] (90pt,13pt)++(-140:39pt)++(-90:26pt)++(-20:26pt)--++(-90:13pt)--++(-20:26pt);
\draw[-] (90pt,13pt)++(-140:52pt)++(-90:13pt)--++(-20:26pt)--++(-90:26pt)--++(-20:39pt)--++(-90:13pt);
\draw[-] (90pt,13pt)++(-140:52pt)++(-90:39pt)++(-20:26pt)--++(-90:13pt)--++(-20:15.55pt);
\draw[-] (90pt,13pt)++(-140:65pt)++(-90:13pt)--++(-20:26pt)++(-90:13pt)--++(-90:26pt);
\draw[-] (90pt,13pt)++(-140:65pt)++(-90:13pt)--++(-90:13pt)--++(-20:26pt);
\draw[-] (90pt,13pt)++(-140:65pt)++(-90:39pt)--++(-20:26pt)++(-140:13pt)--++(-90:13pt);
\draw[-] (90pt,13pt)++(-140:65pt)++(-20:52pt)++(-90:39pt)--++(-20:13pt)--++(-90:13pt)--++(160:13pt)--++(90:13pt);
\draw[-] (90pt,13pt)++(-140:91pt)++(-20:26pt)++(-90:39pt)--++(-90:13pt);
\draw[-] (90pt,13pt)++(-140:78pt)++(-90:39pt)--++(-20:26pt);
\draw[-] (90pt,13pt)++(-140:91pt)++(-90:52pt)++(-20:13pt)--++(-20:13pt);

\draw[-] (90pt,13pt)--++(-140:26pt)--++(-90:13pt)--++(-140:39pt)--++(-90:26pt)--++(-140:13pt)--++(-90:13pt);
\draw[-] (90pt,13pt)++(-20:13pt)--++(-140:39pt)--++(-90:13pt)--++(-140:26pt)--++(-90:26pt)--++(-140:26pt)--++(-90:13pt);
\draw[-] (90pt,13pt)++(-20:26pt)--++(-140:39pt)--++(-90:13pt)--++(-140:26pt)--++(-90:13pt);
\draw[-] (90pt,13pt)++(-20:26pt)++(-140:39pt)++(-90:13pt)--++(40:13pt)++(-90:13pt)--++(-140:39pt)++(-90:13pt)++(40:26pt)--++(-140:52pt)--++(160:13pt)++(-20:13pt)++(-90:13pt)--++(40:39pt);
\draw[-] (90pt,13pt)++(-20:39pt)--++(-140:26pt)--++(-90:26pt)--++(-140:13pt)--++(-90:13pt)--++(-140:13pt)--++(-90:13pt);
\draw[-] (90pt,13pt)++(-20:52pt)--++(-140:26pt)++(-90:26pt)--++(-140:13pt)++(-90:13pt)--++(-140:26pt);
\draw[-] (90pt,13pt)++(-20:52pt)++(-90:13pt)--++(-140:26pt)++(-90:13pt)--++(40:13pt)++(-90:13pt)--++(-140:26pt);
\draw[-] (90pt,13pt)++(-20:65pt)++(-90:13pt)--++(-140:13pt)--++(-90:26pt)--++(-140:52pt);
\draw[-] (90pt,13pt)++(-20:78pt)++(-90:13pt)--++(-140:13pt)--++(-90:39pt);
\draw[-] (90pt,13pt)++(-20:91pt)++(-90:13pt)--++(-140:13pt)++(-90:13pt)--++(40:13pt)++(-90:13pt)--++(-140:13pt)++(-90:13pt)--++(40:13pt);
\draw[-] (90pt,13pt)++(-20:78pt)++(-90:39pt)++(-140:26pt)--++(-140:13pt)++(160:13pt)--++(-90:13pt);
\draw[-] (90pt,13pt)++(-20:78pt)++(-90:52pt)++(-140:26pt)--++(-140:13pt)--++(160:13pt)--++(-140:26pt);

\draw[-] (90pt,6.5pt) node[scale=0.75]{$0$}++(-20:13pt) node[scale=0.75]{$1$}++(-20:13pt) node[scale=0.75]{$2$}++(-20:13pt) node[scale=0.75]{$3$}++(-90:13pt)++(-20:13pt) node[scale=0.75]{$4$}++(-20:13pt) node[scale=0.75]{$0$}++(-20:13pt) node[scale=0.75]{$1$};

\draw[-] (90pt,6.5pt)++(-140:13pt) node[scale=0.75]{$4$}++(-20:13pt) node[scale=0.75]{$0$}++(-20:13pt) node[scale=0.75]{$1$}++(-20:13pt) node[scale=0.75]{$2$}++(-90:39pt)++(-20:13pt) node[scale=0.75]{$3$};

\draw[-] (89pt,6.5pt)++(-140:26pt)++(-90:13pt) node[scale=0.75]{$3$};

\draw[-] (90pt,6.5pt)++(-140:26pt)++(-20:13pt) node[scale=0.75]{$4$}++(-20:13pt)++(-90:26pt) node[scale=0.75]{$0$}++(-20:13pt) node[scale=0.75]{$1$}++(-20:13pt)++(-90:13pt) node[scale=0.75]{$2$}++(-20:13pt) node[scale=0.75]{$3$};

\draw[-] (90pt,6.5pt)++(-140:39pt)++(-90:13pt) node[scale=0.75]{$2$}++(-20:13pt) node[scale=0.75]{$3$}++(-20:13pt)++(-90:26pt) node[scale=0.75]{$4$}++(-20:13pt) node[scale=0.75]{$0$}++(-20:13pt) node[scale=0.75]{$1$};

\draw[-] (90pt,6.5pt)++(-140:52pt)++(-90:13pt) node[scale=0.75]{$1$}++(-20:13pt) node[scale=0.75]{$2$}++(-20:39pt)++(-90:26pt) node[scale=0.75]{$0$};

\draw[-] (90pt,6.5pt)++(-140:65pt)++(-90:39pt) node[scale=0.75]{$0$}++(-20:13pt) node[scale=0.75]{$1$}++(-140:13pt) node[scale=0.75]{$0$};

\filldraw[fill=gray, fill opacity=0.3,line width=0.001pt] (90pt,13pt)--++(-20:52pt)--++(-90:13pt)--++(-20:39pt)--++(-90:39pt)--++(-140:13pt)--++(160:23.6pt)--++(-90:12pt)--++(-140:13pt)--++(160:13pt)--++(-140:26pt)--++(90:13pt)--++(160:13pt)--++(40:13pt)--++(160:13pt)--++(40:13pt)--++(90:13pt)--++(160:13pt)--++(40:13pt)--++(90:26pt)--++(160:13pt)--++(40:13pt)--++(160:13pt)--++(40:13pt);
\filldraw[fill=blue, fill opacity=0.35,line width=0.001pt] (90pt,13pt)++(-140:13pt)--++(-140:13pt)--++(-90:13pt)--++(-140:39pt)--++(-90:26pt)--++(-140:13pt)--++(-90:13pt)--++(-20:2.5pt)--++(-90:12pt)--++(-20:13pt)--++(40:39pt)--++(-20:15.4pt)--++(-90:12pt)--++(-20:13pt)--++(90:13pt)--++(160:13pt)--++(40:13pt)--++(160:13pt) --++(40:13pt)--++(90:13pt)--++(160:13pt)--++(40:13pt)--++(90:26pt)--++(160:13pt)--++(40:13pt)--++(160:13pt);

\draw[->] (90pt,13pt)++(-20:91pt)++(-90:52pt)--++(-20:26pt);
\draw[-] (90pt,13pt)++(-20:115pt)++(-90:42pt) node[scale=0.75]{$i$};
\draw[->] (90pt,13pt)++(-140:78pt)++(-90:52pt)--++(-140:26pt);
\draw[-] (90pt,13pt)++(-140:102pt)++(-90:42pt) node[scale=0.75]{$j$};
\draw[->] (90pt,13pt)--++(90:20pt);
\draw[-] (96pt,26pt) node[scale=0.75]{$k$};
\end{tikzpicture}
\caption{ }
\label{3d-intro}
\end{figure}
Each horizontal layer of a plane $\s$-partition is a $\s$-partition corresponding to a vector in a Fock space. The layers are stacked up with the rule:
if the color $i$ can be repeated in the gray (resp. blue) horizontal layer, then it cannot be repeated on the vertical facet looking in the $i$-th (resp. the $j$-th) direction. So, in our example, the only colors which can be repeated vertically are $0,1$ and $2$ in the gray area, and $1,2$ and $3$ in the blue area. 

The rules for the action of $\Es$ are similar: $F_i(z)$ remove boxes of color $i$, $E_i(z)$ add boxes of color $i$, and $K_i^\pm(z)$ do not change the plane $\s$-partition.

We also construct more general MacMahon modules whose basis is parameterized by plane $\s$-partitions with boundaries, see Figure \ref{3d boundary}. The vertical boundary is determined by a so-called ``colorless self-comparable" $\s$-partition and the horizontal boundary by a usual partition. Note that, in the even case, the horizontal boundary can be extended to a pair of arbitrary partitions. 

\medskip

We discuss the characters of the Fock and MacMahon modules. The characters of Fock spaces do not factorize (unlike the even case),  and we give fermionic type formulas. The characters for pure MacMahon modules do factorize and look similar to the MacMahon formula for plane partitions. We give a proof based on the $\gl_{m|n}\otimes \gl_{p|q}$ duality of \cite{S}, \cite{CW}.

We note that the number of $\s$-partitions with $n$ boxes is larger (or equal) than the number of partitions (as every partition is an $\s$-partition), while the number of plane $\s$-partitions with $n$ boxes is smaller (or equal) than the number of plane partitions. 

For special levels, the MacMahon module has a quotient whose basis is given by all plane $\s$-partitions which do not contain a fixed box. It is an interesting question to compute characters of MacMahon modules with boundaries and prohibited boxes. In general, we think that $\s$-partitions and plane $\s$-partitions deserve a thorough combinatorial study.

\medskip

The quantum toroidal superalgebras $\Es$ are isomorphic for different parities $\s$. It would be interesting to study the Fock and MacMahon modules with respect to these isomorphisms.

\medskip

The paper is organized as follows. In Section \ref{Es sec}, we recall generalities about the quantum toroidal superalgebras $\Es$ and their modules. We construct and study vector and covector representations, Fock modules, and MacMahon modules in Sections \ref{vec sec}, \ref{Fock sec}, and \ref{MacMahon sec}, respectively.  Section \ref{sec: pictures} is devoted to pictures and combinatorial descriptions of bases in the constructed modules. We give characters of vector, covector, Fock, and pure MacMahon modules in Section \ref{char sec}. 

\section{Quantum Toroidal Superalgebras}\label{Es sec}
\subsection{Parities}\label{parities sec}
We work over the field of complex numbers $\mathbb{C}$.

A superalgebra is a $\Z_2$-graded algebra $A=A_0\oplus A_1$.  Elements of $A_0$ are called even and elements of $A_1$ odd. We denote the parity of an element $v\in A_i$ by $|v|=i$, $i\in \Z_2$.

Fix $m, n \in \Z_{\geq 0}$, such that $m\neq n$ and $N=m+n\geq 3$.

Consider the Lie superalgebras $\sl_{m|n}$ and $\slh_{m|n}$. The set of Dynkin nodes are $I=\{1,2,\dots,N-1\}$ and $\hat I=\{0,1,\dots,N-1\}$, respectively. The elements of $\hat I$ are always considered modulo $N$.
It is well-known that there are different choices of the root system, which lead to different Cartan matrices and different Dynkin diagrams.  Such choices are parameterized by the set of $N$-tuples of $\pm 1$ with exactly $m$ positive coordinates. Set
$$\Smn=\{(s_1,\dots,s_N)\ |\  s_i\in\{-1,1\},\ \#\{i\,|\,s_i=1\}=m\}.$$
An element $\s=(s_1,\dots,s_N)\in \Smn$ is called a {\it parity sequence}. 
We extend the parity sequences to the sequences $ \s=(s_i)_{i\in{\Z}}$ by demanding periodicity: $s_{i+N}=s_i$ for all $i$.

The parity sequence of the form $\s=(1,\dots,1,-1,\dots,-1)$ is called the {\it standard parity sequence}. 

\medskip

Given a parity sequence $\s \in \Smn$, we have a Cartan matrix $A^{\s}=(A^{\s}_{i,j})_{i,j\in I}$ and an affine Cartan matrix $\hat{A}^{\s}=(A^{\s}_{i,j})_{i,j\in \hat{I}}$, where
\begin{align}{\label{aff cartan}}
	  & A^{\s}_{i,j}=(s_i+s_{i+1})\delta_{i,j}-s_i\delta_{i,j+1}-s_j\delta_{i+1,j} & (i,j \in \hat{I}). 
\end{align}
The parity of a node $i\in \hat I$ is given by $|i|=(1-s_is_{i+1})/2$.

Denote $\sl_{\s}$ the superalgebra corresponding to Cartan matrix $A^{\s}$.
Denote $\slhs$ the superalgebra corresponding to Cartan matrix $\hat{A}^{\s}$. 
The superalgebras $\sl_{\s}$ are all isomorphic to $\sl_{m|n}$ and the superalgebras $\slhs$  to $\slh_{m|n}$. By abuse of notation we often omit the suffix $\s$ if the parity sequence is clear from the context.

Let $P_\s$ be the integral lattice with basis ${\varepsilon}{_i}$, $i\in \hat I$. We have a bilinear form given by
\begin{align*}
	  & \braket{{\varepsilon}{_i}}{{\varepsilon}{_j}}=s_i\delta_{i,j} & (i,j \in \hat{I}). 
\end{align*}

Let $\Delta_\s=\{\alpha_i:={\varepsilon}{_i}-{\varepsilon}{_{i+1}}\,|\,i\in I\}$ be the set of simple roots of $\sl_{\s}$, and let $Q_\s$ be the root lattice of $\sl_{\s}$. Let also ${\delta}{}$ be the null root of $\slhs$.  We have $\braket{{\delta}{}}{{\delta}{}}=\braket{ {\delta}{}}{ {\alpha}{_i}}=0$, $i\in I$. Set $ {\alpha}{_0}= {\delta}{}+ {\varepsilon}{_N}- {\varepsilon}{_1}$. Then, $\hat{\Delta}_\s=\{ {\alpha}{_i}|i\in \hat{I}\}$ is the set of simple roots of $\slhs$. Let $\widehat Q_\s$ be the root lattice of $\slhs$. Note that $\braket{ {\alpha}{_i}}{ {\alpha}{_j}}= A_{i,j}^{\s}$, $i,j\in \hat{I}$, and the parity of a simple root $\alpha_i$ is given by $|\alpha_i|=|i|=(1-s_is_{i+1})/2$.

\medskip
For $\sigma_1,\sigma_2\in\{+,-\},$ we call $\bs \lambda=(\lambda_1,\lambda_2,\dots,\lambda_k)$ a {\it $\s^{\sigma_1}_{\sigma_2}\ $-partition with $k$ parts} if  $\lambda_i\in \Z_{>0}$, $\lambda_i\geq \lambda_{i+1}$, $i=1,\dots,k-1,$ where the equality is allowed if and only if $s_{\sigma_2\lambda_i}=\sigma_1$. Note that, abusing notation, we use $\sigma_i$ as $+,-$ or as $1,-1$. We call the number of parts $k$ the length of the partition and denote it by $\ell(\bs\lambda)$.

\subsection{Definition of \texorpdfstring{$\Es$}{Es}}
The quantum toroidal superalgebra associated with $\gl_{m|n}$ was introduced in \cite{BM1} with standard parity, and in \cite{BM2} for any choice of parity.

Fix $d,q \in \mathbb{C}^\times$ and define 
$$q_1=d\,q^{-1},\quad q_2=q^2,\quad q_3=d^{-1}q^{-1}.$$

Note that $q_1q_2q_3=1$. In this paper, we always assume that $q_1,q_2$ are generic, meaning that $q_1^{n_1}q_2^{n_2}q_3^{n_3}=1$, $n_1, n_2,n_3\in \mathbb{Z}$, if and only if  $n_1=n_2=n_3$. Fix also square roots $d^{1/2}, q^{1/2}\in \mathbb{C}^\times$ so that $(d^{1/2})^2=d,\; (q^{1/2})^2=q$.

Recall the affine Cartan matrix $\hat A^\s=(A_{i,j}^\s)_{i,j,\in \hat I}$, see \eqref{aff cartan}.

We also define the matrix ${M}^{\s}=({M}^{\s}_{i,j})_{i,j\in \hat{I}}$ by $M^{\s}_{i+1,i}=-M^{\s}_{i,i+1}=s_{i+1}$, and $M^{\s}_{i,j}=0$, $i\neq j\pm 1$.

\begin{dfn}{\label{defE}}
	The \textit{quantum toroidal superalgebra associated with $\gl_{m|n}$} and parity sequence $\s$ is the unital associative superalgebra $\Es=\Es(q_1,q_2,q_3)$  generated by $E_{i,r},F_{i,r},H_{i,r'}$, 
	and invertible elements $K_i$, where $i\in \hat{I}$, $r,r'\in \Z$, $r'\neq 0$, subject to the defining relations \eqref{relCK}--\eqref{Serre6} below.
	The parity of the generators is given by $|E_{i,r}|=|F_{i,r}|=|i|=(1-s_is_{i+1})/2$, and all remaining generators have parity $0$. 
\end{dfn}

We use generating series
\begin{align*}
	E_i(z)=\sum_{k\in\Z}E_{i,k}z^{-k}, \quad                                            
	F_i(z)=\sum_{k\in\Z}F_{i,k}z^{-k}, \quad                                            
	K^{\pm}_i(z)=K_i^{\pm1}\exp\Bigl(\pm(q-q^{-1})\sum_{r>0}H_{i,\pm r}z^{\mp r}\Bigr). 
\end{align*}
	
\medskip
Let also $\displaystyle{\delta\left(z\right)=\sum_{n\in \mathbb{Z}} z^n}$ be the formal delta function.
\medskip
	
Then the defining relations are as follows.
\bigskip

\noindent{\bf $K$ relations}
\begin{align}{\label{relCK}}
	  &   & K_iK_j=K_jK_i,                             
	  &   & K_iE_j(z)K_i^{-1}=q^{A_{i,j}^{\s}}E_j(z),  
	  &   & K_iF_j(z)K_i^{-1}=q^{-A_{i,j}^{\s}}F_j(z). 
\end{align}
\medskip
	
\noindent{\bf $K$-$K$, $K$-$E$ and $K$-$F$ relations}
\begin{align}
	  & K^\pm_i(z)K^\pm_j (w) = K^\pm_j(w)K^\pm_i(z),
	\qquad K^-_i(z)K^+_j (w) 
	=
	K^+_j(w)K^-_i (z),
	\label{KK2}\\
	  & (d^{M_{i,j}^{\s}}z-q^{A_{i,j}^{\s}}w)K_i^\pm(z)E_j(w)=(d^{M_{i,j}^{\s}}q^{A_{i,j}^{\s}}z-w)E_j(w)K_i^\pm(z),   
	\label{KE}\\
	  & (d^{M_{i,j}^{\s}}z-q^{-A_{i,j}^{\s}}w)K_i^\pm(z)F_j(w)=(d^{M_{i,j}^{\s}}q^{-A_{i,j}^{\s}}z-w)F_j(w)K_i^\pm(z). 
	\label{KF}
\end{align}
\medskip

\noindent{\bf $E$-$F$ relations}
\begin{align}\label{EF}
	  & [E_i(z),F_j(w)]=\frac{\delta_{i,j}}{q-q^{-1}} 
	\big(\delta\left(\frac{w}{z}\right)K_i^+(w)
	-\delta\left(\frac{z}{w}\right)K_i^-(z)\big).
\end{align}
\medskip
	
\noindent{\bf $E$-$E$ and $F$-$F$ relations}
\begin{align}
	  & [E_i(z),E_j(w)]=0\,, \quad [F_i(z),F_j(w)]=0\, \quad                                                                     & (A_{i,j}^{\s}=0),   \label{EE FF} 
	\\
	  & (d^{M_{i,j}^{\s}}z-q^{A_{i,j}^{\s}}w)E_i(z)E_j(w)=(-1)^{|i||j|}(d^{M_{i,j}^{\s}}q^{A_{i,j}^{\s}}z-w)E_j(w)E_i(z) \quad   & (A_{i,j}^{\s}\neq0),              
	\\
	  & (d^{M_{i,j}^{\s}}z-q^{-A_{i,j}^{\s}}w)F_i(z)F_j(w)=(-1)^{|i||j|}(d^{M_{i,j}^{\s}}q^{-A_{i,j}^{\s}}z-w)F_j(w)F_i(z) \quad & (A_{i,j}^{\s}\neq0).              
\end{align}
\medskip
	
\noindent{\bf Serre relations}
\begin{align}
	  & \Sym_{{z_1,z_2}}\lb{E_i(z_1),\lb{E_i(z_2),E_{i\pm1}(w)}}=0\,\quad & (A^{\s}_{i,i}\neq 0),\label{Serre1} \\ 
	  & \Sym_{{z_1,z_2}}\lb{F_i(z_1),\lb{F_i(z_2),F_{i\pm1}(w)}}=0\,\quad & (A^{\s}_{i,i}\neq 0),               
	\label{Serre2}
\end{align}
	
If $mn\neq 2$,
\begin{align}
	  & \Sym_{{z_1,z_2}}\lb{E_i(z_1),\lb{E_{i+1}(w_1),\lb{E_i(z_2),E_{i-1}(w_2)}}}=0\,\quad & (A^{\s}_{i,i}= 0),\label{Serre3} 
	\\
	  & \Sym_{{z_1,z_2}}\lb{F_i(z_1),\lb{F_{i+1}(w_1),\lb{F_i(z_2),F_{i-1}(w_2)}}}=0\,\quad & (A^{\s}_{i,i}= 0).\label{Serre4} 
\end{align}
	
If $mn=2$,
\begin{align}
	\label{Serre5} & \Sym_{{z_1,z_2}}\Sym_{{w_1,w_2}}\lb{E_{i-1}(z_1),\lb{E_{i+1}(w_1),\lb{E_{{i-1}}(z_2),\lb{E_{i+1}(w_2),E_{i}(y)}}}}=    \\ \notag
	  & =\Sym_{{z_1,z_2}}\Sym_{{w_1,w_2}}\lb{E_{i+1}(w_1),\lb{E_{i-1}(z_1),\lb{E_{{i+1}}(w_2),\lb{E_{i-1}(z_2),E_{i}(y)}}}}\, &   & (A^{\s}_{i,i}\neq  0), \\
	\label{Serre6} & \Sym_{{z_1,z_2}}\Sym_{{w_1,w_2}}\lb{F_{i-1}(z_1),\lb{F_{i+1}(w_1),\lb{F_{{i-1}}(z_2),\lb{F_{i+1}(w_2),F_{i}(y)}}}}=    \\ \notag
	  & =\Sym_{{z_1,z_2}}\Sym_{{w_1,w_2}}\lb{F_{i+1}(w_1),\lb{F_{i-1}(z_1),\lb{F_{{i+1}}(w_2),\lb{F_{i-1}(z_2),F_{i}(y)}}}}\, &   & (A^{\s}_{i,i}\neq  0). 
\end{align}

Here, the bracket $[X,Y]$ is defined by the standard supersymmetric convention: for $X,Y$ homogeneous with respect to the parity it is given by
\begin{align*}
    [X,Y]=XY-(-1)^{|X||Y|} YX.
\end{align*} 
Additionally, the $q$-deformed bracket $\lb{\cdot,\cdot}$ is defined as follows.
Consider the root space decomposition of $\Es$
\begin{align*}
    \Es=\bigoplus_{\alpha \in \widehat Q_\s} (\Es)_{\alpha},
\end{align*}
where $(\Es)_{\alpha}:=\{X\in \Es\, |\, K_iXK^{-1}_i=q^{\braket{\alpha_i}{\alpha}},\, i\in \hat I\}$. Then, if $X\in (\Es)_{\alpha}$ and $Y\in (\Es)_{\beta}$, define 
\begin{align*}
    \lb{X,Y}=XY-(-1)^{|X||Y|} {q^{\braket{\alpha}{\beta}}}YX.
\end{align*}

\medskip
	
We have a central element $K$ given by
$$K:=K_0 K_1\cdots K_{N-1}.$$

The subalgebra  of $\Es$ generated by $E_i(z), F_i(z), K^\pm_i(z)$, $i\in I$, is isomorphic to $\Us$ at level zero (that is, with central element $c=1$, see for example Section 2.2 in \cite{BM2} for the description of $\slhs$). We denote this subalgebra by $U^{ver}_q\slhs$ and call it the {\it vertical quantum affine $\mathfrak{sl}_{m|n}$}.

The subalgebra  of $\Es$ generated by $E_{i,0}, F_{i,0}, K^{\pm 1}_i$, $i\in \hat I$, is isomorphic to $\Us$. We denote this subalgebra $U^{hor}_q\slhs$ and call it the {\it horizontal quantum affine $\mathfrak{sl}_{m|n}$}.
The  horizontal algebra contains the central element $K$.

\medskip

The general quantum toroidal superalgebra of \cite{BM2} has two central elements, $K$ and $C$. In this paper, we set $C=1$ since this happens in all representations we will be considering. In \cite{BM1}, we studied representations of $\Es$ where $K=1$ and $C$ was not one. Thus, one can think that our generators of $\Es$  in this paper correspond to images of generators of \cite{BM1} under the Miki automorphism, which exchanges $K$ and $C$, see Theorem 5.9 in \cite{BM2}. In particular, the vertical and horizontal subalgebras in this paper correspond to the horizontal and vertical subalgebras, respectively, in \cite{BM1}.

\medskip
It is proved in \cite{BM2} that the superalgebras $\Es$ are isomorphic for all parities $\s \in \Smn$ .

\medskip

		The superalgebra $\Es$ has a {\it topological comultiplication}\footnote{Note that our convention for the coproduct is opposite to that of \cite{BM2}.} given on generators by
		\begin{align}
		  & \Delta (E_i(z))=E_i(z)\otimes K^-_i(z) + 1\otimes E_i(z),      \notag                           \\
		  & \Delta (F_i(z))=F_i(z)\otimes 1 +  K^+_i(z)\otimes F_i(z), \label{coproduct}                                \\
		  & \Delta (K^\pm_i(z))=K^\pm_i(z)\otimes K^\pm_i(z). \notag
	\end{align}

	The corresponding {\it topological antipode} is given by
	\begin{align*}
	    S(E_i(z))=-E_i(z)\left(K^-_i(z)\right)^{-1},                                   \qquad  S(F_i(z))=-\left(K^+_i(z)\right)^{-1}F_i(z),                                   \qquad S(K^\pm_i(z))=\left(K^\pm _i(z) \right)^{-1}.
	\end{align*}
The map $\Delta$ is extended to an algebra homomorphism, and the map $S$ to a superalgebra anti-homomorphism, $S(xy)=(-1)^{|x||y|}S(y)S(x)$.

 Note that the tensor product multiplication is defined for homogeneous elements $x_1, x_2, y_1, y_2 \in\Es $ by
 \begin{align*}
     (x_1\otimes y_1)(x_2\otimes y_2)=(-1)^{|y_1||x_2|}x_1x_2\otimes y_1y_2
 \end{align*}
  and extended to the whole algebra by linearity.  
 
 Given $\Es$-modules $V_1,V_2$, the coproduct \eqref{coproduct} makes  $V_1\otimes V_2$ into $\Es$-module, provided it is well defined. Note the same sign rule applies:
\begin{align}\label{sign tensor}
     (x_1\otimes y_1)(v_1\otimes v_2)=(-1)^{|y_1||v_1|}x_1v_1\otimes y_1v_2,
 \end{align}
	for homogeneous elements $v_1\in V_1$, $v_2\in V_2$, $x_1,x_2\in\Es$.
	
	The vertical subalgebra $U^{ver}_q\widehat{\mathfrak{sl}}_{\s}$ 
	is a Hopf subalgebra of $\Es$.

\medskip

For a parity $\s$ define the parity $\s'$ by the rule $$s_i'=-s_{-i+1}.$$ We have $(\s')'=\s$.

\begin{lem}\label{lem tau} The quantum toroidal superalgebra $\Es(q_1,q_2,q_3)$ associated with $\gl_{m|n}$ and parity sequence $\s$ is isomorphic to the quantum toroidal superalgebra $\E_{\s '}(q_3^{-1},q_2^{-1},q_1^{-1})$ associated with $\gl_{n|m}$ and parity sequence $\s'$. More precisely, we have an isomorphism of toroidal superalgebras
\begin{align*}
    \tau_{\s}:\  & \Es(q_1,q_2,q_3) \to \E_{\s '}(q_3^{-1},q_2^{-1},q_1^{-1}), \\
    & E_i(z)\mapsto E_{-i}(z), \quad  F_i(z)\mapsto -F_{-i}(z)
 ,\quad  K_i^\pm(z)\mapsto K^\pm_{-i}(z).    
\end{align*}
\end{lem}
\begin{proof}
The lemma is proved by a direct computation.
\end{proof}
For the standard parity, the map $\tau_\s$ was given in \cite{BM1}.

\subsection{Generalities on the \texorpdfstring{$\Es$}{Es}-modules} 
Let $\Es^+$, $\Es^-$, $\Es^0$ be the subalgebras of $\Es$ generated by the coefficients of the series $E_i(z)$, $F_i(z)$, and $K^\pm_i(z)$, $i\in \hat I$, respectively.	

We expect to have the {\it  triangular decomposition} $\Es=\Es^-\otimes  \Es^0 \otimes \Es^+$\footnote{Some technical facts about $\E_\s$ are not written in the literature yet: e.g. $\E_s$ is a expected to be a flat deformation of the extension of $\gl_{m|n}[t_1,t_2]$ by a two dimensional center; $\E_\s$ is expected to be a Drinfeld double and, in particular, to have a universal $R$-matrix; etc. The triangular decomposition is one such fact. We do not use the existence of the triangular decomposition in this paper.}.

\medskip

The superalgebra $\Es$ is $\Z^{N}$-{\it graded} with 
\begin{align*}
    \deg E_i(z)=-1_i, \qquad \deg F_i(z)=1_i, \qquad \deg K^{\pm} (z)=0,
\end{align*}
where $1_i=(0,\dots,0,1,0,\dots,0)\in\Z^N$ with the $1$ in the $i$th position. 

\medskip
 Let  $V$ be a graded $\E_s$-module. We use the same notation for the degrees of vectors as in $\Es$. Then
 $$
 V=\mathop{\bigoplus}_{\bs l\in\Z^N} V_{\bs l}, \qquad V_{\bs l}=\{v\in V,\ \deg(v)=\bs l\}.
 $$
A graded module $V$ is called {\it quasi-finite} if all graded components $V_{\bs l}$ are finite-dimensional.

For a quasi-finite $\Es$-module  $V$, we call the sum
\begin{align}\label{character def}
    \chi(V;z_0,\dots,z_{N-1})=\sum_{\bs l=(l_0,\dots,l_{N-1})\in\Z^N} (\dim V_{\bs l})\,\prod_{i=0}^{N-1}{z_i^{l_i}}.
\end{align}
the {\it character of $V$}.

We say a vector $v\in V$ has {\it finite weight $\Lambda=\sum_{i\in \hat I}k_i\Lambda_i$,} if $K_{i}v_i=q^{k_i}v$ for all $i\in \hat I$.
 
 A vector $v\in V$ is called a {\it weight vector with weight $\phi^\pm$}, $\phi^\pm=(\phi_0^\pm(z),\dots,\phi^\pm_{N-1}(z))$, $\phi_i^\pm\in\C[z^{\mp 1}]$ if
 \begin{align*}
     K_i^\pm(z)v=\phi_i^\pm(z)v \qquad (i\in \hat I).
 \end{align*}
 Note that we have $K_iv=\phi_i^+(\infty)v=1/\phi_i^-(0)v$.
 
If the space $V_{\phi^\pm}=\{ v\in V,\  K_i^\pm(z)v=\phi_i^\pm(z)v,\ i\in \hat I\}$ is non-zero, then it is called the {\it weight space of $V$ weight $\phi^\pm$}.
   
An $\Es$-module $V$ is called {\it weighted} if $V$ has a basis consisting of weight vectors. It is important to keep in mind that any submodule of a weighted module is always weighted. Every weighted module is graded.

An $\Es$-module $V$ is called {\it tame} if $V$ is weighted and all weight spaces are 1-dimensional.

A vector $v\in V$ is called {\it singular} if $\Es^+v=0$.

An $\Es$-module is called {\it highest weight module of highest weight $\phi^\pm$} if it is generated by a singular vector of weight $\phi^\pm$. Clearly, irreducible sub-quotients of well-defined tensor products of highest weight modules are also highest weight modules.

The following proposition is standard.
\begin{prop}
For any $\phi^\pm(z)\in\C[[z^{\mp1}]]$ satisfying $\phi^+_i(\infty)\phi^-_i(0)=1$ for all $i\in \hat I$, there exists a unique irreducible highest weight $\E_\s$-module $L_{\phi^\pm}$ with highest weight $\phi^\pm(z)$.

If for each $i\in \hat I$ the series $\phi_i^\pm(z)$ are expansions of a rational function  $\phi_i(z)$ such that $\phi_i(\infty)\phi_i(0)=1$, then the irreducible highest weight module $L_{\phi^\pm}$ is quasi-finite. \qed
\end{prop}

An $\Es$-module $V$ has {\it level} $k\in \C^\times$ if the central element $K$ acts by the constant $q^k$. A well-defined tensor product of modules of levels $k_1$ and $k_2$ is a module of level $k_1+k_2$.

\section{Vector and covector representations}\label{vec sec}
\subsection{Definition of the vector and covector representations}
Fix $u\in \Cx$. We call $u$  the {\it evaluation parameter}. 

Given a parity $\s$, define a map $\ \bar{}\ :\ \Z \to \Z$ by
\begin{align*}
    \bar 0=0, \qquad  \overline {j+1}=\bar j+s_{j+1} \qquad  (j\in \Z). 
\end{align*}
We have the properties
\begin{align}
& \overline{ j}=\displaystyle{\sum_{i=1}^{j}}s_{ i} \qquad  (j\in \Z_{\geq 0}), \label{overline}\\
& \overline{j+m+n}=\overline{j}+m-n  \qquad  (j\in \Z).\label{periodic}
\end{align}
Since $m\neq n$, the map  $\ \bar{}\ :\ \Z \to \Z$ is surjective.

For $k\in\Z$,  let $$\psi_k(z)=\dfrac{q^k-q^{-k}z}{1-z}.$$ 
We have trivial but useful properties
$$
\psi_k(q_2^kz)\psi_k(z^{-1})=1, \qquad \psi_k(z^{-1})=\psi_{-k}(z), \qquad  \psi_l(q_2^{-k}z)\psi_k(z)=\psi_{k+l}(z).
$$
We set $\psi_1(z)=\psi(z)$. Then $\psi(q^2)=0$, $\psi(1)$ is not defined and $\psi(z)\psi(w)=1$ if $zw=q_2$.

\medskip

Let $V(u)$ be the vector superspace with basis $[u]_j$, $j\in \mathbb{Z}$, where the parity is given by $\left|[u]_{j}\right|=(1-s_{j+1})/2$.

\begin{lem}
The superspace $V(u)$ has an irreducible tame $\Es$-module structure of level zero given by
\begin{align*}
    &E_i(z)[u]_{j}=\begin{cases}
                        \delta \left(q_1^{-\bar{j}}u/z\right)[u]_{j-1},     & i\equiv j \pmod{m+n}; \\
                        0,  & \text{otherwise};
                    \end{cases}\\
    &F_i(z)[u]_{j-1}=\begin{cases}
                    s_{j}\, \delta \left(q_1^{-\bar{j}}u/z\right)[u]_{j},     & i\equiv j \pmod{m+n}; \\
                    0,  & \text{otherwise};
                \end{cases}\\
    &K^\pm_i(z)[u]_j=\begin{cases}
                         \psi_{-s_{j+1}} \left(q_1^{-\overline{j}}u/z\right)[u]_{j},     & i\equiv j \pmod{m+n}; \\
                        \psi_{s_{j+1}} \left(q_1^{-(\overline{j+1})}u/z\right)[u]_{j},     & i\equiv j+1 \pmod{m+n}; \\
                        [u]_{j},  & \text{otherwise}.
                    \end{cases}
\end{align*}\qed
\end{lem}

Let $W(u)$ be the vector superspace with basis $[u]^j$, $j\in \mathbb{Z}$, where the parity is given by $|[u]^{j}|=(1-s_j)/2$.

\begin{lem}
The superspace $W(u)$ has an irreducible tame $\Es$-module structure of level zero given by
\begin{align*}
    &E_i(z)[u]^{j}=\begin{cases}
      \delta \left(q_3^{\bar{j}}u/z\right)[u]^{j+1},     & i\equiv j \pmod{m+n}; \\
       0,  & \text{otherwise};
    \end{cases}\\
    &F_i(z)[u]^{j+1}=\begin{cases}
     s_{j+1}\, \delta \left(q_3^{\bar{j}}u/z\right)[u]^{j},     & i\equiv j \pmod{m+n}; \\
       0,  & \text{otherwise};
    \end{cases}\\
    &K^\pm_i(z)[u]^j=\begin{cases}
     \psi_{-s_j} \left(q_3^{\overline{j}}u/z\right)[u]^{j},     & i\equiv j \pmod{m+n}; \\
     \psi_{s_{j}} \left(q_3^{(\overline{j-1})}u/z\right)[u]^{j},     & i\equiv j-1 \pmod{m+n}; \\
       [u]^{j},  & \text{otherwise}.
    \end{cases}
\end{align*}\qed
\end{lem}

We call $V(u)$ the {\it vector representation}\footnote{In \cite{FJMM} the counterparts of the vector and covector representations are denoted by $V^{(k)}(u)$, $W^{(k)}(u)$. In particular, there is an extra parameter $k\in \hat{I}$. The modules $V^{(k)}(u)$ are all isomorphic to each other up to an appropriate shift of the evaluation parameter $u$. The same is true for $W^{(k)}(u)$. In this paper we do not use this extra parameter $k$.} of $\Es$ and $W(u)$ the {\it covector representation} of $\Es$.

The vector and covector representations $V(u)$ and $W(u)$ are periodic. Namely, we have the following simple lemma.

\begin{lem}\label{lem: periodic}
There exist isomorphisms of $\Es$-modules:
\begin{align*}
    &V(uq_1^{m-n})\to V(u), \qquad [uq_1^{m-n}]_i\mapsto [u]_{i-m-n},  \\
    &W(uq_3^{m-n})\to W(u), \qquad [uq_3^{m-n}]^i\mapsto [u]^{i+m+n}.
\end{align*}
\end{lem}
\begin{proof}
The lemma follows immediately from \eqref{periodic}.
\end{proof}

The vector and covector representations are related as follows.
Recall the isomorphism $\tau_\s$, see Lemma \ref{lem tau}.
Let $V'(u)$ be the vector representation of the  algebra $\E_{\s'}(q_3^{-1},q_2^{-1},q_1^{-1})$. We denote the basis of $V'(u)$ by $[u]_j'$.
\begin{lem}\label{lem:V-W}
The covector representation $W(u)$ of $\Es$ is obtained from the  vector representation  $V'(u)$ of  $\E_{\s'}(q_3^{-1},q_2^{-1},q_1^{-1})$ via twisting by the isomorphism $\tau_\s$. Namely, define a linear isomorphism of vector spaces
\begin{align*}
    \tau_W: \  &W(u)\to V'(u), \\
     &[u]^{j} \mapsto [u]_{-j}'.
\end{align*}
Then, 
\begin{align*}
g\, w =\tau_\s(g)\,\tau_W (w), \qquad w\in  W(u), \  g\in \Es.
\end{align*}
\end{lem}
\begin{proof}
The lemma is proved by a direct computation noting that $\overline{(-j)}$ in parity $\s'$ equals $\bar j$ in parity $\s$.
\end{proof}

\medskip

\subsection{Tensor products of vector and covector representations}
In this section, we use the comultiplication \eqref{coproduct} to define an action of $\Es$ on the tensor products of  representations $V(u)\otimes V(v)$, $W(u)\otimes W(v)$, $V(u)\otimes W(v)$ and find irreducible submodules. 

Note that the comultiplication involves infinite sums and, 
therefore, it can be used only under the condition that these infinite sums are convergent. We describe these conditions explicitly.

\begin{lem}{\label{lem VV}}
Let $u,v \in \Cx$.
\begin{enumerate}
    \item The comultiplication \eqref{coproduct} defines an action of $\Es$ on $V(u)\otimes V(v)$ if and only if $u\neq q_1^{(m-n)r} v$, $r\in \Z$.   \label{lem VV.i}
     
     \item If the $\Es$-module $V(u)\otimes V(v)$ is well-defined, then it is a tame module of level zero.
    
    \item If $v\neq q_2^{\pm 1}q_1^{(m-n)r}u$, the module $V(u)\otimes V(v)$ is irreducible.
    \item If $v=q_2q_1^{(m-n)r}u$, the module $V(u)\otimes V(v)$ has an irreducible submodule spanned by $[u]_j\otimes[v]_k$, with $j-k\geq (m+n)r$, where the equality is allowed only if $s_{j+1}=-1$. The quotient module is irreducible.
    
    \item If $v=q_2^{-1}q_1^{(m-n)r}u$, the module $V(u)\otimes V(v)$ has an irreducible submodule spanned by $[u]_j\otimes[v]_k$, with $j-k\geq (m+n)r$, where the equality is allowed only if $s_{j+1}=1$. The quotient module is irreducible.
    
\end{enumerate}
\begin{proof} The non-trivial matrix coefficients of action of $E_k(z)$  and $F_k(z)$ are 
\begin{align}
   [u]_i\otimes [v]_j \to [u]_i\otimes [v]_{j-1}, \qquad  [u]_i\otimes [v]_j \to  [u]_{i+1}\otimes [v]_{j},\notag  \\
   [u]_i\otimes [v]_j \to  [u]_{i-1}\otimes [v]_{j} , \qquad  [u]_i\otimes [v]_j \to  [u]_{i}\otimes [v]_{j+1}. \label{matrix coeff}
\end{align}
The first two come from terms of the form $1\otimes E_j(z)$ and $F_{i+1}(z)\otimes 1$. Thus, they are given by delta functions. Therefore, they are well-defined and non-trivial. 

The last two come from terms of the form $E_{i}(z)\otimes K_{i}^-(z)$ and $K_{j+1}^+(z)\otimes F_{j+1}(z)$. Thus, they are given by delta functions multiplied by the value of a rational function evaluated at the support of the delta function. Therefore, these matrix coefficients could be zero or undefined.

Consider $E_k(z)\otimes K_k^-(z)$. The support of the delta function has the form $z=q_1^{-\overline k+(m-n)r_1}u$, with $r_1\in \Z$. The pole of the eigenfunction of $K_k^-(z)$ is $z=q_1^{-\overline k+(m-n)r_2}V$, with $r_1\in \Z$. Clearly, it can be undefined if and only if $v/u=q_1^{(m-n)r}$, $r\in \Z$. Similarly, one obtains the same condition for $K_k^+(z)\otimes F_k(z)$, and part (i) of the lemma is proved. 

The $v/u\neq q_1^{(m-n)r}$, $r\in \Z$, eigenvalues of $K_k^\pm(z)$ have different poles in $V(u)$ and $V(v)$. Both poles cannot cancel; in addition, $V(u)$  and $V(v)$ are tame, so part (ii) follows.

Zeroes of $K_k^{\pm}(z)$ in $V(u)$ have the form $z=q_2^{\pm 1}q_1^{-\overline k+(m-n)r}$. Therefore, if $u/v\neq q_2^{\pm1}q_1^{(m-n)r}$, we have non-zero matrix coefficients \eqref{matrix coeff}. Since the tensor product is tame, it is clear that it is irreducible.

Consider part (iv).
Due to Lemma \eqref{lem: periodic}, it is enough to treat the case $v=q_2u$. Let $v=q_2u$. Then, the following matrix coefficients become zero:
\begin{align*}
   &[u]_i\otimes [v]_{i-1} \stackrel{E_i}{\to}   [u]_{i-1}\otimes [v]_{i-1}, \qquad \quad\  [u]_{i-1}\otimes [v]_{i-2} \stackrel{F_{i-1}}{\to}  [u]_{i-1}\otimes [v]_{i-1}  \quad\  (s_i=1),  \\
   &[u]_{i-1}\otimes [v]_{i-1} \stackrel{E_{i-1}}{\to}  [u]_{i-2}\otimes [v]_{i-1} , \qquad  [u]_{i-1}\otimes [v]_{i-1} \stackrel{F_i}{\to}  [u]_{i-1}\otimes [v]_{i} \qquad (s_i=-1).
\end{align*}
It follows that $[u]_i\otimes[v]_j$ for all $i>j$, and $[u]_i\otimes[v]_i$ with $i$ such that $s_{i+1}=-1$ form a unique submodule of $V(u)\otimes V(uq_2)$. The part (iv) is proved. 

The part (v) is done similarly.
For $v=q_2^{-1}u$, the matrix coefficients which become zero are:
\begin{align*}
   &[u]_i\otimes [v]_{i-1} \stackrel{E_i}{\to}   [u]_{i-1}\otimes [v]_{i-1}, \qquad \quad\  [u]_{i-1}\otimes [v]_{i-2} \stackrel{F_{i-1}}{\to}  [u]_{i-1}\otimes [v]_{i-1}  \quad (s_i=-1),  \\
   &[u]_{i-1}\otimes [v]_{i-1} \stackrel{E_{i-1}}{\to}  [u]_{i-2}\otimes [v]_{i-1} , \qquad  [u]_{i-1}\otimes [v]_{i-1} \stackrel{F_i}{\to}  [u]_{i-1}\otimes [v]_{i} \qquad\  (s_i=1).
\end{align*}

\end{proof}
\end{lem}

\begin{lem}{\label{lem WW}}
Let $u,v \in \Cx$.
\begin{enumerate}
    \item The comultiplication \eqref{coproduct} defines an action of $\Es$ on $W(u)\otimes W(v)$ if and only if $u\neq q_3^{(m-n)r} v$, $r\in \Z$.   \label{lem WW.i}
     
     \item If the $\Es$-module $W(u)\otimes W(v)$ is well-defined, then it is a tame module of level zero.
    
    \item If $v\neq q_2^{\pm 1}q_3^{(m-n)r}u$, the module $W(u)\otimes W(v)$ is irreducible.
    \item If $v=q_2q_3^{(m-n)r}u$, the module $W(u)\otimes W(v)$ has an irreducible submodule spanned by $[u]^j\otimes[v]^k$, with $j-k\leq (m+n)r$, where the equality is allowed only if $s_{j}=-1$. The quotient module is irreducible.
    
    \item If $v=q_2^{-1}q_3^{(m-n)r}u$, the module $W(u)\otimes W(v)$ has an irreducible submodule spanned by $[u]^j\otimes[v]^k$, with $j-k\leq (m+n)r$, where the equality is allowed only if $s_{j}=1$. The quotient module is irreducible.
    
    \end{enumerate}
 \end{lem}   
\begin{proof}
The lemma can be proved similarly to Lemma \ref{lem VV}. Alternatively, it immediately follows from Lemma \ref{lem VV} and Lemma \ref{lem:V-W}.

We list the matrix coefficients which become zero in part (iv).
Let $v=q_2u$. Then, the vanishing matrix coefficients are:
\begin{align*}
   &[u]^{i-1}\otimes [v]^{i} \stackrel{E_{i-1}}{\to}   [u]^i\otimes [v]^i, \qquad  [u]^i\otimes [v]^{i+1} \stackrel{F_i}{\to}  [u]^i\otimes [v]^i  \qquad (s_i=1),  \\
   &[u]^i\otimes [v]^i \stackrel{E_i}{\to}  [u]^{i+1}\otimes [v]^{i} , \qquad\ \   [u]^i\otimes [v]^i \stackrel{F_{i-1}}{\to}  [u]^i\otimes [v]^{i-1} \quad\  (s_i=-1).
\end{align*}

\end{proof}

Next, we consider tensor products of the forms $V(u)\otimes W(v)$ and $W(u)\otimes V(v)$.
\begin{lem}{\label{lem VW}}
\begin{enumerate}Let $u,v \in \Cx$.
 \item The comultiplication \eqref{coproduct} defines an action of $\Es$ on $V(u)\otimes W(v)$ if and only if $u\neq q_1^aq_3^b v$, $a,b\in \Z$, $a\equiv b \pmod {m-n}$.   \label{lem VW.i}
    
\item The comultiplication \eqref{coproduct} defines an action of $\Es$ on $W(u)\otimes V(v)$ if and only if $u\neq q_1^aq_3^b v$, $a,b\in \Z$, $a\equiv b \pmod {m-n}$.   \label{lem VW.ii}
     
 \item If the $\Es$-module $V(u)\otimes W(v)$ is well-defined, then it is an irreducible tame module of level zero.
 
 \item If the $\Es$-module $W(u)\otimes V(v)$ is well-defined, then it is an irreducible tame module of level zero.
       \end{enumerate}
\end{lem}
\begin{proof}
The lemma is similar to Lemmas \ref{lem VV},  \ref{lem WW}.  If
$u\neq q_1^aq_3^b v$, $a,b\in \Z$, $a\equiv b$ $\pmod {m-n}$, then there are no non-trivial zeroes of matrix coefficients. 
\end{proof}    
We note that the module $V(u)\otimes W(v)$ is well-defined if and only if the module $W(u)\otimes V(v)$ is well-defined.

We now describe poles and non-trivial zeroes of matrix coefficients in the spaces $V(u)\otimes W(v)$ and $W(u)\otimes V(v)$ when the parameters do not satisfy (i) (or (ii)) of Lemma \ref{lem:V-W}. That is, $u= q_2^aq_3^{(m-n)r}v$. Again, with Lemma \ref{lem: periodic} in mind, we concentrate on the case $u= q_2^av$.

For $i\in\Z$, let 
\begin{align*}
B_\s(i)=\{j\in\Z\ |\ \bar j=i \},
\end{align*}
be the preimage of $i$ under the bar map with respect to parity $\s$. Then, $B_\s(i)$ is a non-empty finite set, we have $\sqcup_i B_\s(i)=\Z$. Note also that if $j_1,j_2\in B_\s(i)$, then $j_1\not\equiv j_2 \pmod{m+n}$.

Let $v=q_2^au$ where $a\in \Z$.

The poles of matrix coefficients in $V(u)\otimes W(v)$ are
\begin{align}
    [u]_i\otimes [v]^i \stackrel{E_i}{\to}  [u]_{i-1}\otimes [v]^i,\qquad  [u]_i\otimes [v]^{i+1} \stackrel{E_i}{\to}  [u]_{i-1}\otimes [v]^{i+1} \qquad (i\in B_\s(a)), \label{E poles 1}\\
    [u]_i\otimes [v]^{i+1}\stackrel{F_i}{ \to}  [u]_i\otimes [v]^i,\qquad  [u]_{i-1}\otimes [v]^{i+1} \stackrel{F_i}{\to} [u]_{i-1}\otimes [v]^i \qquad (i\in B_\s(a)).\label{F poles 1}
\end{align}

The zeroes of matrix coefficients in $V(u)\otimes W(v)$ are
\begin{align*}
    [u]_i\otimes [v]^i \stackrel{E_i}{\to}  [u]_{i-1}\otimes [v]^i,\qquad [u]_{i-1}\otimes [v]^{i+1} \stackrel{F_i}{\to}  [u]_{i-1}\otimes [v]^i \qquad (i\in B_\s(a+s_i)), \\
    [u]_i\otimes [v]^{i+1} \stackrel{E_i}{\to}  [u]_{i-1}\otimes [v]^{i+1},\qquad [u]_i\otimes [v]^{i+1} \stackrel{F_i}{\to}  [u]_i\otimes [v]^i \qquad (i\in B_\s(a-s_{i+1})).
\end{align*}

The poles of matrix coefficients in $W(u)\otimes V(v)$ are
\begin{align}
    [u]^i\otimes [v]_i \stackrel{E_i}{\to}  [u]^{i+1}\otimes [v]_i,\qquad  [u]^i\otimes [v]_{i-1} \stackrel{E_i}{\to}  [u]^{i+1}\otimes [v]_{i-1} \qquad (i\in B_\s(-a)), \label{E poles 2}\\
    [u]^i\otimes [v]_{i-1} \stackrel{F_i}{\to}  [u]^i\otimes [v]_i,\qquad  [u]^{i+1}\otimes [v]_{i-1} \stackrel{F_i}{\to}  [u]^{i+1}\otimes [v]_{i} \qquad (i\in B_\s(-a)).\label{F poles 2}
\end{align}

The zeroes of matrix coefficients in $W(u)\otimes V(v)$ are
\begin{align*}
    [u]^i\otimes [v]_i \stackrel{E_i}{\to}  [u]^{i+1}\otimes [v]_i,\qquad [u]^{i+1}\otimes [v]_{i-1} \stackrel{F_i}{\to}  [u]^{i+1}\otimes [v]_{i} \qquad  (i\in B_\s(-a-s_{i+1})), \\
    [u]^i\otimes [v]_{i-1} \stackrel{E_i}{\to}  [u]^{i+1}\otimes [v]_{i-1}, \qquad [u]^i\otimes [v]_{i-1} \stackrel{F_i}{\to}  [u]^i\otimes [v]_i \qquad (i\in B_\s(s_{i}-a)).
\end{align*}

\medskip

In the following sections, we will work with (often semi-infinite) tensor products of copies of the same module with the evaluation parameter shifted by $q_2^{\pm 1}$. We say a tensor product such that the following pattern occurs
$$\cdots\otimes M(u)\otimes M(uq_2^{s})\otimes  M(uq_2^{2s})\otimes \cdots\ $$
has {\it positive direction} if $s=1$, and {\it negative direction} if $s=-1$.

\section{Fock modules}\label{Fock sec}
In this section, we discuss a family of $\Es$-modules that we call Fock spaces.  The Fock modules $\mathcal F_\Lambda(u)$ are parameterized by finite highest weights $\Lambda$ of the form $\Lambda=\pm\Lambda_i$ or $\Lambda=s_i(r\Lambda_{i-1}-(r+1)\Lambda_{i})$, $i\in \hat I,\ r\in\Z$, and the arbitrary evaluation parameter $u\in\C^\times$.

\subsection{The case of \texorpdfstring{$\mc F_{\La_i}$}{Fi}}\label{sec pure fock 1}

The modules  $\mc F_{\La_i}$ are constructed from the tensor products of vector and covector representations in positive direction.
We start with the case of $i=0$. 

Consider the following tensor product of vector spaces in positive direction
$$
V(u)\otimes W(u) \otimes V(uq_2)\otimes W(uq_2)\otimes \dots \otimes V(uq_2^{k-1})\otimes W(uq_2^{k-1}).
$$
It has a basis given by vectors
\begin{align}
\ket{\bs \lambda,\bs \mu}^{(k)}:=[u]_{\bs \lambda}^{\bs \mu}=[u]_{\lambda_1-1}\otimes [u]^{-\mu_1}\otimes [uq_2]_{\lambda_2-1}\otimes [uq_2]^{-\mu_2}\otimes \cdots\otimes [uq_2^{k-1}]_{\lambda_k-1}\otimes [uq_2^{k-1}]^{-\mu_k},
\end{align}
where $\lambda_i,\mu_i\in\Z$.

The coproduct \eqref{coproduct} has poles and does not define an $\Es$-module structure on the whole tensor product. However, there exists a subspace where the $\Es$-action is well-defined.

Recall the definition of $\s^{\sigma_1}_{\sigma_2}$-partitions given at the end of Section \ref{parities sec}. 

First, consider the subspace $\mc F^{(k)}$ of the tensor product given by
$$
\mc F^{(k)}=\langle \ket{\bs \lambda,\bs \mu}^{(k)}\ | \ \bs \lambda\  {\text{is a}}\ \s_+^-{\text{-partition}}, \ \bs \mu\  {\text{is a}}\ \s_-^-{\text{-partition, and}}\  k\geq\ell(\bs \lambda)\geq\ell(\bs\mu) \rangle. 
$$

Let $k>3$. Consider the subspace $\mathring{\mc F}^{(k)}\subset \mc F^{(k)}$ given by
$$
\mathring{\mc F}^{(k)}=\langle \ket{\bs \lambda, \bs \mu}^{(k)}\in \mc F^{(k)}\ |\ k-3\geq \ell(\bs \lambda)\geq\ell(\bs\mu) \rangle. 
$$

We can make the action of generators of $\Es$ on $\mathring{\mc F}^{(k)}$ well-defined as follows. 
Let 
$$
F_i^{(k)}(z)=\Delta^{(2k)}F_i(z) \qquad (i\in\hat I),
$$
and
$$
E_i^{(k)}(z)=\begin{cases}
    \psi(q_2^ku/z) \Delta^{(2k)} E_0(z)  &(i=0);\\
    \Delta^{(2k)} E_i(z)  &(i\in I),
\end{cases}  \qquad 
(K_i^{\pm})^{(k)}(z)=\begin{cases}
    \psi(q_2^ku/z) \Delta^{(2k)} K_0^\pm(z)  &(i=0);\\
    \Delta^{(2k)} K_i^\pm (z)  &(i\in I),
\end{cases}
$$
where $\Delta^{(2k)}$ is the $(2k)$-th iteration of the coproduct \eqref{coproduct}.

Then, we have the following lemma.
\begin{lem}\label{U^k lem}
Let $v\in \mathring{\mc F}^{(k)}$. Then, $E_i^{(k)}(z)v$,  $F_i^{(k)}(z)v$, and  $K_i^{(k)}(z)v$, $i\in\hat I$, are well-defined vectors in $\mc F^{(k)}$. Moreover, relations \eqref{relCK}--\eqref{Serre6} are satisfied when applied to $v$.
\end{lem}
\begin{proof}
We consider the case $E_i^{(k)}(z)$. 
As before, the matrix coefficients are delta functions multiplied by rational functions. 

First, we claim that there are no poles of the rational functions at the support of the delta functions. Indeed, such poles could happen if two of the factors are as in \eqref{E poles 1} or \eqref{E poles 2}, where $v=uq_2^a$ and $a\geq 0$. Consider the first case in \eqref{E poles 1}. Since the bottom indices are at least $-1$ and the upper indices at most $0$, we must have $i=0$ or $i=-1$. Since $\bar{0}=0$, we have $\overline{-1}=-s_0$. So, if $s_0=-1$, then we have two possibilities $a=0$, $i=0$ and $a=1$, $i=-1$. And, if $s_0=1$, then only the case $a=0$, $i=0$ is present. We treat $i=a=0$. 

Suppose $[uq_2^s]_0\otimes [uq_2^s]^0$ occurs in some vector belonging to the submodule. The presence of $[uq_2^s]^0$ forces the presence of $[uq_2^{s+1}]^0$. So, we must have
$$v=w\otimes [uq_2^s]_0\otimes [uq_2^s]^0\otimes [uq_2^{s+1}]_{j}\otimes [uq_2^{s+1}]^0\otimes\dots\otimes[uq_2^{k-1}]_{-1}\otimes [uq_2^{k-1}]^0,
$$ 
where, a priori,  $j=-1$ or $j=0$.

If $s_0=-1$, then the possible pole of $E_0(z)$, coming from the $K_0^-(z)$ eigenfunction of $[uq_2^s]^0$, is canceled by the  $K_0^-(z)$ eigenfunction of $[uq_2^{s+1}]^0$. 

If $s_0=1$, we have to look at $s_1$. If $s_1=1$, then $j=-1$. Whereas, if $s_1=-1$, then  $j=-1$ or $j=0$. In either case, the possible pole of $E_0(z)$, coming from the $K_0^-(z)$ eigenfunction of $[uq_2^s]^0$, is canceled by the  $K_0^-(z)$ eigenfunction of $[uq_2^{s+1}]_j$. 

The other cases are checked similarly.

Next, we check that the result of the application is inside $\mc F^{(k)}$. For that, we use the information about zeroes of matrix coefficients.  
The fact that $\bs\lambda$ and $\bs\mu$ stay $\s^-_+$ and $\s^-_-$ partitions follows from Lemmas \ref{lem VV} and \ref{lem WW}.
We check that the condition $\ell(\bs\lambda)\geq\ell(\bs\mu)$ is preserved. Indeed, if $v=w_1\otimes [uq_2^s]_{-1}\otimes [uq_2^s]^0\otimes w_2$, then $F_{-1}(z)$ acts on $[uq_2^s]^0$ by zero, which comes from the eigenfunction of $K^+_{-1}(z)$ on $[uq_2^s]_{-1}$. Similarly, if $v=w_1\otimes [uq_2^s]_{0}\otimes [uq_2^s]^j\otimes [uq_2^{s+1}]_{-1}\otimes [uq_2^{s+1}]^0\otimes w_2$, $j<0$, then $E_{0}(z)$ acts on $[uq_2^s]_0$ by zero, which comes from the eigenfunction of $K^-_{0}(z)$ on $[uq_2^{s+1}]^0$ if $s_0=-1$, or from the eigenfunction of $K^-_{0}(z)$ on $[uq_2^{s+1}]_{-1}$ if $s_0=1$.
\end{proof}

Moreover, we now show that the modified action is stable under the change of $k$. We have a natural embedding 
$$\phi^{(k)}:\ \mc F^{(k)}\hookrightarrow \mc F^{(k+1)}, \qquad  \ket{\bs\lambda,\bs\mu}^{(k)} \mapsto  \ket{\bs\lambda,\bs\mu}^{(k+1)}.$$

Note that $\phi^{(k)}(\mathring{\mc F}^{(k)})\subset \mathring{\mc F}^{(k+1)}$ and $\phi^{(k+2)}\circ \phi^{(k+1)}\circ \phi^{(k)}(\mc F^{(k)})=\mathring{\mc F}^{(k+3)}$.

Then, we have the following lemma.

\begin{lem}\label{stable lem}
 For any $v\in \mathring{\mc F}^{(k)}$ and $i\in \hat I$, we have
\begin{align*}
 &\phi^{(k)} E_i^{(k)}(z)v=  E_i^{(k+1)}(z)\phi^{(k)}v, \\
 &\phi^{(k)} F_i^{(k)}(z)v=  F_i^{(k+1)}(z)\phi^{(k)}v, \\
 &\phi^{(k)} (K_i^\pm)^{(k)}(z)v=  (K^\pm_i)^{(k+1)}(z)\phi^{(k)}v.
\end{align*}
\end{lem}
\begin{proof}
The lemma follows from the fact that all operators act trivially on the tail of the form $[uq_2^k]_{-1}\otimes [uq_2^k]^0$. We note that the modification of $K^\pm_0(z)$ is important here.
\end{proof}

Let $\mc F(u)$ be the inductive limit of $\mathring{\mc F}^{(k)}$. The basis of $\mc F(u)$ is given by $\ket{\bs\lambda,\bs\mu}$ with $\bs\lambda$ being a $\s^-_+$-partition and $\bs\mu$ a $\s^-_-$-partition and $\ell(\bs\lambda)\geq \ell(\bs\mu)$. 

We define the action of $\Es$ on $\mc F(u)$ in the obvious way. Namely, for any $v\in\mc F(u)$ there exists $k$ such that $v\in \mathring{\mc F}^{(k)}$. Then we let $E_i(z)$, $F_i(z)$, $K_i^\pm(z)$ act on $v\in\mc F(u)$ as $E_i^{(k)}(z)$, $F_i^{(k)}(z)$, $(K_i^\pm)^{(k)}(z)$.

\begin{thm}\label{pure thm 1}
The space $\mc F(u)$ is an irreducible tame highest weight $\Es$-module with highest weight $(\psi(u/z),1,\dots,1)$.
\end{thm}
\begin{proof}
Lemmas \ref{U^k lem}, \ref{stable lem} show that the action of $\Es$ is well-defined.

We prove the property of being tame at the end of Section \ref{sec: pictures}.

Since the module is tame, any submodule of $\mc F(u)$ has a basis consisting of eigenvectors of $K_i^\pm(z)$, that is, of vectors of the form $\ket{\bs \lambda,\bs\mu}$. Moreover, if $\ket{\bs \lambda,\bs\mu}$ is in a submodule, then
every vector of the form $\ket{{\bs \lambda}',{\bs\mu}'}$ in the linear combinations $F_i(z)\ket{\bs \lambda,\bs\mu}$ or $E_i(z)\ket{\bs \lambda,\bs\mu}$ is in the same submodule. Note that, for each vector of the form $\ket{\bs \lambda,\bs\mu}$, $(\bs\la,\bs \mu)\neq (\bs\emptyset,\bs\emptyset)$, there exists at least one $i\in \hat{I}$ such that $E_i(z)\ket{\bs \lambda,\bs\mu}\neq 0$. It follows that $\mc F(u)$ has no non-trivial submodules. 
\end{proof}

The finite weight of the highest weight vector is $\Lambda_0$, therefore we reflect it in the notation and denote $\mc F(u)$ by $\mc F_{\Lambda_0}(u)$.

\medskip

Alternatively, one can construct $\mc F_{\Lambda_0}(u)$ using the product (also in the positive direction)
$$
W(u)\otimes V(u) \otimes W(uq_2)\otimes V(uq_2)\otimes W(uq_2^2)\otimes V(uq_2^2)\otimes \cdots \ .
$$
The highest weight vector is given by
$$
[u]^1\otimes [u]_{0}\otimes [uq_2]^{1}\otimes [uq_2]_{0}\otimes  [uq_2^2]^{1}\otimes [uq_2^2]_{0}\otimes \cdots \ .
$$

\medskip

With obvious changes, we similarly construct $\mc F_{\Lambda_i}(u)$,  for each $i\in \hat I$. In particular, such a module is constructed through tensor products in positive direction. The basis of $\mc F_{\Lambda_i}(u)$
is given by 
$$
\ket{\bs\lambda,\bs\mu}_i^+:=
[uq_1^{\bar{i}}]_{\lambda_1+i-1}\otimes [uq_1^{\bar{i}}q_2^{\bar{i}}]^{i-\mu_1}\otimes [uq_1^{\bar{i}}q_2]_{\lambda_2+i-1}\otimes [uq_1^{\bar{i}}q_2^{\bar{i}+1}]^{i-\mu_2} \otimes \cdots, $$
where $\la_j,\mu_j\in\Z_{\geq 0}$ are such that $\lambda_1\geq\la_2\geq \dots$ with equality $\lambda_j=\lambda_{j+1}$ allowed if $s_{i+\la_j}=-1$ or if $\lambda_j=\mu_j=0$, and
$\mu_1\geq \mu_2\geq \dots$ with equality $\mu_j=\mu_{j+1}$ allowed if $s_{i-\mu_j}=-1$ or if $\mu_j=\la_{j+1}=0$. In addition, only finitely many $\la_j,\mu_j$ are non-zero and $\mu_j=0$ whenever $\la_j=0$.

The highest weight vector of $\mc F_{\Lambda_i}(u)$ is then $\ket{\bs\emptyset,\bs\emptyset}_i^+$ and the highest weight of $\mc F_{\Lambda_i}(u)$ is given by
$(1,\dots, 1,\psi(u/z), 1,\dots, 1)$, where $\psi(u/z)$ is in the $i$-th position.

\subsection{Explicit formulas}
The action of generators of $\Es$ on $\mc F_{\Lambda_i}(u)$ is easily obtained from the construction. We give here the explicit formulas for $i=0$.

Given a partition $\bs \lambda=(\lambda_1\geq \lambda_2\geq\dots\geq  \lambda_k)$ with at most $k$ parts and $j\in\{1,\dots,k\}$, define the new sequence $\bs \lambda -\bs 1_j=(\lambda_1, \lambda_2,\dots, \lambda_j-1, \dots,\lambda_k)$.

We consider the module $\mc F_{\Lambda_0}$ defined in Theorem \ref{pure thm 1}.  

We call a pair $(\bs\lambda,\bs\mu)$ {\it admissible} if $\bs\lambda$ is a $\s_+^-$-partition, $\bs \mu$ is a $\s_-^-$-partition, and $\ell(\bs\lambda)\geq \ell(\bs\mu)$. The basis of $\mc F_{\Lambda_0}$ is given by $\ket{\bs\lambda,\bs\mu}_0^+$ with admissible $(\bs\lambda,\bs\mu)$.

We list all non-trivial matrix coefficients. We skip the zero and pluses from the notation of vectors writing simply $\ket{\bs\lambda,\bs\mu}$ for $\ket{\bs\lambda,\bs\mu}_0^+$. We also use the bra-ket notation.

Let $(\bs\lambda,\bs\mu)$ and $(\bs\lambda-\bs 1_j,\bs\mu)$ be admissible pairs of partitions. If $i\equiv \lambda_j-1$, then
\begin{align}
    \bra{\bs\lambda-\mathbf{1}_j,\bs\mu}E_i(z)\ket{\bs\lambda,\bs\mu}=
        (-1)^{|i|(\sum_{k=1}^{j-1}(1-s_{\lambda_k})/2+(1-s_{-\mu_k})/2)}
        \delta\left(q_1^{-(\overline{\lambda_j-1})}q_2^{j-1}u/z\right) \times\\
        \times \prod_{\substack{k=j+1\\i\equiv \lambda_k-1}}^\infty \psi\left(\left(q_1^{-(\overline{\lambda_k-1})}q_2^{k-1}u/z\right)^{-s_{\lambda_k}}\right)
        \prod_{\substack{k=j+1\\i\equiv \lambda_k}}^\infty \psi\left(\left(q_1^{-(\overline{\lambda_k})}q_2^{k-1}u/z \right)^{s_{\lambda_k}}\right) \times \notag\\
        \times \prod_{\substack{k=j\\i\equiv -\mu_k}}^\infty \psi\left(\left(q_3^{(\overline{-\mu_k})}q_2^{k-1}u/z \right)^{-s_{-\mu_k}}\right) 
    \prod_{\substack{k=j\\i\equiv -\mu_k-1}}^\infty \psi\left(\left(q_3^{(\overline{-\mu_k-1})}q_2^{k-1}u/z \right)^{s_{-\mu_k}}\right). \notag
\end{align}

Let $(\bs\lambda,\bs\mu)$ and $(\bs\lambda,\bs\mu-\bs 1_j)$ be admissible pairs of partitions. If $i\equiv -\mu_j$, then
\begin{align}
    \bra{\bs\lambda,\bs\mu-\mathbf{1}_j}E_i(z)\ket{\bs\lambda,\bs\mu}=
        (-1)^{|i|((1-s_{\lambda_j})/2+\sum_{k=1}^{j-1}(1-s_{\lambda_k})/2+(1-s_{-\mu_k})/2)}
        \delta\left(q_3^{(\overline{-\mu_j})}q_2^{j-1}u/z\right) \times\\
        \times \prod_{\substack{k=j+1\\i\equiv \lambda_k-1}}^\infty \psi\left(\left(q_1^{-(\overline{\lambda_k-1})}q_2^{k-1}u/z\right)^{-s_{\lambda_k}}\right)
        \prod_{\substack{k=j+1\\i\equiv \lambda_k}}^\infty \psi\left(\left(q_1^{-(\overline{\lambda_k})}q_2^{k-1}u/z \right)^{s_{\lambda_k}}\right) \times \notag\\
        \times \prod_{\substack{k=j+1\\i\equiv -\mu_k}}^\infty \psi\left(\left(q_3^{(\overline{-\mu_k})}q_2^{k-1}u/z \right)^{-s_{-\mu_k}}\right) 
    \prod_{\substack{k=j+1\\i\equiv -\mu_k-1}}^\infty \psi\left(\left(q_3^{(\overline{-\mu_k-1})}q_2^{k-1}u/z \right)^{s_{-\mu_k}}\right). \notag
\end{align}

Let $(\bs\lambda,\bs\mu)$ and $(\bs\lambda+\bs 1_j,\bs\mu)$ be admissible pairs of partitions.
If $i\equiv \lambda_j$, then
\begin{align}
    \bra{\bs\lambda+\mathbf{1}_j,\bs\mu}F_i(z)\ket{\bs\lambda,\bs\mu}=
        s_{\lambda_j}(-1)^{|i|(\sum_{k=1}^{j-1}(1-s_{\lambda_k})/2+(1-s_{-\mu_k})/2)}
        \delta\left(q_1^{-(\overline{\lambda_j})}q_2^{j-1}u/z\right) \times\\
        \times \prod_{\substack{k=1\\i\equiv \lambda_k-1}}^{j-1} \psi\left(\left(q_1^{-(\overline{\lambda_k-1})}q_2^{k-1}u/z\right)^{-s_{\lambda_k}}\right)
        \prod_{\substack{k=1\\i\equiv \lambda_k}}^{j-1} \psi\left(\left(q_1^{-(\overline{\lambda_k})}q_2^{k-1}u/z \right)^{s_{\lambda_k}}\right) \times \notag\\
        \times \prod_{\substack{k=1\\i\equiv -\mu_k}}^{j-1} \psi\left(\left(q_3^{(\overline{-\mu_k})}q_2^{k-1}u/z \right)^{-s_{-\mu_k}}\right) 
    \prod_{\substack{k=1\\i\equiv -\mu_k-1}}^{j-1} \psi\left(\left(q_3^{(\overline{-\mu_k-1})}q_2^{k-1}u/z \right)^{s_{-\mu_k}}\right). \notag
\end{align}

Let $(\bs\lambda,\bs\mu)$ and $(\bs\lambda,\bs\mu+\bs 1_j)$ be admissible pairs of partitions.
If $i\equiv -\mu_j-1$, then
\begin{align}
    \bra{\bs\lambda,\bs\mu+\mathbf{1}_j}F_i(z)\ket{\bs\lambda,\bs\mu}=
        s_{-\mu_j+1}(-1)^{|i|((1-s_{\lambda_j})/2+\sum_{k=1}^{j-1}(1-s_{\lambda_k})/2+(1-s_{-\mu_k})/2)}
        \delta\left(q_3^{(\overline{-\mu_j})}q_2^{j-1}u/z\right) \times\\
        \times \prod_{\substack{k=1\\i\equiv \lambda_k-1}}^{j} \psi\left(\left(q_1^{-(\overline{\lambda_k-1})}q_2^{k-1}u/z\right)^{-s_{\lambda_k}}\right)
        \prod_{\substack{k=1\\i\equiv \lambda_k}}^{j} \psi\left(\left(q_1^{-(\overline{\lambda_k})}q_2^{k-1}u/z \right)^{s_{\lambda_k}}\right) \times \notag\\
        \times \prod_{\substack{k=1\\i\equiv -\mu_k}}^{j-1} \psi\left(\left(q_3^{(\overline{-\mu_k})}q_2^{k-1}u/z \right)^{-s_{-\mu_k}}\right) 
    \prod_{\substack{k=1\\i\equiv -\mu_k-1}}^{j-1} \psi\left(\left(q_3^{(\overline{-\mu_k-1})}q_2^{k-1}u/z \right)^{s_{-\mu_k}}\right). \notag
\end{align}

Finally, let $(\bs\lambda,\bs\mu)$  be an admissible pair of partitions. Then
\begin{align}
    \bra{\bs\lambda,\bs\mu}K^\pm_i(z)\ket{\bs\lambda,\bs\mu}=
        \prod_{\substack{k=1\\i\equiv \lambda_k-1}}^\infty \psi\left(\left(q_1^{-(\overline{\lambda_k-1})}q_2^{k-1}u/z\right)^{-s_{\lambda_k}}\right)
        \prod_{\substack{k=1\\i\equiv \lambda_k}}^\infty \psi\left(\left(q_1^{-(\overline{\lambda_k})}q_2^{k-1}u/z \right)^{s_{\lambda_k}}\right) \times \notag\\
        \times \prod_{\substack{k=1\\i\equiv -\mu_k}}^\infty \psi\left(\left(q_3^{(\overline{-\mu_k})}q_2^{k-1}u/z \right)^{-s_{-\mu_k}}\right) 
    \prod_{\substack{k=1\\i\equiv -\mu_k-1}}^\infty \psi\left(\left(q_3^{(\overline{-\mu_k-1})}q_2^{k-1}u/z \right)^{s_{-\mu_k}}\right). \notag
\end{align}

We will not write explicit formulas for other cases. Instead, we give a combinatorial description of the action in Section \ref{sec: pictures}.

\subsection{The case of \texorpdfstring{$\mc F_{-\La_i}$}{F-i}}\label{sec pure fock 2}

The construction of $\mc F_{-\La_i}$ is parallel to the construction of $\mc F_{\La_i}$ but it uses tensor products in negative direction. Essentially, one changes $q$ to $q^{-1}$ in all formulas. We give a few details.

Consider the tensor product of vector spaces in negative direction
$$
V(u)\otimes W(u) \otimes V(uq_2^{-1})\otimes W(uq_2^{-1})\otimes \dots \otimes V(uq_2^{-k+1})\otimes W(uq_2^{-k+1}).
$$

We consider a subspace $\mc F^{(k)}$ of the tensor product given by
$$
\mc F^{(k)}=\langle \ket{\bs \lambda,\bs \mu}^{(k)} \ |\ \bs \lambda\  {\text{is a}}\ \s^+_+{\text{-partition}}, \ \bs \mu\  {\text{is a}}\ \s^+_-{\text{-partition, and}}\  k\geq\ell(\bs \lambda)\geq\ell(\bs\mu) \rangle. 
$$

In this case, we modify the action as follows. Let
$$
F_i^{(k)}(z)=\Delta^{(2k)}F_i(z) \qquad (i\in\hat I),
$$
and
$$
E_i^{(k)}(z)=\begin{cases}
    \psi_{-1}(q_2^{-k} u/z) \Delta^{(2k)} E_0(z)  &(i=0);\\
    \Delta^{(2k)} E_i(z)  &(i\in I),
\end{cases}  \qquad 
(K_i^{\pm})^{(k)}(z)=\begin{cases}
    \psi_{-1}(q_2^{-k}u/z) \Delta^{(2k)} K_0^\pm(z)  &(i=0);\\
    \Delta^{(2k)} K_i^\pm (z)  &(i\in I).
\end{cases}
$$
With this modification, the action stabilizes as $k$ increases.

Again we consider the inductive limit of ${\mc F}^{(k)}$ and denote it by $\mc F(u)$. The basis of $\mc F(u)$ is given by $\ket{\bs\lambda,\bs\mu}$ with $\bs\lambda$ being a $\s^+_+$-partition, $\bs\mu$ a $\s_-^+$-partition, and $\ell(\bs\lambda)\geq \ell(\bs\mu)$. We define the action of $\Es$ on $\mc F(u)$ in the obvious way. 

\begin{thm}\label{pure thm 2}
The space $\mc F(u)$ is an irreducible tame highest weight $\Es$-module with highest weight $(\psi_{-1}(u/z),1,\dots,1)$.
\end{thm}
\begin{proof}
The proof of the theorem is parallel to that of Theorem \ref{pure thm 1}.
\end{proof}

The finite weight of the highest weight vector is $-\Lambda_0$. Therefore, we denote $\mc F(u)$ constructed in Theorem \ref{pure thm 2} by $\mc F_{-\Lambda_0}(u)$.

\medskip 

With obvious changes, we construct $\mc F_{-\Lambda_i}(u)$  for each $i\in \hat I$. The basis of $\mc F_{-\Lambda_i}(u)$
is given by 
$$\ket{\bs\lambda,\bs\mu}_i^{-}:=[q_1^{\bar{i}}u]_{\lambda_1+i-1}\otimes [q_1^{\bar{i}}q_2^{\bar{i}}u]^{i-\mu_1}\otimes [q_1^{\bar{i}}q_2^{-1}u]_{\lambda_2+i-1}\otimes [q_1^{\bar{i}}q_2^{\bar{i}-1}u]^{i-\mu_2} \otimes \cdots , $$
where $\la_j,\mu_j\in\Z_{\geq 0}$ are such that $\lambda_1\geq\la_2\geq \dots$ with equality $\lambda_j=\lambda_{j+1}$ allowed if $s_{i+\la_j}=1$ or if $\lambda_j=\mu_j=0$, and
$\mu_1\geq \mu_2\geq \dots$ with equality $\mu_j=\mu_{j+1}$ allowed if $s_{i-\mu_j}=1$ or if $\mu_j=\la_{j+1}=0$. In addition, only finitely many $\la_j,\mu_j$ are non-zero and $\mu_j=0$ whenever $\la_j=0$.

The highest weight vector of $\mc F_{-\Lambda_i}(u)$ is  $\ket{\bs\emptyset,\bs\emptyset}_i^{-}$ and the highest weight  of $\mc F_{-\Lambda_i}(u)$ is given by
$(1,\dots, 1,\psi_{-1}(u/z), 1,\dots, 1)$, where $\psi_{-1}(u/z)$ is in the $i$-th position.

\subsection{The case of \texorpdfstring{$\mc F_{s_{i+1}(r\La_{i+1}-(r+1)\La_{i})}$}{F,r,i+1,i}}\label{sec non pure fock 1}

We consider tensor products of Fock spaces constructed in the previous sections with appropriate tensor products of vector representations to obtain a larger set of Fock modules.

We start with the case of $i=0$. Set $s=s_1$. Let $r\in\Z_{>0}$.

The modules $\mc F_{\La_{0}}(u)$ and $\mc F_{-\La_{0}}(u)$ are constructed in Theorems \ref{pure thm 1}, \ref{pure thm 2} using tensor product in positive and negative directions, respectively. We multiply it by an extra tensor product of vector representations in the same direction. Namely,  consider:
$$
U=V(u)\otimes V(uq_2^{-s})\otimes V(uq_2^{-2s})\otimes \cdots\otimes V(uq_2^{s(1-r)})\otimes \mc F_{-s\La_{0}}(uq_2^{-sr}).
$$
The basis of $U$ is given by
$$
\ket{(\bs\nu,\bs\lambda),\bs\mu}_{0,r}^{-s}=[u]_{\nu_1}\otimes[uq_2^{-s}]_{\nu_{2}}\otimes \dots\otimes [uq_2^{s(1-r)}]_{\nu_{r}}\otimes \ket{\bs\lambda,\bs\mu}_{0}^{-s}.
$$
Here, $\ket{\bs\lambda,\bs\mu}_{0}^{-s}$ denotes 
$\ket{\bs\lambda,\bs\mu}_{0}^{+}$ if $s=-1$, and  $\ket{\bs\lambda,\bs\mu}_{0}^{-}$ if $s=1$.

We consider a subspace $\mc F(u)\subset U$,
$$\mc F(u)=\langle \ket{(\bs\nu,\bs\lambda),\bs\mu}_{0,r}^{-s}\ |\ (\bs\nu+1,\bs\lambda)\ {\text {is a }} \s^-_+ {\text  {-partition if }} s_1=-1 {\text { and }} \s^+_+ {\text  {-partition if }} s_1=1 \rangle.
$$
Here, $\bs\nu+1=(\nu_1+1,\dots,\nu_r+1)$.
\begin{thm}\label{non pure 1 0,r}
The space $\mc F(u)$ is an irreducible tame highest weight $\Es$-module with highest weight vector $\ket{(\bs\emptyset,\bs\emptyset),\bs\emptyset}_{0,r}^{-s}$ and highest weight   $(\psi_{-s(r+1)}(q_2^{-sr}u/z),\psi_{sr}(q_1^{-s}u/z),1,\dots,1)$.
\end{thm}
\begin{proof}
The proof is done by a direct check.
\end{proof}
The finite weight of the highest weight vector is ${s_{1}(r\La_{1}-(r+1)\La_{0})}$. Therefore, we denote $\mc F(u)$ constructed in Theorem \ref{non pure 1 0,r} by $\mc F_{s_{1}(r\La_{1}-(r+1)\La_{0})}(u)$.

\medskip 

With obvious changes, we construct the module $\mc F_{s_{i+1}(r\La_{i+1}-(r+1)\La_{i})}(u)$,  for each $i\in \hat I$. 

The basis of $\mc F_{s_{i+1}(r\La_{i+1}-(r+1)\La_{i})}(u)$
is given by 
$$\ket{(\bs\nu,\bs\lambda),\bs\mu}_{i,r}^{-s_{i+1}}:=[uq_1^{\overline{i}}]_{\nu_1+i}\otimes [uq_1^{\overline{i}}q_2^{-s_{i+1}}]_{\nu_2+i}\otimes \cdots \otimes [uq_1^{\overline{i}}q_2^{s_{i+1}(1-r)}]_{\nu_r+i}\otimes \ket{\bs\lambda,\bs\mu}_{i}^{-s_{i+1}}, $$
where $\ket{\bs\lambda,\bs\mu}_{i}^{-s_{i+1}}\in \mc F_{-s\La_i}(uq_2^{-s_{i+1}r})$, $\nu_j\in\Z_{\geq 0}$ are such that $\nu_1\geq\nu_2\geq \dots\geq \nu_r$ with equality $\nu_j=\nu_{j+1}$ allowed if $s_{i+\nu_j+1}=s_{i+1}$, and such that $\nu_r+1\geq \la_1$ with equality allowed if $s_{i+\la_1}=s_{i+1}$.

The highest weight of the module $\mc F_{s_{i+1}(r\La_{i+1}-(r+1)\La_{i})}(u)$ is given by the $\hat I$-tuple $$(1,\dots,1,\psi_{-s_{i+1}(r+1)}(q_2^{-s_{i+1}r}u/z),\psi_{s_{i+1}r}(q_1^{-s_{i+1}}u/z),1,\dots,1),$$ where the non-trivial functions are at the $i$-th and $(i+1)$-st positions.

\subsection{The case of \texorpdfstring{$\mc F_{s_{i}(r\La_{i-1}-(r+1)\La_{i})}$}{F,r,i,i+1}}\label{sec non pure fock 2}

The construction of the previous section can also be repeated using tensor products of pure Fock modules and covector representations yielding more modules.

We again start with the case of $i=0$. Set $s=s_0$. Let $r\in\Z_{>0}$.

The modules $\mc F_{\La_{0}}(u)$ and $\mc F_{-\La_{0}}(u)$ are constructed in Theorems \ref{pure thm 1}, \ref{pure thm 2} using tensor product in positive and negative directions, respectively. We multiply it by an extra tensor product of covector representations in the same direction. Namely,  consider:
$$
U=W(u)\otimes W(uq_2^{-s})\otimes\cdots\otimes W(uq_2^{s(1-r)})\otimes \mc F_{-s\La_{0}}(uq_2^{-sr}).
$$
The basis of $U$ is given by
$$
\ket{\bs\lambda,(\bs\nu,\bs\mu)}_{0,r}^{-s}=[u]^{-\nu_1}\otimes[uq_2^{-s}]^{-\nu_{2}}\otimes\dots\otimes [uq_2^{s(1-r)}]^{-\nu_{r}}\otimes \ket{\bs\lambda,\bs\mu}_{0}^{-s}.
$$
We consider a subspace $\mc F(u)\subset U$,
$$\mc F(u)=\langle \ket{\bs\lambda,(\bs\nu,\bs\mu)}_{0,r}^{-s}\ |\ (\bs\nu,\bs\mu)\ {\text {is a }} \s^-_- {\text  {-partition if }} s_0=-1 {\text { and }} \s^+_- {\text  {-partition if }} s_0=1 \rangle.
$$

\begin{thm}\label{non pure 2 0,r}
The space $\mc F(u)$ is an irreducible tame highest weight $\Es$-module with highest weight vector $\ket{\bs\emptyset,(\bs\emptyset,\bs\emptyset)}_{0,r}^{-s}$ and highest weight   $(\psi_{-s(r+1)}(q_2^{-sr}u/z),1,\dots,1,\psi_{sr}(q_3^{-s}u/z))$.
\end{thm}
\begin{proof}
The proof is done by a direct check.
\end{proof}
The finite weight of the highest weight vector is ${s_{0}(r\La_{-1}-(r+1)\La_{0})}$. Therefore, we denote $\mc F(u)$ constructed in Theorem \ref{non pure 2 0,r} by $\mc F_{s_{0}(r\La_{-1}-(r+1)\La_{0})}(u)$.

\medskip 

With obvious changes, we construct the module $\mc F_{s_{i}(r\La_{i-1}-(r+1)\La_{i})}(u)$,  for each $i\in \hat I$. 

The basis of $\mc F_{s_{i}(r\La_{i-1}-(r+1)\La_{i})}(u)$
is given by 
$$\ket{\bs\lambda,(\bs\nu,\bs\mu)}_{i,r}^{-s_i}:=[uq_3^{-\overline{i}}]^{i-\nu_1}\otimes [uq_3^{-\overline{i}}q_2^{-s_i}]^{i-\nu_2}\otimes \cdots \otimes [uq_3^{-\overline{i}}q_2^{s_i(1-r)}]^{i-\nu_r}\otimes \ket{\bs\lambda,\bs\mu}_{i}^{-s_i}, $$
where $\ket{\bs\lambda,\bs\mu}_{i}^{-s_i}\in \mc F_{-s_i\La_i}(uq_2^{-s_ir})$, $\nu_j\in\Z_{\geq 0}$ are such that $\nu_1\geq\nu_2\geq \dots\geq \nu_r$ with equality $\nu_j=\nu_{j+1}$ allowed if $s_{i-\nu_j}=s_i$, and such that $\nu_r\geq \mu_1$ with equality allowed if $s_{i-\mu_1}=s_i$.

The highest weight of the module  $\mc F_{s_{i}(r\La_{i-1}-(r+1)\La_{i})}(u)$ is given by the $\hat I$-tuple $$(1,\dots,1,\psi_{s_ir}(q_3^{-s_i}u/z),\psi_{-s_i(r+1)}(q_2^{-s_ir}u/z),1,\dots,1),$$ where the non-trivial functions are at the $(i-1)$-st and $i$-th positions.

\section{MacMahon Modules}\label{MacMahon sec}

\subsection{Tensor products of Fock modules}
First, we study tensor products of Fock modules. A generic tensor product of Fock modules is irreducible. We are interested in non-generic cases.

The Fock module $\mc F_{\Lambda_k}(u)$ is constructed using tensor products in positive direction. We consider the space
$\mc F_{\Lambda_k}(u)\otimes \mc F_{\Lambda_k}(uq_2^{-a})$, where $a\in\Z_{\geq 1}$. Thus, as in the previous cases, we cannot use the  coproduct  \eqref{coproduct} to define an $\Es$-module structure on the whole tensor product, however, we can do it on a subspace.

\begin{lem}\label{lem: tensor fock +} Assume $s_k=-1$.
The subspace of  $\mc F_{\Lambda_k}(u)\otimes \mc F_{\Lambda_k}(uq_2^{-a})$  spanned by 
$$
\ket*{\bs \la^{(1)},\bs \mu^{(1)}}_k^+\otimes \ket*{\bs \la^{(2)},\bs \mu^{(2)}}_k^+,
$$
with $\la^{(1)}_{i}\geq \la^{(2)}_{i+a-1}$, where the equality is allowed only if $s_{\la^{(1)}_i+k}=1$ or $\la_i^{(1)}=0$, and $\mu^{(1)}_{i}\geq \mu^{(2)}_{i+a-1}$, where the equality is allowed only if $s_{-\mu^{(1)}_i+k}=1$ or $\mu^{(1)}_i=0$, and where $\ell(\bs \la^{(2)})\leq \ell (\bs \mu^1)+a-1$, is 
an irreducible tame highest weight $\Es$-module with highest weight $(1,\dots, 1,\psi(u/z)\psi(q_2^{-a}u/z),1,\dots,1)$, where the non-trivial function is in the $k$-th position. 
\end{lem}
\begin{proof}
Without loss of generality, let $k=0$.

Recall that the Fock space $\mc F_{\Lambda_0}(u)$ is realized as a subspace of tensor products of $V(uq_2^{j-1})$ and $W(uq_2^{j-1})$ with $j\in\Z_{> 0}$.  

We study the poles and the non-trivial zeros of matrix coefficients.
There are four cases: $E_i(z)$ acting on $V(uq_2^{j-1})$ or $W(uq_2^{j-1})$ coming from the first Fock space, and $F_i(z)$ acting on $V(uq_2^{j-1})$ or $W(uq_2^{j-1})$ coming from the second Fock space.

Consider  $F_i(z)$ acting on $V(uq_2^{j-1})$ coming from the second space. The possible poles come from
$$
[uq_2^{j-1}]_{i-1}\otimes [uq_2^{j-1}]_{i-1}, \qquad [uq_2^{j-1}]_i\otimes [uq_2^{j-1}]_{i-1},\qquad [uq_2^{j_1-1}]^i\otimes [uq_2^{j-1}]_{i-1},\qquad [uq_2^{j_1-1}]^{i+1}\otimes [uq_2^{j-1}]_{i-1},
$$
where the second factor comes from the second Fock space.

Note that the fourth case is not present in the subspace, since $\mu^{(1)}_{j_1}= -i-1\geq 0$ and $\la^{(2)}_{j+a}-1=i-1\geq -1$.
Also, the first and third cases force $i=0$, $j_1=j$. Indeed, for the third case we need $i\leq 0$ and $i-1\geq -1$, and for the first case both inequalities $\la^{(1)}_{j}\geq \la^{(2)}_{j+a-1} \geq \la^{(2)}_{j+a}$ are equalities, which is possible only if $\lambda_{j}^{(1)}-1=i=-1$.

Consider the third case, $[uq_2^{j-1}]^0\otimes [uq_2^{j-1}]_{-1}$. Then, $[uq_2^{j}]^0$ and $[uq_2^{j-2}]_{-1}$ must be present. A second pole could be coming from $[uq_2^{j-1}]_{-1}$ (the first case), but in this case, $[uq_2^{j}]_{-1}$ must also be present. In both cases, the total matrix coefficient is zero.

Consider the second case, $[uq_2^{j-1}]_i\otimes [uq_2^{j-1}]_{i-1}$. Then, $\la^{(2)}_{j+a-1}$ is either $i$ or $i+1$. If $\la^{(2)}_{j+a-1}=i$, we must have $s_i=s_{i+1}=-1$, and $[uq_2^{j}]_{i}$ and $[uq_2^{j-2}]_{i-1}$ must be present.  If $\la^{(2)}_{j+a-1}=i+1$, we must have $s_i=s_{i+1}=1$, and $[uq_2^{j}]_{i-1}$ and $[uq_2^{j-2}]_{i}$ must be present. In both cases, the total matrix coefficient is zero.

\medskip 

Let us now look at $E_i(z)$ acting on $W(uq_2^{j-1})$ coming from the first space. The possible poles come from
$$
[uq_2^{j-1}]^{i}\otimes [uq_2^{j-1}]^{i}, \qquad [uq_2^{j-1}]^i\otimes [uq_2^{j-1}]^{i+1},\qquad [uq_2^{j-1}]^i\otimes [uq_2^{j_1-1}]_{i},\qquad [uq_2^{j-1}]^{i}\otimes [uq_2^{j_1-1}]_{i-1},
$$
where the second factor comes from the second Fock space.

The first and fourth cases happen only if $i=0$ and $j_1=j$, and the third case only if $i=-1$ and $j_1=j-1$. The third case could also happen if $i=0$ and $j_1=j$, but this means that $\mu^{(1)}_j=0$ and $\la^{(2)}_{j+a}=1$, but $\mu^{(1)}_j=0$ implies $\la^{(2)}_{j+a-1}=0$.

Consider the first case, $[uq_2^{j-1}]^{0}\otimes [uq_2^{j-1}]^{0}$.  A second pole could be coming from $[uq_2^{j-1}]_{-1}$ (fourth case). But in both cases, we must have two factors $[uq_2^{j}]^{0}$, and $\mu^{(1)}_j=0$ implies $\la^{(2)}_{j+a-1}=0$. Thus, $[uq_2^{j-2}]_{-1}$ is also present, which makes the matrix coefficient zero.

Consider the second case, $[uq_2^{j-1}]^i\otimes [uq_2^{j-1}]^{i+1}$. Then, $\mu^{(2)}_{j+a-1}=-i$ or  $\mu^{(2)}_{j+a-1}=-i-1$. If $\mu^{(2)}_{j+a-1}=-i$, we must have $s_i=s_{i+1}=1$, and $[uq_2^{j}]^{i+1}$ and $[uq_2^{j-2}]^{i}$ must be present. If $\mu^{(2)}_{j+a-1}=-i-1$, we must have $s_i=s_{i+1}=-1$, and $[uq_2^{j}]^{i}$ and $[uq_2^{j-2}]^{i+1}$ must be present. In both cases, the total matrix coefficient is zero.

Consider the third case, $[uq_2^{j-1}]^{-1}\otimes [uq_2^{j-2}]_{-1}$. Then, $[uq_2^{j-2}]^{0}$ must be present, which makes the matrix coefficient well-defined.

The other checks are done similarly.

\medskip

Finally, we note that, in the absence of the poles, the existence of the zero matrix coefficients keeping the subspace preserved is straightforward.

\end{proof}

If $s_k=1$, one can also prove an analog of Lemmas \ref{lem: tensor fock +}. The only change is that in this case there is no conditions $\ell(\bs \la^{(2)})\leq \ell (\bs \mu^1)+a-1$. We do not need this case for the construction of MacMahon modules.

\medskip

The Fock module $\mc F_{-\Lambda_k}(u)$ is constructed using tensor products in negative direction. We consider the space
$\mc F_{-\Lambda_k}(u)\otimes \mc F_{-\Lambda_k}(uq_2^{a})$, where $a\in\Z_{\geq 1}$.

\begin{lem}\label {lem: tensor fock -} Assume $s_k=1$. 
The subspace of  $\mc F_{-\Lambda_k}(u)\otimes \mc F_{-\Lambda_k}(uq_2^{a})$  spanned by 
$$
\ket*{\bs \la^{(1)},\bs \mu^{(1)}}_k^-\otimes \ket*{\bs \la^{(2)},\bs \mu^{(2)}}_k^-,
$$
with $\la^{(1)}_{i}\geq \la^{(2)}_{i+a-1}$, where the equality is allowed only if $s_{\la^{(1)}_i+k}=-1$ or $\la_i^{(1)}=0$, and $\mu^{(1)}_{i}\geq \mu^{(2)}_{i+a-1}$, where the equality is allowed only if $s_{-\mu^{(1)}_i+k}=-1$ or $\mu^{(1)}_i=0$, and where $\ell(\bs \la^{(2)})\leq \ell (\bs \mu^{(1)})+a-1$, is 
an irreducible tame highest weight $\Es$-module with highest weight $(1,\dots, 1,\psi(z/u)\psi(q_2^{-a}z/u),1,\dots,1)$, where the non-trivial function is in the position $k$. 
\end{lem}
\begin{proof}
Lemma \ref{lem: tensor fock -} is proved similarly to Lemma \ref{lem: tensor fock +}.
\end{proof}

If $s_k=-1$, one can also prove an analog of Lemmas \ref{lem: tensor fock -}. The only change is that in this case there is no conditions $\ell(\bs \la^{(2)})\leq \ell (\bs \mu^{(1)})+a-1$. 

\medskip

One can prove similar lemmas for other tensor products of Fock spaces. We do not give these lemmas here as it does not add anything extra to our main goal of constructing the MacMahon modules.

\subsection{Vacuum MacMahon modules}
Let $s_0=-1$. We construct the vacuum MacMahon module by a procedure similar to the one used for the construction of the Fock modules, cf. Section \ref{sec pure fock 1}. See also \cite{FJMM}.

We consider
$$
\mc F_{\Lambda_0}(u)\otimes \mc F_{\Lambda_0}(uq_2^{-1})\otimes \mc F_{\Lambda_0}(uq_2^{-2})\otimes \dots \otimes \mc F_{\Lambda_0}(uq_2^{-k+1}).
$$
Note that this tensor product has negative direction, while the Fock module $\mc F_{\Lambda_0}(u)$ is constructed from tensor product with positive direction.

This space has a basis given by
$$
\ket*{\vec{\bs \lambda},\vec{\bs \mu}}^{(k)}=\ket*{\bs \lambda^{(1)},\bs \mu^{(1)}}_0^+\otimes \ket*{\bs \lambda^{(2)},\bs \mu^{(2)}}_0^+ \otimes \dots \otimes \ket*{\bs \lambda^{(k)},\bs \mu^{(k)}}_0^+.
$$

Given two pairs of partitions $(\bs\lambda^{(1)},\bs\mu^{(1)})$, $(\bs\lambda^{(2)},\bs \mu^{(2)})$ and $j\in\hat I$,
we say $(\bs\lambda^{(1)},\bs\mu^{(1)})\geq_j (\bs\lambda^{(2)},\bs \mu^{(2)})$ if the following three conditions are met.
\begin{itemize}
 \item We have $\ell(\bs\lambda^{(2)})\leq \ell(\bs\mu^{(1)}), \quad \ell(\bs\lambda^{(2)})\leq \ell(\bs\lambda^{(1)}), \quad \ell(\bs\mu^{(2)})\leq \ell(\bs\mu^{(1)})$.

\item For $i=1,\dots,\ell(\bs\lambda^{(2)}),\ $ we have the inequality $\lambda^{(1)}_i\geq \lambda^{(2)}_i$, where equality is allowed if either $s_{\lambda^{(1)}_i+j}=-s_j\ {\text {or}}\  \lambda_i^{(1)}=0$.  
    
\item For $i=1,\dots,\ell(\bs\mu^{(2)}),$ we have  the inequality
    $\mu^{(1)}_i\geq \mu^{(2)}_i,$  where equality is allowed if either $s_{-\mu^{(1)}_i+j}=-s_j$ or  $\mu_i^{(1)}=0$.
\end{itemize}

The relation $\geq_j$ is transitive in the following sense.
If  $(\bs\lambda^{(1)},\bs\mu^{(1)})\geq_j(\bs\lambda^{(2)},\bs \mu^{(2)})$ and $(\bs\lambda^{(2)},\bs\mu^{(2)})\geq_j(\bs\lambda^{(3)},\bs \mu^{(3)})$ then we have
$(\bs\lambda^{(1)},\bs\mu^{(1)})\geq_j(\bs\lambda^{(3)},\bs \mu^{(3)})$ provided $\ell(\bs\mu^{(2)})\leq \ell (\bs\la^{(2)})$.

The relation $\geq_j$ is not reflexive in general. We call a pair $(\bs \lambda,\bs \mu)$ $j$-self-comparable if $(\bs \lambda,\bs \mu)\geq_j (\bs \lambda,\bs \mu)$. In particular, $(\bs\emptyset,\bs\emptyset)$ is $j$-self-comparable.

We set 
\begin{equation}\label{3d condition}
\mc M^{(k)}(u)=\langle  \ket*{\vec{\bs \lambda},\vec{\bs \mu}}^{(k)}\ |\ (\bs\lambda^{(i)},\bs\mu^{(i)})\geq_0 (\bs\lambda^{(i+1)},\bs \mu^{(i+1)}) , \ i=1,\dots,k-1 \rangle.
\end{equation}

Note that $\mc M^{(k)}(u)$ is an $\Es$-module with action given by the coproduct $\eqref{coproduct}$. However, this action is not stable under the change of $k$. We can make it stable on a subspace as follows.

Let $k>3$ and set
$$
\mathring{\mc M}^{(k)}(u)=\langle \ket*{\vec{\bs \lambda},\vec{\bs \mu}}^{(k)}\in \mc M^{(k)}(u)\ |\   \bs\lambda^{(j)}=\bs\mu^{(j)}=\emptyset,\ j=k-2,k-1,k \rangle. 
$$

For $r\in \Z$ and $K\in \Cx$, let
\begin{align}\label{f Macmahon}
    f^{(r)}(K,z)=\dfrac{K-K^{-1}z}{q^r-q^{-r}z}.
\end{align}

Set
$$
F_i^{(k)}(z)=\Delta^{(k)}F_i(z) \qquad (i\in\hat I),
$$
and
$$
E_i^{(k)}(z)=\begin{cases}
    f^{(k)}(K,u/z) \Delta^{(k)} E_0(z)  &(i=0);\\
    \Delta^{(k)} E_i(z)  &(i\in I),
\end{cases}  \qquad 
(K_i^{\pm})^{(k)}(z)=\begin{cases}
    f^{(k)}(K,u/z) \Delta^{(k)} K_0^\pm(z)  &(i=0);\\
    \Delta^{(k)} K_i^\pm (z)  &(i\in I),
\end{cases}
$$
where $\Delta^{(k)}$ is the $k$-th iteration of the coproduct \eqref{coproduct}.

Now, we have the following lemma.
\begin{lem}\label{M^k lem}
Let $v\in \mathring{\mc M}^{(k)}$. Then $E_i^{(k)}(z)v$,  $F_i^{(k)}(z)v$,  $(K_i^\pm)^{(k)}(z)v$, are well defined vectors in $\mc M^{(k)}$. Moreover,
these operators satisfy the relations \eqref{relCK}-\eqref{Serre6} on $v$.
\end{lem}
\begin{proof}
The proof of Lemma \ref{M^k lem} is similar to the proof of Lemma \ref{U^k lem}.
\end{proof}
We have a natural embedding 
$$\phi^{(k)}:\ \mc M^{(k)}(u)\hookrightarrow \mc M^{(k+1)}(u), \qquad  \ket*{\vec{\bs \lambda},\vec{\bs \mu}}^{(k)} \mapsto  \ket*{\vec{\bs \lambda},\vec{\bs \mu}}^{(k)}\otimes \ket{{\bs \emptyset},{\bs \emptyset}}_0.$$
Note that $\phi^{(k)}(\mathring{\mc M}^{(k)})\subset \mathring{\mc M}^{(k+1)}$ and $\phi^{(k+2)}\circ \phi^{(k+1)}\circ \phi^{(k)}(\mc M^{(k)})=\mathring{\mc M}^{(k+3)}$.

Then, we have the following lemma.

\begin{lem}\label{stable lem Macmahon}
For any $v\in \mathring{\mc M}^{(k)}(u)$ and $i\in \hat I$, we have
\begin{align*}
 &\phi^{(k)} E_i^{(k)}(z)v=  E_i^{(k+1)}(z)\phi^{(k)}v, \\
 &\phi^{(k)} F_i^{(k)}(z)v=  F_i^{(k+1)}(z)\phi^{(k)}v, \\
 &\phi^{(k)} (K_i^\pm)^{(k)}(z)v=  (K^\pm_i)^{(k+1)}(z)\phi^{(k)}v.
\end{align*}
\end{lem}
\begin{proof}
Lemma \ref{stable lem Macmahon} is similar to Lemma \ref{stable lem}.
\end{proof}

Let $\mc M(u)$ be the inductive limit of $\mc M^{(k)}(u)$. The basis of $\mc M(u)$ is given by $\ket*{\vec{\bs \lambda},\vec{\bs \mu}}$ with only finitely many non-empty paritions $\bs\lambda^{(i)},\bs\mu^{(i)}$, such that $(\bs\lambda^{(i)},\bs\mu^{(i)})\geq_0 (\bs\lambda^{(i+1)},\bs \mu^{(i+1)})$ for all $i\in \Z_{>0}$. 

We define an action of $\Es$ on $\mc M(u)$ in the obvious way. Namely, for any $v\in\mc M(u)$, there exists $k$, such that $v\in \mathring{\mc M}^{(k)}$. Then $E_i(z)$, $F_i(z)$, $K_i^\pm(z)$ act on $v\in\mc M(u)$ as $E_i^{(k)}(z)$, $F_i^{(k)}(z)$, $(K_i^\pm)^{(k)}(z)$.

We say that $K$ is {\it generic} if $K\neq q_1^aq_2^b$, $a,b \in \Z$.

\begin{thm}\label{pure thm 1 Macmahon}
The space $\mc M(u)$ is an irreducible tame highest weight $\Es$-module with highest weight $(f^{(0)}(K,u/z),1,\dots,1)$. If $K$ is generic, it is irreducible and tame.
\end{thm}
\begin{proof}
Theorem \ref{pure thm 1 Macmahon} is similar to Theorem \ref{pure thm 1}.
\end{proof}

Let now $s_0=1$. We construct the vacuum MacMahon module similarly. In particular, the direction of tensor products is reversed.
In this case, we consider the submodule of
$$
\mc F_{-\Lambda_0}(u)\otimes \mc F_{-\Lambda_0}(q_2u)\otimes \mc F_{-\Lambda_0}(uq_2^{2})\otimes \dots \otimes \mc F_{-\Lambda_0}(uq_2^{k-1}).
$$
described by the same condition \eqref{3d condition}. The submodules are stabilized when $k$ is increasing in the same way. The only difference is the change $f^{(k)}(K,u/z) \mapsto f^{(-k)}(K,u/z)$ in the renormalization of $E_0^{(k)}(z)$ and $(K_0^\pm)^{(k)}(z)$.

As the result, for every $i\in\hat I$, independent on the parity $s_i$,  we obtain the highest weight $\Es$-module with highest weight $(1,\dots,1, (K-K^{-1}u/z)/(1-u/z),1,\dots,1)$, where the non-trivial function is in the $i$-th position. We denote this module $\mc M_i(u)$ and call it the {\it MacMahon module}. The basis of $\mc M_i(u)$ is parameterized by the set
\begin{equation*}
\{  (\vec{\bs \lambda},\vec{\bs \mu}),{\text{where } } (\bs\lambda^{(j)},\bs\mu^{(j)})\geq_i (\bs\lambda^{(j+1)},\bs \mu^{(j+1)}) , \ j\in\Z_{>0} , \ {\text{and }} (\bs\lambda^{(j)},\bs\mu^{(j)})=(\bs\emptyset,\bs\emptyset) {\text{ for almost all } j} \}.
\end{equation*}

\subsection{MacMahon modules with vertical boundaries}
We can repeat the construction of the vacuum MacMahon module in a more general setting. This generalization depends on a pair of partitions with special properties. We call such partitions colorless self-comparable pairs of partitions.

Let $s_0=-1$.
Let $\ket{\bs \lambda, \bs \mu}_0\in\mc F_{\Lambda_0}$. Then $\bs \lambda$ is a $\s^-_+$-partition, $\bs \mu$ is a $\s^-_-$-partition, and $\ell(\bs \lambda)\geq\ell(\bs \mu)$.  For $i\in\hat I$, let 
\begin{align}\label{a and b}
a_i=\#\{j \ |\ \lambda_j\equiv i+1,\ j\leq \ell(\lambda)\},\qquad b_i=\#\{j \ |\ -\mu_j\equiv i,\ j\leq \ell(\lambda)\}.
\end{align}

In particular, a pair of partitions $(\bs \lambda, \bs \mu)$ is $0$-self-comparable if and only if $a_i=0$ whenever $s_{i+1}=s_0=-1$ and $b_i=0$ whenever $s_{i}=s_0=-1$. 

We call a pair of partitions $(\bs \lambda, \bs \mu)$ {\it  colorless} if $a_i=b_{i+1}$ for all $i\in\hat I$. It is not hard to see that $(\bs \lambda, \bs \mu)$ is colorless if and only if the number of boxes of each color in its diagram (cf. Figure \ref{s-tab}) is the same.

\medskip

Let $(\bs \gamma,\bs \epsilon)$ be a colorless $0$-self-comparable 
pair of partitions, with  $\ell(\bs \gamma)\geq \ell (\bs\epsilon)$.
We again consider the tensor product in negative direction
$$
\mc F_{\Lambda_0}(u)\otimes \mc F_{\Lambda_0}(uq_2^{-1})\otimes \mc F_{\Lambda_0}(uq_2^{-2})\otimes \dots \otimes \mc F_{\Lambda_0}(uq_2^{-k+1})
$$
and the subspace
\begin{equation}\label{}
\mc M^{(k)}_{(\bs \gamma,\bs\epsilon)} (u)=\langle  \ket*{\vec{\bs \lambda},\vec{\bs \mu}}^{(k)}\ |\ (\bs\lambda^{(i)},\bs\mu^{(i)})\geq_0 (\bs\lambda^{(i+1)},\bs \mu^{(i+1)}) , \ i=1,\dots,k-1,\   (\bs\lambda^{(k)},\bs\mu^{(k)})\geq_0 (\bs\gamma,\bs \epsilon) \rangle.
\end{equation}
This subspace is not preserved by the $\Es$-action. We renormalize the action and go to the inductive limit.

Let
$$
\mathring{\mc M}^{(k)}_{(\bs \gamma,\bs\epsilon)}(u)=\langle \ket*{\vec{\bs \lambda},\vec{\bs \mu}}^{(k)}\in \mc M^{(k)}_{(\bs \gamma,\bs\epsilon)}(u)\ |\   (\bs\lambda^{(j)},\bs\mu^{(j)})=(\bs \gamma,\bs\epsilon)\ j=k-2,k-1,k \rangle. 
$$

Set
$$
F_i^{(k)}(z)=\Delta^{(k)}F_i(z) \qquad (i\in\hat I),
$$
and

\begin{align*}
&E_i^{(k)}(z)=\begin{cases}
    f^{(k-\ell(\gamma))}(K,u/z)f^{(k)}_{0,(\bs \gamma,\bs\epsilon)}(u/z) \Delta^{(k)} E_0(z)  &(i=0);\\
   f^{(k)}_{i,(\bs \gamma,\bs\epsilon)}(u/z) \Delta^{(k)} E_i(z)  &(i\in I),
\end{cases}\\
&(K_i^{\pm})^{(k)}(z)=\begin{cases}
     f^{(k-\ell(\gamma))}(K,u/z)f^{(k)}_{0,(\bs \gamma,\bs\epsilon)}(u/z) \Delta^{(k)} K_0^\pm(z)  &(i=0);\\
    f^{(k)}_{i,(\bs \gamma,\bs\epsilon)}(u/z) \Delta^{(k)} K_i^\pm (z)  &(i\in I),
\end{cases}
\end{align*}
where $\Delta^{(k)}$ is the $k$-th iteration of the coproduct \eqref{coproduct}, $f^{(k-\ell(\gamma))}(K,u/z)$ is the function \eqref{f Macmahon}, and
\begin{align*}
    f^{(k)}_{i,(\bs \gamma,\bs\epsilon)}(z)=\prod_{\substack{j=1\\i\equiv \gamma_j-1}}^{\ell(\gamma)}\left(q^{k-j}q_1^{\overline{(\gamma_j-1)}/2}-q^{j-k}q_1^{-\overline{(\gamma_j-1)}/2}z\right)   \prod_{\substack{j=1\\i\equiv -\epsilon_j}}^{\ell(\gamma)} \left(q^{k-j}q_3^{-\overline{(-\epsilon_j)}/2}-q^{j-k}q_3^{\overline{(-\epsilon_j)}/2}z \right) \times \\ \prod_{\substack{j=1\\i\equiv \gamma_j}}^{\ell(\gamma)} \left(q^{k-j+1}q_1^{\overline{\gamma_j}/2}-q^{j-k-1}q_1^{-\overline{\gamma_j}/2}z\right)^{-1} \prod_{\substack{j=1\\i\equiv -\epsilon_j-1}}^{\ell(\gamma)} \left(q^{k-j+1}q_3^{-\overline{(-\epsilon_j-1)}/2}-q^{j-k-1}q_3^{\overline{(-\epsilon_j-1)}/2}z \right)^{-1}.
\end{align*}

Note that $f^{(k)}_{i,(\bs \gamma,\bs\epsilon)}(z)$ are degree zero rational functions, with inverse values at $z=0$ and $z=\infty$.  Moreover, all zeroes and poles non-trivially depend on $k$. Given that $(\bs \gamma,\bs\epsilon)$ is self-comparable, the existence of such functions is equivalent to the property of 
$(\bs \gamma,\bs\epsilon)$ being colorless.

Then, we have an analog of Lemma \ref{M^k lem}.
\begin{lem}\label{M^k_gamma lem}
Let $v\in \mathring{\mc M}^{(k)}_{(\bs \gamma,\bs\epsilon)}$. Then $E_i^{(k)}(z)v$,  $F_i^{(k)}(z)v$,  $(K_i^\pm)^{(k)}(z)v$, are well defined vectors in $\mc M^{(k)}_{(\bs \gamma,\bs\epsilon)}$. Moreover,
these operators satisfy the relations \eqref{relCK}-\eqref{Serre6} on $v$.
\end{lem}
\begin{proof}
The proof of Lemma \ref{M^k_gamma lem} is similar to proof of Lemma \ref{U^k lem}.
\end{proof}

We have an embedding 
$$\phi^{(k)}_{(\bs \gamma,\bs\epsilon)}:\ \mc M^{(k)}_{(\bs \gamma,\bs\epsilon)}(u)\hookrightarrow \mc M^{(k+1)}_{(\bs \gamma,\bs\epsilon)}(u), \qquad  \ket*{\vec{\bs \lambda},\vec{\bs \mu}}^{(k)} \mapsto  \ket*{\vec{\bs \lambda},\vec{\bs \mu}}^{(k)}\otimes \ket{(\bs \gamma,\bs\epsilon)}_0.$$

Let $\mc M^{(\bs \gamma,\bs\epsilon)}(u)$ be the inductive limit of $\mc M^{(k)}_{(\bs \gamma,\bs\epsilon)}(u)$. The basis of $\mc M^{(\bs \gamma,\bs\epsilon)}(u)$ is given by $\ket*{\vec{\bs \lambda},\vec{\bs \mu}}$ with only finitely many $(\bs\lambda^{(i)},\bs\mu^{(i)})$ different from $(\bs \gamma,\bs\epsilon)$, such that $(\bs\lambda^{(i)},\bs\mu^{(i)})\geq_0 (\bs\lambda^{(i+1)},\bs \mu^{(i+1)})$ for all $i\in \Z_{>0}$. 

\begin{thm}\label{gamma thm Macmahon}
The space $\mc M^{(\bs \gamma,\bs\epsilon)}(u)$ is an irreducible tame highest weight $\Es$-module with highest weight vector $\ket{\bs \gamma,\bs\epsilon}_0\otimes \ket{\bs \gamma,\bs\epsilon}_0\otimes \dots\ $. If $K$ is generic, it is irreducible and tame.
\end{thm}
\begin{proof}
Theorem \ref{gamma thm Macmahon} is similar to Theorem \ref{pure thm 1 Macmahon}.
\end{proof}

We call $\mc M^{(\bs \gamma,\bs\epsilon)}(u)$ the MacMahon module with vertical boundary $(\bs \gamma,\bs\epsilon)$.

In a similar fashion, for every $i\in\hat I$, for any parity $s_i$, and any colorless $i$-self-comparable  pair of partitions $(\bs \gamma,\bs \epsilon)$, we construct an irreducible tame $\Es$-submodule of the following tensor product in direction $s_i$
$$
\mc F_{-s_i\Lambda_i}(u)\otimes \mc F_{-s_i\Lambda_i}(uq_2^{s_i})\otimes \mc F_{-s_i\Lambda_i}(uq_2^{2s_i})\otimes \cdots 
$$
and with the highest weight vector  $\ket{\bs \gamma,\bs\epsilon}_i\otimes \ket{\bs \gamma,\bs\epsilon}_i\otimes \cdots\,$ .
We denote this module $\mc M_i^{(\bs \gamma,\bs\epsilon)}(u)$ and call it the {\it MacMahon module with vertical boundary $(\bs \gamma,\bs\epsilon)$}. The basis of $\mc M_i^{(\bs \gamma,\bs\epsilon)}(u)$ is parameterized by the set
\begin{equation*}
\{  (\vec{\bs \lambda},\vec{\bs \mu}),{\text{where } } (\bs\lambda^{(j)},\bs\mu^{(j)})\geq_i (\bs\lambda^{(j+1)},\bs \mu^{(j+1)}) , \ j\in\Z_{>0} , \ {\text{and }} (\bs\lambda^{(j)},\bs\mu^{(j)})=(\bs \gamma,\bs\epsilon), {\text{for almost all } j} \}.
\end{equation*}

\subsection{MacMahon modules with general boundaries}
Let $s_0=-1$. 
Let $(\bs \gamma,\bs \epsilon)$ be a colorless $0$-self-comparable 
pair of partitions, with  $\ell(\bs \gamma)\geq \ell (\bs\epsilon)$. Let $\bs\alpha=\alpha_1\geq \alpha_2 \geq \dots \geq \alpha_k>0$ be a partition with $k$ parts.

The module $\mc M_{0}^{(\bs \gamma,\bs \epsilon)}(u)$ is constructed using tensor product in negative direction. We multiply it by extra tensor product of Fock modules:
\begin{align}\label{prod alpha}
\mc F_{\Lambda_0}(uq_2^{\alpha_1+k})\otimes\mc F_{\Lambda_0}(uq_2^{\alpha_2+k-1})\otimes \dots\otimes \mc F_{\Lambda_0}(uq_2^{\alpha_k+1})\otimes \mc M_0^{(\bs \gamma,\bs \epsilon)}(u).
\end{align}

It has a basis of the form
\begin{align}\label{vectors alpha}
\ket{\tilde{\bs\nu}^{(1)},\tilde{\bs\pi}^{(1)}}_0^+\otimes\ket{\tilde{\bs\nu}^{(2)},\tilde{\bs\pi}^{(2)}}_0^+\otimes \dots \otimes \ket{\tilde{\bs\nu}^{(k)},\tilde{\bs\pi}^{(k)}}_0^+\otimes\ket*{\vec{\bs\lambda},\vec{\bs\mu}}.
\end{align}

We define new (generalized) partitions $\bs\nu^{(i)},\bs \pi^{^{(i)}}$ as follows.
$$
\bs \nu^{(i)}_{j}=\begin{cases} \infty   & (j\leq \alpha_i);\\
\tilde{\bs \nu}^{(i)}_{j-\alpha_i} & (j>\alpha_i),
\end{cases} \qquad 
\bs \pi^{(i)}_{j}=\begin{cases} \infty   & (j\leq \alpha_i);\\
\tilde{\bs \pi}^{(i)}_{j-\alpha_i} & (j>\alpha_i).
\end{cases} 
$$
Consider the subspace $\mc M_0^{(\bs \gamma,\bs \epsilon),\bs\alpha}(u)$ of the tensor product \eqref{prod alpha}  spanned by vectors \eqref{vectors alpha} such that
$$
(\bs{\nu}^{(1)},\bs{\pi}^{(1)})\geq_0(\bs{\nu}^{(2)},\bs{\pi}^{(2)})\geq_0\dots \geq_0(\bs{\nu}^{(k)},\bs{\pi}^{(k)})\geq_0 (\bs\lambda^{(1)},\bs\mu^{(1)}).
$$
Here, by convention, $\infty$ is larger than any integer and also $\infty\geq\infty$ is allowed for any parity.

The coproduct \eqref{coproduct} does not define a $\Es$-module structure on the whole space \eqref{prod alpha}, however, it does on the subspace $\mc M_0^{(\bs \gamma,\bs \epsilon),\bs\alpha}(u)$. More precisely, we have the following theorem.
 
\begin{thm}\label{MacMahon with boundaries}
The subspace $\mc M_0^{(\bs \gamma,\bs \epsilon),\bs\alpha}(u)$ is an irreducible tame highest weight $\Es$-module with highest weight vector given by the product of empty partitions.
\end{thm}
\begin{proof}
The theorem follows from Lemma \ref{lem: tensor fock +} and the construction of the MacMahon module $\mc M_0(u)$
\end{proof}

More generally, for any $i\in\hat I$, any $s_i$, any colorless $i$-self-comparable 
pair of partitions $(\bs \gamma,\bs \epsilon)$, with  $\ell(\bs \gamma)\geq \ell (\bs\epsilon)$, and a partition $\bs\alpha$ with $k$ parts,  we construct an irreducible tame $\Es$-module $\mc M_i^{(\bs \gamma,\bs \epsilon),\bs\alpha}(u)$. The basis of  $\mc M_i^{(\bs \gamma,\bs \epsilon),\bs\alpha}(u)$ is parameterized by an infinite sequence of pairs of partitions $(\bs \lambda^{(i)},\bs\mu^{(i)})$, $i\in\Z_{>0}$, such that $(\bs \lambda^{(i)},\bs\mu^{(i)})\geq_i(\bs \lambda^{(i+1)},\bs\mu^{(i+1)})$, such that $\lambda^{(i)}_j=\mu^{(i)}_j=\infty$ if and only if $i\leq k$, $j\leq \alpha_i$, and such that $(\bs \lambda^{(i)},\bs\mu^{(i)})=(\bs \gamma,\bs \epsilon)$ for all but finitely many $i$.

We call the module $\mc M_i^{(\bs \gamma,\bs \epsilon),\bs\alpha}(u)$ the {\it MacMahon module with vertical boundary $(\bs \gamma,\bs \epsilon)$ and horizontal boundary $\bs \alpha$}.

\section{Pictures}\label{sec: pictures}
\subsection{General conventions}
The study of highest weight $\Es$-modules is related to combinatorics of supersymmetric versions of partitions and plane partitions. We describe a convenient way to visualize the modules appearing in this paper. 

We picture basis vectors by a collection of colored boxes in 2 (tensor products of vector and covector representations, Fock modules) or 3 (tensor products of Fock modules, MacMahon modules) dimensions. The basic vectors are eigenvectors of the generators $K^{\pm}_i(z)$. The eigenvalues are rational functions and are computed from the picture as we describe below. Operators $F_i(z)$ act by adding boxes of color $i$, and operators $E_i(z)$ by removing boxes of color $i$. The matrix coefficients of $F_i(z)$ and $E_i(z)$ are delta functions multiplied by some constants. The support of the delta function is given by the $q$-content of the box which is being added or removed.

\medskip

We consider boxes with coordinates $(i,j)$ or $(i,j,k)$, where $i,j,k\in\Z$. In our pictures, the axes always look as given in Figure \ref{axes pic}. The box with coordinates $(0,0)$ or $(0,0,0)$ will have color shown in red.

\begin{figure}[H]
\begin{tikzpicture}
\tgyoung(0,0pt,<\textcolor{red}{0}>)
\draw[->] (13pt,13pt)--++(180:-26pt);
\draw[->] (0pt,0pt)--++(-90:26pt);
\draw[-] (39pt,3pt) node[scale=0.75]{$i$};
\draw[-] (-9pt,-25pt) node[scale=0.75]{$j$};

\draw[->] (120pt,-13pt)++(-20:13pt)--++(-20:26pt);
\draw[->] (120pt,-13pt)++(90:13pt)--++(90:26pt);
\draw[->] (120pt,-13pt)++(-140:13pt)--++(-140:26pt);
\draw[-] (120pt,0pt)--++(-20:13pt)--++(-140:13pt)--++(160:13pt)--++(40:13pt);
\draw[-] (120pt,0pt)++(-20:13pt)--++(-90:13pt)--++(-140:13pt)--++(90:13pt);
\draw[-] (120pt,0pt)++(-140:13pt)--++(-90:13pt)--++(-20:13pt);
\draw[-] (129pt,26pt) node[scale=0.75]{$k$};
\draw[-] (129pt,26pt)++(-20:30pt)++(-90:30pt) node[scale=0.75]{$i$};
\draw[-] (111pt,26pt)++(-140:30pt)++(-90:32pt) node[scale=0.75]{$j$};

\draw[-] (121pt,-6.5pt) node[red]{$0$};
\end{tikzpicture}
\caption{The labeling of axes.}\label{axes pic}
\end{figure}

Define the color of the box with coordinates $(i,j,k)$ to be $i-j+\ell$ modulo $m+n$, where $\ell$ is the color of the box with coordinates $(0,0,0)$. So the color increases in the $i$-th direction, decreases in the $j$-th direction, and is constant in the $k$-th direction.

Recall that our constructions of Fock and MacMahon modules depend on the direction $s$ of the tensor products. The Fock modules $\mc F_{\Lambda_i}$, for all $i\in \hat I$, the Fock modules $\mc F_{(r+1)\Lambda_{i-1}- r\Lambda_{i}}$, $\mc F_{(r+1)\Lambda_{i}- r\Lambda_{i-1}}$, $r>0$, $i\in \hat I$, such that $s_i=-1$, correspond to positive directions. For the MacMahon modules $\mc M_i$, $i\in\hat I$, such that $s_i=-1$, the direction is negative, however, each layer corresponds to positive direction. For all these cases, we say the directions is positive. In other cases we say it is negative.

For a module in positive direction, we set $s=1$. 
For a module in negative direction, we set $s=-1$.
Given a picture of a basis vector of a module,
define the $q$-content $c(i,j,k)$ of the box with coordinates $(i,j,k)$ by the rule
\begin{align}\label{q content}
c(i,j,k)=\begin{cases} q_1^{-\overline{(i-j+\ell)}}q_2^{s(j-k)} & (i\geq j); \\
q_3^{\overline{(i-j+\ell)}}q_2^{s(i-k)} & (i<j),
\end{cases}
\end{align}
where the bar is defined by \eqref{overline}, \eqref{periodic}, and $\ell$ is the color of the box with coordinates $(0,0,0)$.
If we are in two dimensions, then the formula for the $q$-content is the same with $k$ set to be $0$.
 
\medskip

We choose for our examples of pictures the following values: $m=3,n=2$ and $\s=(1,1,1,-1,-1)$. In particular, $s_0=s_5=-1$.

\subsection{Pictures of tensor products of  vector and covector representations}\label{sec vector pic}

Basic vectors in modules $V(u)$ and $W(u)$ are pictured as semi-infinite arrays of colored boxes, see Figure \ref{fig1}.

The vector $[u]_r\in V(u)$ is represented as a horizontal semi-infinite row of boxes with coordinates $(i,0)$, $i\leq r$, see Figure \ref{fig1}.

Note that here we chose the box with coordinates $(0,0)$ to have color $0$, in particular, the last box $(r,0)$ has color $r$.
The $i$ coordinates are written below the boxes. 

Similarly, the vector $[u]^{-r}\in W(u)$ is represented as a vertical semi-infinite column of boxes with coordinates $(0,j)$, $j\leq r$.
The $j$ coordinates are written to the right of the boxes. Note that the last box $(0,r)$ has color $-r$.

\begin{figure}[H]
\centering
\begin{tikzpicture}
\node at (0pt,7pt) {$[u]_4$:};
\tgyoung(24pt,3pt,:\hdots)
\tgyoung(40pt,0cm,4<\textcolor{red}{0}>1234)
\tgyoung(40pt,-14pt,:<^{-1}>:<^0>:<^1>:<^2>:<^3>:<^4>)

\node at (196pt,7pt) {$[u]^{-2}$:};
\tgyoung(220pt,26pt,:\vdots)
\tgyoung(220pt,13pt,1,<\textcolor{red}{0}>,4,3)
\tgyoung(235pt,13pt,:<_{-1}>,:<_0>,:<_1>,:<_2>)
\end{tikzpicture}
\caption{Semi-infinite arrays depicting the vectors $[u]_4$ and $[u]^{-2}$ with $N=5$. }
\label{fig1}
\end{figure}
The horizontal row ending at the box of color $i$ is odd if and only if $s_{i+1}=-1$, while the vertical column ending at the box of color $i$ is odd if and only is $s_{i}=-1$. In our example, $[u]_4$ is odd and $[u]^{-2}$ is even.

The vector $[u]_{(r_1,\dots,r_k)}=[u]_{r_1}\otimes[uq_2^{s}]_{r_2}\otimes\cdots \otimes[uq_2^{s(k-1)}]_{r_k}$ is then represented as $k$ semi-infinite rows stacked from top to bottom and shifted by one box to the right at each level. 
This is dictated by the definition of the $q$-content. For example, the $q$-content of the box corresponding to $[uq_2^{st}]_0$ must be $q_2^{st}$, so it is placed in the box with coordinates $(t,t)$.

\begin{figure}[H]
\centering
\begin{tikzpicture}
\node at (0pt,7pt) {$[u]_4$:};
\tgyoung(24pt,3pt,:\hdots)
\tgyoung(40pt,0cm,4<\textcolor{red}{0}>1234)
\tgyoung(40pt,-14pt,:<^{-1}>:<^0>:<^1>:<^2>:<^3>:<^4>)

\node at (0pt,-23pt) {$[uq_2]_4$:};
\tgyoung(24pt,-27pt,:\hdots)
\tgyoung(40pt,-30pt,401234)
\tgyoung(40pt,-44pt,:<^{-1}>:<^0>:<^1>:<^2>:<^3>:<^4>)

\node at (0pt,-53pt) {$[uq_2^{2}]_4$:};
\tgyoung(24pt,-57pt,:\hdots)
\tgyoung(40pt,-60pt,401234)
\tgyoung(40pt,-74pt,:<^{-1}>:<^0>:<^1>:<^2>:<^3>:<^4>)

\node at (0pt,-83pt) {$[uq_2^{3}]_3$:};
\tgyoung(24pt,-87pt,:\hdots)
\tgyoung(40pt,-90pt,40123)
\tgyoung(40pt,-104pt,:<^{-1}>:<^0>:<^1>:<^2>:<^3>)

\node at (0pt,-113pt) {$[uq_2^{4}]_2$:};
\tgyoung(24pt,-117pt,:\hdots)
\tgyoung(40pt,-120pt,4012)
\tgyoung(40pt,-134pt,:<^{-1}>:<^0>:<^1>:<^2>)

\node at (230pt,-53pt) {$[u]_{(4,4,4,3,2)}$:};
\tgyoung(270pt,-31pt,:\hdots)
\tgyoung(283pt,-44pt,:\hdots)
\tgyoung(296pt,-57pt,:\hdots)
\tgyoung(309pt,-70pt,:\hdots)
\tgyoung(322pt,-83pt,:\hdots)
\tgyoung(286pt,-34pt,4<\textcolor{red}{0}>1234:<\textcolor{green!50!black}{0}>,:;401234,::;40123<\textcolor{blue}{4}>,:::;40123:<\textcolor{green!50!black}{4}>,::::;401<\textcolor{blue}{2}>:<\textcolor{green!50!black}{3}>)
\end{tikzpicture}
\caption{Semi-infinite arrays depicting the vectors $[u]_4$, $[uq_2]_4$, $[uq_2^{2}]_4$, $[uq_2^{3}]_3$, $[uq_2^{4}]_2$, and $[u]_{(4,4,4,3,2)}$. }
\label{fig2}
\end{figure}
Now it is easy to describe the $\Es$-module inside $V(u)\otimes V(uq_2^s)\otimes\dots\otimes V(uq_2^{s(k-1)})$ from Lemma \ref{lem VV}.  

If $s=1$, it has a basis $[u]_{\bs \lambda}$ parameterized by $\s^-_+$-partitions $\bs \lambda$ with $k$ parts (which can be negative). In the pictures, each row should end no further than the row above, except that it can have one extra box of color $i$ if $s_{i+1}=-1$. In our example, the possible colors for the extra box are $3$ and $4$. We give an example in Figure \ref{fig2}.  Note the rows $1$, $2$ and $3$ all end at $4$ and rows $2$, $3$ have the extra box. 

If $s=-1$, it has a basis $[u]_{\bs \lambda}$ parameterized by $\s^+_+$-partitions $\bs \lambda$ with $k$ parts (which can be negative).
In the pictures, each row should end no further than the row above, except that it can have one extra box of color $i$ if $s_{i+1}=1$. In our example, the possible colors for the extra box would be $0$, $1$ and $2$. 

We call diagrams satisfying these rules {\it admissible}.

\medskip

We explain how to write the action of generators on such pictures. We start with $K_i^\pm(z)$. 

Given a diagram, a box is called {\it $i$-convex} if it has color $i$ and it can be removed (that is this box is present in the diagram and removing that box produces an admissible diagram).

We define the {\it order $d$} of an $i$-convex box as follows.

Suppose an $i$-convex box has coordinates $(\lambda_j-1+j,j)$, then we have two possibilities depending on the parity.
\begin{itemize}
    \item If $s_{i+1}=s$, then the order of the $i$-convex box is $d$ if $\la_j-1=\la_{j+d-1}\neq \la_{j+d}$. Note that,  if $s_i=s_{i+1}=s$, we must have $d=1$.
    \item If $s_{i+1}=-s$, then the order of the $i$-convex box is $d$ if $\la_j=\la_{j-d+1}\neq \la_{j-d}$ ($j-d$ is allowed to be zero).
\end{itemize}

In Figure \ref{fig2}, we have one $4$-convex box of order 3 in the third row and one $2$-convex box of order $1$ in the fifth row colored in blue.

\medskip

Similarly, a box is called {\it $i$-concave} if it has color $i$, and it can be added (that is this box is not present in the diagram and adding that box produces an admissible diagram). 

We also define the {\it order $d$} of an $i$-concave box as follows.

Suppose an $i$-concave box has coordinates $(\lambda_j+j,j)$, then we again have two possibilities depending on the parity.
\begin{itemize}
    \item If $s_i=-s$, then the order of the $i$-concave box is $d$ if $\la_j=\la_{j+d-1}\neq \la_{j+d}$.
    \item If $s_{i}=s$, then the order of the $i$-concave box is $d$ if $\la_j+1=\la_{j-d+1}\neq \la_{j-d}$ ($j-d$ is allowed to be zero). Note that if $s_{i+1}=s_i=s$, we must have $d=1$.
\end{itemize}

In Figure \ref{fig2}  we have one $0$-concave box of order 3 in the first row, one $4$-concave box of order $1$ in the fourth row, and one $3$-concave box of order 2 in the fifth row colored in green.

\medskip

Similar to the even case \cite[Lemma 3.4]{FJMM}, the action of $K_i^\pm(z)$ on a diagram can be described in terms of its convex and concave boxes which are now supplemented by their orders.

The eigenvalue of $K_i^\pm(z)$ on a diagram is a product of $\psi_d(cu/z)$ over $i$-concave and $i$-convex boxes, where $d$ is determined by the degree and $c$ by the $q$-content:
\begin{align}\label{V eigen}
    \prod_{i-{\text{convex}\ \Box}}\psi_{-s_{i+1} d(\Box)}(c(\Box)u/z)\,\prod_{i-{\text{concave}\ \Box}}\psi_{s_{i} d(\Box)}(c(\Box)u/z).
\end{align}

\medskip

In our example, $K_0^\pm(z)$ has eigenvalue $\psi_{-3}(q_1^{-1}u/z)$, $K_2^\pm(z)$ has eigenvalue $\psi_{-1}(q_1^{-2}q_2^4u/z)$, $K_3^\pm(z)$ has eigenvalue $\psi_{2}(q_1^{-3}q_2^4u/z)$, $K_4^\pm(z)$ has eigenvalue $\psi_{3}(q_1^{-2}q_2^2u/z)\psi_{-1}(q_1^{-2}q_2^3u/z)$, and  $K_1^{\pm}(z)$ has eigenvalue 1.

The operators $F_i(z)$ add $i$-concave boxes. Then, for a chosen $i$-concave box $\Box$ in the $j$-th row, the corresponding matrix coefficient is given by $s_i\delta(c(\Box)u/z)$ multiplied by the value of $K_i^+(z)$ acting on the diagram obtained by removing the row which contains the box and all rows below it, and by an extra sign given by $(-1)^{|i|\sum_{k=1}^{j-1}(1-s_{\lambda_k})/2}$. This sign is non-trivial  if and only if $F_i(z)$ is odd, and the total parity of the rows preceding row $j$ is also odd, see \eqref{sign tensor}.

The operators $E_i(z)$ remove $i$-convex boxes. Then, for a chosen $i$-convex box $\Box$ in the $j$-th row, the corresponding matrix coefficient is given by $\delta(c(\Box)u/z)$ multiplied by the value of $K_i^-(z)$ acting on the diagram obtained by removing the row which contains the box and all rows above it, and by the same extra sign $(-1)^{|i|\sum_{k=1}^{j-1}(1-s_{\lambda_k})/2}$ as in $F_i(z)$ case.

\medskip

Similarly, the vector $[u]^{(r_1,\dots,r_k)}=[u]^{-r_1}\otimes[uq_2^s]^{-r_2}\otimes\cdots \otimes [uq_2^{s(k-1)}]^{-r_{k}}$ is pictured as $k$ semi-infinite columns stacked from left to right and shifted down by one box at each level. 

\begin{figure}[H]
\centering
\begin{tikzpicture}
\node at (0pt,7pt) {$[u]^{-3}$:};
\tgyoung(24pt,26pt,:\vdots)
\tgyoung(24pt,13pt,1,<\textcolor{red}{0}>,4,3,2)
\tgyoung(38pt,13pt,:<_{-1}>,:<_0>,:<_1>,:<_2>,:<_3>)

\node at (90pt,7pt) {$[uq_2]^{-1}$:};
\tgyoung(119pt,26pt,:\vdots)
\tgyoung(119pt,13pt,1,0,4)
\tgyoung(133pt,13pt,:<_{-1}>,:<_0>,:<_1>)

\node at (180pt,7pt) {$[uq_2^2]^{-1}$:};
\tgyoung(209pt,26pt,:\vdots)
\tgyoung(209pt,13pt,1,0,4)
\tgyoung(223pt,13pt,:<_{-1}>,:<_0>,:<_1>)

\node at (270pt,7pt) {$[uq_2^3]^{0}$:};
\tgyoung(294pt,26pt,:\vdots)
\tgyoung(294pt,13pt,1,0)
\tgyoung(308pt,13pt,:<_{-1}>,:<_0>)

\node at (385pt,7pt) {$[u]^{(3,1,1,0)}$:};
\tgyoung(424pt,26pt,:\vdots)
\tgyoung(437pt,13pt,:\vdots)
\tgyoung(450pt,0pt,:\vdots)
\tgyoung(463pt,-13pt,:\vdots)
\tgyoung(424pt,13pt,1,<\textcolor{red}{0}>1,401,3401,<\textcolor{blue}{2}>:<\textcolor{green!50!black}{3}><\textcolor{blue}{4}><\textcolor{blue}{0}>,:<\textcolor{green!50!black}{1}>:::<\textcolor{green!50!black}{4}>)
\end{tikzpicture}
\caption{Semi-infinite arrays depicting the vectors $[u]^{-3}$, $[uq_2]^{-1}$, $[uq_2^2]^{-1}$, $[uq_2^3]^{0}$ and $[u]^{(3,1,1,0)}$. }
\label{fig3}
\end{figure}

It is again easy to describe the $\Es$-module inside $W(u)\otimes W(uq_2^s)\otimes\dots\otimes W(uq_2^{s(k-1)})$ from Lemma \ref{lem WW}. 

If $s=1$, it has a basis $[u]^{\bs \mu}$ parameterized by $\s^-_-$-partitions $\bs \mu$ with $k$ parts (which can be negative). In the pictures, each column should end no further than the column to the right, except that it can one extra box of color $i$ if $s_{i}=-1$. In our example, the possible colors for the extra box are $0$ and $4$.  We show an example in Figure \ref{fig3}. Note that the columns $2$ and $3$ end at $4$ and column $3$ has the extra box.

If $s=-1$, it has a basis $[u]^{\bs \mu}$ parameterized by $\s^+_-$-partitions $\bs \mu$ with $k$ parts (which can be negative).
In the pictures, each column should end no further than the column to the right, except that it can one extra box of color $i$ if $s_{i}=1$. In our example, the possible colors for the extra box are $1$, $2$, and $3$. 

We call diagrams satisfying the above rules admissible.

We again explain how to write the action of generators on such a picture. The procedure is similar to the previous case but with a few differences.

Given a diagram, a box is called {\it $i$-convex} if it has color $i$ and it can be removed (that is, removing that box produces an admissible diagram).

We define the {\it order $d$} of an $i$-convex box as follows.

Suppose an $i$-convex box has coordinates $(j,j+\mu_j)$, then we have two possibilities depending on the parity.
\begin{itemize}
    \item If $s_{i}=s$, then the order of the $i$-convex box is $d$ if $\mu_j-1=\mu_{j+d-1}\neq \mu_{j+d}$. Note that if $s_{i+1}=s_i=s$, we must have $d=1$.
    \item If $s_{i}=-s$, then the order of the $i$-convex box is $d$ if $\mu_j=\mu_{j-d+1}\neq \mu_{j-d}$ ($j-d$ is allowed to be zero).
\end{itemize}

In Figure \ref{fig3}, we have one $0$-convex box of order $1$ in the fourth column, one $2$-convex box of order $1$ in the first column, and one $4$-convex box of order $2$ in the third column colored in blue. 

Similarly, a box is called {\it $i$-concave} if it has color $i$, and it can be added (that is, adding that box produces an admissible diagram).

We also define the {\it order $d$} of an $i$-concave box as follows.

Suppose an $i$-concave box has coordinates $(j,j+\mu_j+1)$, then we again have two possibilities depending on the parity.
\begin{itemize}
    \item If $s_{i+1}=-s$, then the order of the $i$-concave box is $d$ if $\mu_j=\mu_{j+d-1}\neq \mu_{j+d}$.
    \item If $s_{i+1}=s$, then the order of the $i$-convex box is $d$ if $\mu_j+1=\mu_{j-d+1}\neq \mu_{j-d}$ ($j-d$ is allowed to be zero). Note that if $s_{i}=s_{i+1}=s$, we must have $d=1$.
\end{itemize}

 In Figure \ref{fig3}, we have one $1$-concave box of order $1$ in the first column, one $3$-concave box of order $2$ in the second column, and one $4$-concave box of order $1$ in the fourth column colored in green. 

\medskip

The eigenvalue of $K_i^\pm(z)$ on a diagram is a product of $\psi_d(cu/z)$ over $i$-concave and $i$-convex boxes, where $d$ is determined by the degree and $c$ by the $q$-content:
\begin{align}\label{W eigen}
    \prod_{i-{\text{convex}\ \Box}}\psi_{-s_{i} d(\Box)}(c(\Box)u/z)\,\prod_{i-{\text{concave}\ \Box}}\psi_{s_{i+1} d(\Box)}(c(\Box)u/z).
\end{align}

In our example, $K_1^{\pm}(z)$ has eigenvalue $\psi_{1}(u/z)$, $K_2^\pm(z)$ has eigenvalue $\psi_{-1}(q_3^{-1}u/z)$, 
$K_3^\pm(z)$ has eigenvalue $\psi_{-2}(q_3^{-2}q_2u/z)$,
$K_4^\pm(z)$ has eigenvalue $\psi_{2}(q_3^{-1}q_2^2u/z)\psi_{-1}(q_3^{-1}q_2^3u/z)$, and $K_0^\pm(z)$ has eigenvalue $\psi_{1}(q_2^3u/z)$.

The operators $F_i(z)$ add $i$-concave boxes. Then, for a chosen $i$-concave box $\Box$ in the $j$-th column, the corresponding matrix coefficient is given by $s_{i+1}\delta(c(\Box)u/z)$ multiplied by the value of $K_i^+(z)$ acting on the diagram obtained by removing the column which contains the box and all columns to the right of it, and an extra sign given by $(-1)^{|i|\sum_{k=1}^{j-1}(1-s_{-\mu_k})/2}$. This sign is nontrivial if and only if $F_i(z)$ is odd, and the total parity of the columns preceding column $j$ is also odd, see \eqref{sign tensor}.

The operators $E_i(z)$ remove $i$-convex boxes. Then, for a chosen $i$-convex box $\Box$ in the $j$-th column, the corresponding matrix coefficient is given by $\delta(c(\Box)u/z)$ multiplied by the value of $K_i^-(z)$ acting on the diagram obtained by removing the column which contains the box and all columns to the left of it, and by the same extra sign $(-1)^{|i|\sum_{k=1}^{j-1}(1-s_{-\mu_k})/2}$ as in $F_i(z)$ case.

\subsection{Pictures of Fock modules}
We start with the case of $\mc F_{\Lambda_0}(u)$ constructed in Theorem \ref{pure thm 1}. Then, a basis vector $\ket{\bs\lambda,\bs\mu}_0^+$ is pictured by combining the horizontal picture of $[u]_{\bs\lambda}$ and vertical picture of $[u]^{\bs \mu}$ aligned along the blue broken line.  We call such pictures $\s$-partitions. An example is given in Figure \ref{s-tab}.

\begin{figure}[H]
\centering
\begin{tikzpicture}
\node at (0pt,7pt) {$\ket{\bs\lambda,\bs\mu}_0^+$:};
\tgyoung(40pt,0cm,<\textcolor{red}{0}>123401<\textcolor{blue}{2}>:<\textcolor{green!50!black}{3}>,401234:<\textcolor{green!50!black}{0}>,340123<\textcolor{blue}{4}>,2340123:<\textcolor{green!50!black}{4}>,123<\textcolor{blue}{4}>012<\textcolor{blue}{3}>,012:<\textcolor{green!50!black}{3}>:<\textcolor{green!50!black}{4}>;<\textcolor{blue}{0}>:<\textcolor{green!50!black}{1}>,401::::<\textcolor{violet}{0}>,<\textcolor{blue}{3}>:<\textcolor{green!50!black}{4}>;<\textcolor{blue}{0}>,:<\textcolor{green!50!black}{2}>)
\draw[-,line width=0.5mm, blue] (40pt,13pt)-- ++(0pt,-13pt)-- ++(13pt,0)-- ++(0pt,-13pt)-- ++(13pt,0)-- ++(0pt,-13pt)-- ++(13pt,0)-- ++(0pt,-13pt)-- ++(13pt,0)-- ++(0pt,-13pt)-- ++(13pt,0)-- ++(0pt,-13pt)-- ++(13pt,0)-- ++(0pt,-13pt)-- ++(13pt,0);
\end{tikzpicture}
\caption{Young diagram depicting  $\ket{\bs\lambda,\bs\mu}_0^+\in \mathcal{F}_{\Lambda_0}(u)$ with $\bs\lambda=(8,5,5,4,4,1)$ and $\bs\mu=(7,5,5,1)$. }
\label{s-tab}
\end{figure}

The $i$-convex and $i$-concave boxes (pictured in blue and green, respectively) are the boxes that can be removed and added, respectively.
Note that now our vectors are infinite tensor products, so one is allowed to add boxes in the first empty row or column (if it does not break the repetition rule). The orders of the convex and concave boxes are computed as in Section \ref{sec vector pic}, with the convention that $\bs\mu$ is appended with zero parts so that  $\ell(\bs\mu)=\ell(\bs\la)$.

\medskip

Then, the eigenvalue of $K_i^\pm(z)$ on $\ket{\bs\la, \bs\mu}_0^+$ is the product over all $i$-concave and $i$-convex boxes of the form $\eqref{V eigen}$ if the box is in the $\bs \la$-diagram (above the broken blue line), and of the form $\eqref{W eigen}$ if the box is in the $\bs \mu$-diagram (below the broken blue line) with the following modification when $i=0$.

We color the $0$ box with coordinates $(\ell(\bs\la),\ell(\bs\la))$ in violet color. 
We include this box in the product as $0$-convex on the $\bs \mu$-side. In addition, if it is $0$-concave on the $\bs\la$-side, then this is not counted as $0$-concave for the eigenvalue of $K_0^\pm(z)$ (but instead it is counted as $0$-convex on the $\bs \mu$-side).

In general, the violet box in $\mc F_{\Lambda_i}$ is counted as $i$-convex on the $\bs\mu$-side  if $s_i=-1$ and as $i$-concave on the $\bs\lambda$-side if $s_i=1$. It accounts for the contribution of the tail to the eigenvalue of $K_i(z)$.

\medskip

Again, the operators $F_i(z)$ add $i$-concave boxes. Then, if the added box is in the $\lambda_j$ part, the corresponding matrix coefficient is given by $s_{i}\delta(c(\Box)u/z)$ multiplied by the value of $K_i^+(z)$ acting on the diagram obtained by removing all boxes except the first $j-1$ rows and the first $j-1$ columns, and the extra sign $(-1)^{|i|\sum_{k=1}^{j-1} (2-s_{\lambda_k}-s_{-\mu_k})/2}$.

If the added box is in the $\mu_j$ part, the corresponding matrix coefficient is given by $s_{i+1}\delta(c(\Box)u/z)$ multiplied by the value of $K_i^+(z)$ acting on the diagram obtained by removing all boxes except  the first $j$ rows and the first $j-1$ columns, and by the
extra sign 
$(-1)^{|i|(1-s_{\lambda_j}+\sum_{k=1}^{j-1}(2-s_{\lambda_k}-s_{-\mu_k}))/2}$.

The operators $E_i(z)$ remove $i$-convex boxes. Then, for a chosen box $\Box$, the corresponding matrix coefficient is given by $\delta(c(\Box)u/z)$ multiplied by the value of $K_i^-(z)$ acting on the diagram obtained by removing all boxes in the first $j$ rows and the first $j-1$ columns if the removed box is in the $\lambda_j$ part, and all boxes in the first $j$ rows and the first $j$ columns if the removed box is in the $\mu_j$ part with the extra sign as in the $F_i(z)$ case.

\medskip

Next, we look at the non-pure cases.  

We  consider the case $\mc F_{3\Lambda_3-4\Lambda_{2}}(uq_1^{-2})$.  The module $\mc F_{3\Lambda_3-4\Lambda_{2}}(uq_1^{-2})$ is constructed in Theorem \ref{non pure 1 0,r} using the inclusion 
$$\mc F_{3\Lambda_3-4\Lambda_{2}}(uq_1^{-2})\subset V(u)\otimes V(uq_2^{-1})\otimes V(uq_2^{-2})\otimes \mc F_{-\Lambda_2}(uq_1^{-2}q_2^{-3}).$$
Note that this module is in the negative direction.

Thus, the pictures are the same as for $\mc F_{\Lambda_2}(u)$ except that we have three extra rows, corresponding to the first three factors. An example is given in Figure \ref{nu mu la V}.  
The extra rows are on top of the green line. Note the alignment of these rows. We again use the color blue for the numbers in convex boxes, green for the numbers in concave boxes, and violet for the extra zero, which is counted as convex in the vertical direction.
The rules for the action of generators of $\Es$ are also the same.
 Note that the box with violet $2$ is counted as $2$-concave for the $\bs\lambda$-part.

\begin{figure}[H]
\centering
\begin{tikzpicture}
\node at (-5pt,7pt) {$\ket{(\bs\nu,\bs\lambda),\bs\mu}_{2,3}^{-}$:};
\tgyoung(40pt,0cm,:;<\textcolor{red}{3}>401:<\textcolor{green!50!black}{2}>,::;340<\textcolor{blue}{1}>,:::;340:<\textcolor{green!50!black}{1}>,:::;2340,:::;123<\textcolor{blue}{4}>:<\textcolor{green!50!black}{0}>,:::;<\textcolor{blue}{0}>:<\textcolor{green!50!black}{1}><\textcolor{blue}{2}>:<\textcolor{green!50!black}{3}>,::::<\textcolor{green!50!black}{4}>:::<\textcolor{violet}{2}>)
\draw[-,line width=0.5mm,blue] (40pt,13pt)-- ++(0pt,-13pt)-- ++(13pt,0pt)-- ++(0pt,-13pt)-- ++(13pt,0pt)-- ++(0pt,-13pt)-- ++(13pt,0pt)-- ++(0pt,-13pt)-- ++(13pt,0pt)-- ++(0pt,-13pt)-- ++(13pt,0pt)-- ++(0pt,-13pt)-- ++(13pt,0pt)-- ++(0pt,-13pt)-- ++(13pt,0pt);
\draw[-,line width=0.5mm,green!50!black] (79pt,-26pt)-- ++(52pt,0pt);
{\Yfillcolour{gray}
\Yfillopacity{.5}
\tgyoung(40pt,0pt,|7:,:|6:,::|5)}
\end{tikzpicture}
\caption{Young diagram depicting the vector $\ket{(\bs\nu,\bs\lambda),\bs\mu}_{2,3}^{-}\in \mathcal{F}_{3\Lambda_{3}-4\Lambda_{2}}(uq_1^{-2})$ for $\bs\nu=(4,4,3)$, $\bs\lambda=(4,3,1)$, and $\bs\mu=(2)$. }
\label{nu mu la V}
\end{figure}

Finally, one more type of pictures appears if we  consider the case $\mc F_{4\Lambda_0-3\Lambda_{4}}(uq_3)$.  The module $\mc F_{4\Lambda_0-3\Lambda_{4}}(uq_3)$ is constructed in Theorem \ref{non pure 2 0,r} using the inclusion 
$$\mc F_{4\Lambda_0-3\Lambda_{4}}(uq_3)\subset W(uq_3)\otimes W(uq_2q_3)\otimes W(uq_2^{2}q_3)\otimes \mc F_{\Lambda_0}(uq_2^3q_3).$$

Thus, the pictures are the same as for $\mc F_{\Lambda_0}(u)$ except that we have three extra columns, corresponding to the first three factors. The example is given in Figure \ref{nu mu la +}.  
The extra columns are to the left of the green line. Note the alignment of these columns.

\begin{figure}[H]
\centering
\begin{tikzpicture}
\node at (0pt,7pt) {$\ket{\bs\lambda,(\bs\nu,\bs\mu)}_{0,3}^{+}$:};
\tgyoung(40pt,0cm,<\textcolor{red}{4}>,34,2340123:<\textcolor{green!50!black}{4}>,12:<\textcolor{green!50!black}{3}><\textcolor{blue}{4}>012<\textcolor{blue}{3}>,01:::<\textcolor{green!50!black}{4}><\textcolor{blue}{0}>:<\textcolor{green!50!black}{1}>,:<\textcolor{green!50!black}{4}><\textcolor{blue}{0}>:::::<\textcolor{violet}{0}>)
\draw[-,line width=0.5mm,blue] (40pt,26pt)-- ++(0pt,-13pt)-- ++(13pt,0pt)-- ++(0pt,-13pt)-- ++(13pt,0pt)-- ++(0pt,-13pt)-- ++(13pt,0pt)-- ++(0pt,-13pt)-- ++(13pt,0pt)-- ++(0pt,-13pt)-- ++(13pt,0pt)-- ++(0pt,-13pt)-- ++(13pt,0pt)-- ++(0pt,-13pt)-- ++(13pt,0pt);
\draw[-,line width=0.5mm,green!50!black] (79pt,-26pt)-- ++(0pt,-39pt);
{\Yfillcolour{gray}
\Yfillopacity{.5}
\tgyoung(40pt,13pt,_8,:_7,::_6 )}
\end{tikzpicture}
\caption{Young diagram depicting the vector $\ket{\bs\lambda,(\bs\nu,\bs\mu)}_{0,3}^{+}\in \mathcal{F}_{4\Lambda_0-3\Lambda_{4}}(uq_3^{-1})$ for $\bs\nu=(5,5,1)$, $\bs\lambda=(4,4,1)$, and $\bs\mu=(1)$. }
\label{nu mu la +}
\end{figure}
As always, we use the color blue for the numbers in convex boxes, green for the numbers in concave boxes, and violet for the extra zero, which is counted as convex in the vertical direction.
The rules for the action of generators of $\Es$ are also the same.

\subsection{Pictures of MacMahon modules}
We discuss the pictures of basic vectors in the MacMahon modules. Such pictures, which we call plane $\s$-partitions, are naturally shown in three dimensions.

We start with the pure case $\mc M_0(u)$ constructed in Theorem \ref{pure thm 1 Macmahon}. An example of a basis vector is shown in Figure \ref{3d}.
\begin{figure}[H]
\centering
\begin{tikzpicture}
\draw[-] (90pt,6.5pt) node[scale=0.75]{$0$}++(-20:13pt) node[scale=0.75]{$1$}++(-20:13pt) node[scale=0.75]{$2$}++(-90:26pt)++(-20:13pt) node[scale=0.75]{$3$}++(-90:13pt)++(-20:13pt)++(4.75pt,0.75pt) node[scale=0.75]{$4$}
++(-20:13pt)++(-7.5pt,-1pt) node[scale=0.75]{$\textcolor{green!50!black}{0}$};
\draw[-] (90pt,6.5pt)++(-140:13pt)++(-20:13pt) node[scale=0.75]{$0$}++(-20:13pt) node[scale=0.75]{$1$}++(-90:26pt)++(-20:13pt) node[scale=0.75]{$2$} ++(-20:13pt) node[scale=0.75]{$3$}++(-90:13pt)++(-20:13pt) node[scale=0.75]{$\textcolor{blue}{4}$};
\draw[-] ((90pt,6.5pt)++(-140:26pt)++(-20:26pt) node[scale=0.75]{$\textcolor{blue}{0}$}++(-20:13pt)++(-90:26pt) node[scale=0.75]{$1$}++(-20:13pt) node[scale=0.75]{$2$} ++(-20:13pt)++(-90:13pt) node[scale=0.75]{$3$}++(-20:13pt)++(-3pt,-.5pt) node[scale=0.75]{$\textcolor{green!50!black}{4}$};
\draw[-] (90pt,6.5pt)++(-140:65pt)++(-20:45.5pt)++(-90:39pt) node[scale=0.75]{$\textcolor{green!50!black}{3}$}++(40:14pt)++(-90:7pt) node[scale=0.75]{$\textcolor{green!50!black}{0}$};
\draw[-] (90pt,6.5pt)++(-140:39pt)++(-20:39pt)++(-90:26pt) node[scale=0.75]{$\textcolor{violet}{0}$}++(-20:13pt) node[scale=0.75]{$\textcolor{blue}{1}$}++(-90:13pt)++(-20:13pt) node[scale=0.75]{$2$} ++(-20:13pt) node[scale=0.75]{$\textcolor{blue}{3}$};
\draw[-] (90pt,6.5pt)++(-140:13pt) node[scale=0.75]{$4$}++(-140:13pt) node[scale=0.75]{$3$}++(-140:13pt) node[scale=0.75]{$2$} ++(-140:13pt) node[scale=0.75]{$\textcolor{blue}{1}$}++(-140:13pt)++(-90:39pt)
node[scale=0.75]{$0$}++(-140:13pt)++(3pt,0) node[scale=0.75]{$\textcolor{green!50!black}{4}$};
\draw[-] (90pt,6.5pt)++(-140:26pt)++(-90:21pt)++(-20:13pt)++(180:6.5pt) node[scale=0.75]{$\textcolor{green!50!black}{3}$};
\draw[-] (90pt,6.5pt)++(-140:26pt)++(-90:13pt)++(-20:13pt)++(180:2pt) node[scale=0.75]{$4$}++(-140:13pt)++(-90:13pt) node[scale=0.75]{$3$}++(0:3.25pt)++(-140:13pt) node[scale=0.75]{$\textcolor{blue}{2}$} ++(-140:13pt)++(-90:13pt) node[scale=0.75]{$1$}++(-140:13pt)
node[scale=0.75]{$0$};
\draw[-] (90pt,6.5pt)++(-140:39pt)++(-20:26pt)++(-90:13pt) node[scale=0.75]{$4$}++(-140:13pt)++(-90:13pt) node[scale=0.75]{$3$}++(-140:13pt)++(-90:13pt) node[scale=0.75]{$2$} ++(-140:13pt) node[scale=0.75]{$1$}++(-140:13pt)
node[scale=0.75]{$\textcolor{blue}{0}$};
\draw[-] (90pt,6.5pt)++(-140:52pt)++(-20:39pt)++(-90:39pt) node[scale=0.75]{$4$};
\draw[-] (90pt,13pt)--++(-20:39pt)--++(-90:26pt)--++(-20:13pt)--++(-90:12pt)++(-90:1pt)++(-20:2.3pt)--++(-20:10.7pt)--++(-90:12pt);
\draw[-] (90pt,13pt)--++(-140:13pt)--++(-20:39pt)--++(-90:26pt)--++(-20:26pt)--++(-90:13pt)--++(-20:13pt)--++(-90:13pt);
\draw[-] (90pt,13pt)++(-20:13pt)--+(-140:13pt)++(-20:13pt)--+(-140:13pt)++(-20:13pt)--+(-140:13pt)++(-90:13pt)--+(-140:13pt)++(-90:13pt)--+(-140:13pt)++(-20:13pt)--+(-140:13pt)++(-90:13pt)++(-20:13pt)--+(-140:13pt);
\draw[-] (90pt,13pt)++(-20:13pt)++(-140:26pt)--++(-20:26pt)--++(-90:26pt)--++(-20:26pt)--++(-90:13pt)--++(-20:13pt)--++(-90:12pt);
\draw[-] (90pt,13pt)++(-140:13pt)++(-20:13pt)--+(-140:13pt)++(-20:13pt)--+(-140:13pt)++(-20:13pt)--+(-140:13pt)++(-90:13pt)--+(-140:13pt)++(-90:13pt)--+(-140:13pt)++(-20:13pt)--+(-140:13pt)++(-20:13pt)--+(-140:13pt)++(-90:13pt)--+(-140:13pt)++(-20:13pt)--+(-140:13pt)++(-90:13pt)--+(-140:12pt);
\draw[-] (90pt,13pt)++(-140:26pt)++(-20:26pt)--++(-140:13pt)--++(-20:13pt)--++(-90:26pt)--++(-20:26pt)--++(-90:13pt)--++(-20:13pt)--++(-20:13pt)--++(-90:13pt);
\draw[-] (90pt,13pt)++(-140:26pt)++(-20:39pt)--+(-140:13pt)++(-90:13pt)--+(-140:13pt)++(-90:13pt)--+(-140:13pt)++(-20:13pt)--+(-140:13pt)++(-20:13pt)--+(-140:13pt)++(-90:13pt)--+(-140:13pt)++(-20:13pt)--+(-140:13pt);
\draw[-] (90pt,13pt)++(-140:39pt)++(-20:39pt)++(-90:26pt)--++(-140:13pt)--++(-20:26pt)--++(-90:13pt)--++(-20:26pt)--++(-90:13pt);
\draw[-] (90pt,13pt)++(-140:39pt)++(-20:52pt)++(-90:26pt)--+(-140:13pt)++(-20:13pt)--+(-140:13pt)++(-90:13pt)--+(-140:13pt)++(-20:13pt)--+(-140:13pt)++(-20:13pt)--+(-140:13pt)++(-90:13pt)--+(-140:13pt);
\draw[-] (90pt,13pt)++(-140:13pt)--++(-140:52pt)--++(-90:39pt)--++(-140:13pt)--++(-90:13pt);
\draw[-] (90pt,13pt)++(-140:26pt)++(-20:13pt)--++(-140:39pt)--++(-90:26pt);
\draw[-] (90pt,13pt)++(-140:26pt)--+(-20:13pt)++(-140:13pt)--+(-20:13pt)++(-140:13pt)--+(-20:13pt)++(-140:13pt)--+(-20:13pt)++(-90:13pt)--+(-20:13pt)++(-90:13pt)--+(-20:13pt)++(-90:13pt)--+(-20:13pt)++(-140:13pt)--+(-20:13pt)++(-90:13pt)--+(-20:2.3pt);
\draw[-] (90pt,13pt)++(-20:13pt)++(-140:39pt)--++(-90:26pt)--++(-140:26pt)--++(-90:13pt)--++(-140:26pt)--++(-90:13pt);
\draw[-] (90pt,13pt)++(-20:13pt)++(-140:26pt)--++(-90:13pt)--++(-20:2.2pt)++(160:2pt)--++(-140:39pt);
\draw[-] (90pt,13pt)++(-20:26pt)++(-140:39pt)--++(-90:13pt)--++(-140:13pt)--++(-90:13pt)--++(-140:13pt)--++(-90:13pt)--++(-140:39pt)--++(-90:13pt);

\draw[-] (90pt,13pt)++(-140:39pt)++(-20:13pt)++(-90:13pt)--+(-20:13pt)++(-90:13pt)--+(-20:2.3pt)++(-140:13pt)--+(-20:13pt)++(-140:13pt)--+(-20:13pt)++(-90:13pt)--+(-20:13pt)++(-140:13pt)--+(-20:13pt)++(-140:13pt)--+(-20:13pt)++(-90:13pt)--+(-20:2.3pt);
\draw[-] (90pt,13pt)++(-20:13pt)++(-140:52pt)--++(-90:26pt);
\draw[-] (90pt,13pt)++(-20:39pt)++(-140:39pt)++(-90:13pt)--++(-140:13pt)--++(-90:13pt)--++(-140:13pt)--++(-90:13pt)--++(-140:26pt)--++(-140:13pt)--++(-90:13pt);
\draw[-] (90pt,13pt)++(-140:39pt)++(-20:26pt)++(-90:13pt)--+(-20:13pt)++(-140:13pt)--+(-20:13pt)++(-90:13pt)--+(-20:13pt)++(-140:13pt)--+(-20:13pt)++(-90:13pt)--+(-20:13pt)++(-140:13pt)--+(-20:13pt)++(-140:13pt)--+(-20:13pt)++(-140:13pt)--+(-20:13pt)++(-90:13pt)--+(-20:13pt);
\draw[-] (90pt,13pt)++(-90:39pt)++(-20:39pt)++(-140:91pt)--+(-90:13pt)++(40:13pt)--+(-90:13pt)++(40:13pt)--+(-90:13pt)++(40:13pt)++(90:13pt)++(-20:13pt)--+(-90:26pt)++(-20:13pt)++(-90:13pt)--+(-90:13pt)++(-20:13pt)--+(-90:13pt);
\draw[-] (90pt,13pt)++(-90:52pt)++(-20:39pt)++(-140:104pt)--++(40:39pt)--++(-20:13pt)++(40:13pt)--++(-20:39pt);
\draw[-] (90pt,13pt)++(-90:39pt)++(-20:39pt)++(-140:65pt)--++(-20:13pt)--++(40:13pt)--++(-20:13pt);
\draw[-] (90pt,13pt)++(-90:39pt)++(-20:39pt)++(-140:65pt)--++(40:13pt);
\draw[-,line width=0.5mm,blue] (90pt,13pt)--++(-140:13pt)--++(-20:13pt)--++(-140:13pt)--++(-20:13pt)--++(-140:13pt)--++(-20:13pt)--++(-90:26pt)--++(-140:13pt)--++(-20:13pt)--++(-90:26pt);
\draw[-,line width=0.5mm,blue] (90pt,13pt)++(-140:26pt)++(-20:13pt)--++(-90:13pt)--++(-20:2pt);
\draw[-,line width=0.5mm,blue] (90pt,13pt)++(-140:39pt)++(-20:26pt)--++(-90:13pt)--+(-20:13pt)++(-140:13pt)++(-20:13pt)++(-90:13pt)--+(-90:13pt)++(-90:13pt)--++(-20:13pt)--++(-140:13pt)--++(-90:13pt);
\draw[-] (90pt,6.5pt)++(-140:65pt)++(-20:63pt)++(-90:39pt) node[scale=0.75]{$\textcolor{violet}{0}$};
\draw[-,line width=0.5mm,blue] (90pt,13pt)++(-140:52pt)++(-20:52pt)++(-90:52pt)--++(-140:13pt)--++(-20:13pt) ;
\draw[-,line width=0.5mm,blue] (90pt,13pt)++(-140:13pt)--++(90:13pt)--++(40:13pt)--++(-90:13pt);
\draw[-] (90pt,6.5pt)++(160:6.5pt)++(90:5.5pt)++(40:1pt) node[scale=0.75]{$\textcolor{green!50!black}{0}$};
\end{tikzpicture}
\caption{An example of a plane $\s$-partition.}
\label{3d}
\end{figure}

 One should think that each horizontal layer corresponds to a basis vector in the corresponding Fock space. In our example, we have 4 layers, each being an $\s$-partition:
\begin{align}\label{layers ex}
\ket{(5,5,4,4),(5,5,5,1)}^+_0\otimes \ket{(4,4,3,2),(4,3,2)}^+_0\otimes\ket{(3,2,1),(4,1,1)}^+_0\otimes\ket{(3,2,1),(4)}^+_0.
\end{align}
In the pure MacMahon module, we also multiply by the tail consisting of pairs of empty partitions. The tail is not shown in the picture.

Note the following features. The black and the blue numbers indicate the colors of the box. All boxes stacked vertically on each other carry the same color. The box of color zero with coordinates $(0,0,0)$ is the bottom box of the column below the zero box at the top of the picture (not visible).

The blue surface (broken plane) indicates the boundary between $\bs\lambda^{(i)}$ and $\bs\mu^{(i)}$. 

The blue number $i$ indicates that this box is $i$-convex (removable). The green number $i$ indicates that there is an  $i$-concave box, which could be attached along the facet where the green $i$ is written.

We also have violet zeroes indicating that there is a convex box of color $0$ on the top of that facet considered as part of the $\bs \mu$ partitions. We now define the degree of each violet zero as the number of (imaginary) violet zeros above it if each layer was treated separately. In our example, both violet zeroes have degree 2.

The combinatorial rules for plane $\s$-partitions are as follows. Each layer should satisfy the rules described before for the Fock module, see Figure \ref{s-tab}. Every box is on top of an existing box.
There is a vertical repetition rule. The facets on the right of the blue surface cannot have boxes of color $i$ on the top of each other if $s_{i+1}=-1$. The facets on the left of the blue surface cannot have boxes of color $i$ on top of each other if $s_{i}=-1$. In our example, we cannot have boxes of colors $3,4$ on top of each other on facets to the right of the blue surface, and $0,4$ on the left. Note that this is exactly the opposite of the horizontal rule: you can repeat a box vertically on a facet if and only if you cannot repeat it horizontally.

The rules for the action of generators of $\Es$ are essentially the same.
The $K_i^\pm(z)$ act diagonally with the eigenvalues given by the product over convex and concave boxes on the right of the blue surface given by \eqref{V eigen}, and by \eqref{W eigen} to the left of that. Note that the latter case includes the violet convex boxes of the degree defined above. In addition, in the eigenvalue of $K_0^\pm(z)$, the numerator of $\psi_1$ corresponding to the very top concave zero has to be replaced by the factor $(K-K^{-1}u/z)$.

The operators $F_i(z)$ add boxes of color $i$. Choose a box $\Box$ of color $i$ which can be added without breaking the combinatorial rules. Suppose it is on the $j$-th level (that is it has coordinates $(a,b,j)$).
Then, the corresponding matrix coefficient for adding this box is the product of the matrix coefficient of $F_i(z)$ adding that box to the $j$-th layer (as described before) $\ket{\lambda^{(j)},\mu^{(j)}}$ in the corresponding Fock space multiplied by
 the value of $K_i^+(z)$ acting on the diagram obtained by removing all boxes with coordinates $(a',b',j')$ with $j'\geq j$ with an additional sign if $|i|=-1$ and the total parity of first $j-1$ layers is also odd.

Similarly, operators $E_i(z)$ remove boxes of color $i$. Choose a box $\Box$ of color $i$ which can be removed without breaking the combinatorial rules. Suppose it is on the $j$-th level (that is it has coordinates $(a,b,j)$).
Then, the corresponding matrix coefficient for removing this box is the product of the matrix coefficient of $E_i(z)$ removing that box from the $j$-th layer (as described before) $\ket{\lambda^{(j)},\mu^{(j)}}$ in the corresponding Fock space multiplied by the value of $K_i^-(z)$ acting on the diagram obtained by removing all boxes with coordinates $(a',b',j')$ with $j'\leq j$ with possibly an additional sign as in $F_i(z)$ case.

\medskip

Finally, we show the picture of the vacuum state in the MacMahon module with boundaries $\mc M^{(\bs\gamma,\bs\epsilon),\alpha}(u)$, constructed in Theorem \ref{MacMahon with boundaries}, in Figure \ref{3d boundary}.
\begin{figure}[ht]
\centering
\begin{tikzpicture}
\filldraw[fill=gray, fill opacity=0.5]  (90pt,0pt)--++(-30:39pt)--++(-90:39pt)--++(-30:52pt)--++(-90:52pt)--++(-150:39pt)--++(150:52pt)--++(-150:65pt)--++(150:39pt)--++(90:52pt)--++(30:52pt)--++(90:39pt)--++(30:52pt);
\draw[-](90pt,0pt)++(-30:39pt)--++(-150:26pt)--++(150:13pt)--++(-150:26pt)--++(150:26pt);
\draw[-](90pt,0pt)++(-30:39pt)++(-150:26pt)--+(-90:52pt)++(150:13pt)--+(-90:78pt)++(-150:26pt)--++(-90:78pt)--+(30:26pt);
\draw[-](90pt,0pt)++(-30:39pt)++(-90:39pt)--++(-150:13pt)--++(-90:13pt)--++(-150:13pt)--++(-90:26pt)--++(-150:13pt)--++(-90:13pt);
\draw[-](90pt,0pt)++(-30:91pt)++(-90:39pt)--++(-150:13pt)--++(-90:13pt)--++(-150:13pt)--++(-90:26pt)--++(-150:13pt)--++(-90:13pt);
\draw[-](90pt,0pt)++(-30:91pt)++(-90:39pt)++(-150:13pt)--+(150:52pt)++(-90:13pt)--+(150:52pt)++(-150:13pt)--+(150:52pt)++(-90:13pt)--+(150:52pt)++(-90:13pt)--+(150:65pt)++(-150:13pt)--+(150:65pt);
\draw[-](90pt,0pt)++(-30:91pt)++(-90:52pt)++(-150:13pt)--+(30:13pt)++(-150:13pt)++(-90:13pt)--+(30:26pt)++(-90:13pt)--+(30:26pt)++(-90:13pt)--+(90:13pt)++(30:13pt)--+(90:39pt);
\filldraw[fill=blue!50!gray, fill opacity=0.35] (90pt,0pt)++(-150:52pt)++(-90:39pt)--++(-30:13pt)--++(-90:13pt)--++(-30:13pt)--++(-90:26pt)--++(30:26pt)--++(-30:13pt)--++(90:26pt)--++(30:13pt)--++(90:13pt)--++(30:13pt)--++(-30:52pt)--++(-90:52pt)--++(-150:39pt)--++(150:52pt)--++(-150:65pt)--++(150:39pt)--++(90:52pt)--++(30:52pt);
\draw[-](90pt,0pt)++(-150:52pt)++(-90:39pt)++(-30:13pt)--+(-150:52pt)++(-90:13pt)--+(-150:52pt)++(-30:13pt)--+(-150:52pt)++(-90:13pt)--+(-150:52pt)++(-90:13pt)--+(-150:52pt)++(-30:13pt)++(30:13pt)--+(-150:65pt);
\draw[-](90pt,0pt)++(-150:104pt)++(-90:39pt)--++(-30:13pt)--++(-90:13pt)--++(-30:13pt)--++(-90:26pt)--++(-30:13pt)--++(-90:13pt);
\draw[-](90pt,0pt)++(-150:104pt)++(-90:52pt)++(-30:13pt)--+(-90:39pt)++(-30:13pt)++(-90:26pt)--+(-90:13pt);
\draw[-](90pt,0pt)++(-150:104pt)++(-90:52pt)--+(-30:13pt)++(-90:13pt)--+(-30:26pt)++(-90:13pt)--+(-30:26pt);
\draw[-](90pt,0pt)++(-150:13pt)--+(-30:39pt)++(-150:13pt)--+(-30:26pt)++(-150:13pt)--+(-30:26pt);
\draw[-](90pt,0pt)++(-30:13pt)--+(-150:52pt)++(-30:13pt)--+(-150:26pt);
\draw[-](90pt,0pt)++(-30:39pt)++(-90:13pt)--++(-150:26pt)--++(150:13pt)--++(210:26pt)--++(150:26pt);
\draw[-](90pt,0pt)++(-30:39pt)++(-90:26pt)--++(-150:26pt)--++(150:13pt)--++(210:26pt)--++(150:26pt);
\draw[-](90pt,0pt)++(-30:39pt)++(-90:39pt)++(-150:13pt)--++(-150:13pt)--++(150:13pt)--++(210:26pt)--++(150:13pt);
\draw[-](90pt,0pt)++(-30:39pt)++(-90:52pt)++(-150:26pt)--++(150:13pt)--++(210:26pt);
\draw[-](90pt,0pt)++(-30:39pt)++(-90:65pt)++(-150:26pt)--++(150:13pt)--++(210:26pt);
\draw[-](90pt,0pt)++(-150:52pt)++(-30:13pt)--+(-90:39pt)++(-30:13pt)++(30:13pt)--+(-90:78pt)++(-30:13pt)++(30:26pt)--+(-90:39pt);
\draw[-] (90pt,-6.5pt) node[scale=0.75]{$0$}++(-30:13pt) node[scale=0.75]{$1$}++(-30:13pt) node[scale=0.75]{$2$}++(-30:6.5pt)++(-90:33pt)
node[scale=0.75]{$\textcolor{green!50!black}{3}$}++(-150:13pt)++(-90:13pt)
node[scale=0.75]{$\textcolor{green!50!black}{2}$}++(150:13pt)++(-150:13pt)++(-90:26pt)
node[scale=0.75]{$\textcolor{green!50!black}{0}$}++(-30:13pt)++(-150:6.5pt)
node[scale=0.75]{$\textcolor{violet}{0}$}++(-150:13pt)++(-30:6.5pt)++(-90:6pt)
node[scale=0.75]{$\textcolor{green!50!black}{0}$}++(-30:13pt)++(-150:6.5pt)
node[scale=0.75]{$\textcolor{violet}{0}$};
\draw[-] (90pt,-6.5pt)++(-150:13pt) node[scale=0.75]{$4$}++(-30:13pt) node[scale=0.75]{$0$}++(-30:13pt) node[scale=0.75]{$1$};
\draw[-] (90pt,-6.5pt)++(-150:26pt) node[scale=0.75]{$3$}++(-30:13pt) node[scale=0.75]{$4$}++(-30:13pt);
\draw[-] (90pt,-6.5pt)++(-150:39pt) node[scale=0.75]{$2$}++(-30:13pt) node[scale=0.75]{$3$}++(-30:13pt);
\draw[-] (90pt,-6.5pt)++(-150:45.5pt)++(-90:33pt)  node[scale=0.75]{$\textcolor{green!50!black}{1}$}++(-30:13pt)++(-90:13pt)  node[scale=0.75]{$\textcolor{green!50!black}{2}$};
\draw[-,line width=0.5mm,blue] (90pt,0pt)--++(-150:13pt)--++(-30:13pt)--++(-150:13pt)--++(-30:13pt)--++(-90:78pt)--++(-150:13pt)--++(-30:13pt)--++(-90:13pt)--++(-150:13pt)--++(-30:13pt);
\draw[-,line width=0.5mm,blue] (90pt,0pt)++(-150:39pt)++(-30:26pt)--++(-90:78pt)++(90:78pt)--+(30:13pt)++(-90:13pt)--+(30:13pt)++(-90:13pt)--+(30:13pt)++(-90:13pt)--+(30:13pt)++(-90:13pt)--+(30:13pt)++(-90:13pt)--+(30:13pt)++(-90:13pt)++(-150:13pt)++(-30:13pt)--++(30:13pt)++(-150:13pt)--++(-90:13pt);
\draw[-](90pt,0pt)++(-30:52pt)++(-150:78pt)++(-90:52pt)--++(-30:13pt)--+(-90:13pt);
\end{tikzpicture}
\caption{The vacuum  of $\mc M^{(\bs \gamma,\bs \epsilon),\bs \alpha}_0$, $\bs\gamma=(3,2)$, $\bs\epsilon=(3,2)$, $\bs\alpha=(3,2,2,1)$. }
\label{3d boundary}
\end{figure}

For the vertical boundary (shown in gray), we chose the admissible self-comparable pair of partitions $(\bs\gamma,\bs\epsilon)=((3,2),(3,2))$. Note that self-comparable combinatorially means that the pair of partitions can be placed on top of itself.  In particular, in accordance with the general rule, looking from the right, we can see on the facets only boxes of colors $1$ and $2$ (more importantly, we do not see any boxes of colors $3,4$). And looking from the left, we see only boxes of colors $(2,3)$ (no boxes of colors $0,4$).

For the horizontal boundary (shown in dark blue), we chose the partition $\bs \alpha=(3,2,2,1)$. 

Then, the boxes are added to that shape, according to the same rules, and the action of the generators is described in the same way as well.

We again showed the broken blue surface and concave boxes. Then, the eigenvalue of $K_0^\pm(z)$ is $\psi(q_2^3 u/z)f^{(-1)}(K,u/z)$, the eigenvalue of $K_1^\pm(z)$ is $\frac{(q-q^{-1}q_1^{-1}u/z)}{(q^4-q^{-4}u/z)}q_1^{1/2}$, the eigenvalue of $K_2^\pm(z)$ is $\psi(q_1^{-2}q_2^{-2})\psi(q_3q_2^{-2})$, the eigenvalue of $K_3^\pm(z)$ is $\frac{(1-q_2^{-1}q_3^{2}u/z)}{(1-q_2^{-4}q_1^{-3}u/z)}q_3^{1/2}$, and the eigenvalue $K_4^\pm(z)$ is $1$.

\medskip

{\it Non-generic levels and MacMahon modules with prohibited boxes.}
In Theorems \ref{pure thm 1 Macmahon} and \ref{gamma thm Macmahon}, we assumed that the level $K$ is generic. More precisely, $(K-K^{-1}u/z)$ does not vanish when multiplied by the delta function coming from the action of $E_i(z)$  of $F_i(z)$.

If the matrix coefficient of $F_i(z)$ in a MacMahon module $\mc M(K;u)$ corresponding to an addition of a box contains $\delta(v/z)$, then,  at level $K$ satisfying $K^2=u/v$, the MacMahon module contains a submodule consisting of all plane-super partitions which do not 
contain this box.

\medskip

Let $s_0=-1$. Then, at $K=q$ the MacMahon module $M_0(q;u)$ has the same highest weight with the Fock space $\mc F_{\Lambda_0}(u)$. The module $\mc F_{\Lambda_0}(u)$ is a subquotient of  $M_0(q;u)$. Indeed, $\s$-partitions are plane $\s$-partitions that do not contain the box with coordinates $(0,0,1)$.

Similarly, at $K=q^{-1}$ we prohibit the box with coordinates $(1,1,0)$. Then, the irreducible quotient is the Fock module $\mc F_{-\Lambda_0}(u)$. The columns of plane $\s$-partitions of $\mc M_0(q^{-1},u)$ with prohibited box $(1,1,0)$ are diagonals of $\s$-partitions of $\mc F_{-\Lambda_0}(u)$.

\medskip

{\it The proof of modules being tame.}
We need to show that no two pictures have the same eigenvalues of $K_{i}^\pm(z)$. 

We start with $\mc F_{\Lambda_0}(u)$, see Figure \ref{s-tab}. First, the $q$-content \eqref{q content} and the color uniquely determine the box. The poles of $K_i^\pm(z)$ are exactly the $q$-content of the boxes which can be added or removed. Note that we can always add a box to the first row and to the first column (unless our picture is empty).  Thus, looking at the poles, we can determine the length of the first column and of the first row. Now, delete the first row, the first column, change the $K_i^\pm(z)$ accordingly, and repeat. This restores the shape uniquely.

Other Fock spaces are treated similarly.

\medskip

Consider now the MacMahon module $\mc M_0(u)$, see Figure \ref{3d}. It is sufficient to show how to recover the bottom layer. Note that, if a color and $q$-content of the box is given, then the position of the box is uniquely determined up to a multiple of $(1,1,1)$. Thus, if $i$ is the largest such that some $K_s^\pm(z)$ has a pole with coordinated $(i,0,0)$, then the first row of the first layer has $i-1$ boxes. Similarly, if
$j$ is the largest such that some $K_s^\pm(z)$ has a pole with coordinated $(0,j,0)$, then the first column of the first layer has $j-1$ boxes. We again remove the first row and first column from the shape (ending to a possible non-admissible picture) and divide by the corresponding $K$-values. Then, we repeat and find the second row and second column in the same way.

More general MacMahon modules are treated similarly.

\qed

\section{Characters}\label{char sec}
All $\Es$-modules we have constructed are weighted and graded.
In this section, we discuss their characters, see \eqref{character def}.

Set $p=\prod_{i\in \hat I} z_i$. 
We use the periodic convention  $z_{i+N}=z_i$.

Note that $z_i=e^{-\alpha_i}$, $i\in \hat{I}$, $p=e^{-\delta},$ and, combinatorially, the variable $z_i$ counts the number of boxes of color $i$.

The character of a module does not depend on the choice of the evaluation parameter $u$. We skip $u$ from our notation of all modules in this section, writing, for example, simply $V,W,\mc F_{\Lambda_i},\mc M_i$ for the modules $V(u)$, $W(u),\mc F_{\Lambda_i}(u),\mc M_i(u)$.

\subsection{Characters of vector and covector representations}
Consider the vector representation $V$. 
The grading of $V$ is defined up to a shift. Let us choose $\deg [u]_{-1}=(0,\dots,0)$. Then, the character is given by
\begin{align*}
&\chi(V)=
(\sum_{j=-1}^\infty \prod_{i=0}^j z_i+\sum_{j=-1}^{-\infty} \prod_{i=-1}^j z_i^{-1})=\\
& (1+z_0+z_0z_1+\dots+z_0z_1 \dots z_{N-2})\big(\frac1{1-p}+\frac{p^{-1}}{1-p^{-1}}\big)=\big(\sum_{j=-1}^{N-2}\prod_{i=0}^j z_i\big)\delta(p).
\end{align*}
Similarly, choose $\deg[u]^0=(0,\dots,0)$, then 
\begin{align*}
&\chi(W)=
\big(\sum_{j=0}^{N-1}\prod_{i=1}^j z_{-i}\big)\delta(p).
\end{align*}
We have $\chi(V)(z_0^{-1},\dots,z_{N-1}^{-1})=\chi(W)(z_0,\dots,z_{N-1})$.

\subsection{Characters of Fock spaces} We describe a straightforward way to write a closed fermionic type character formula for the Fock spaces we have constructed.

We choose the degree of the highest weight vector to be $(0,\dots,0)$.
We use the standard notation $(p)_s=\prod_{i=1}^s(1-p^i)$, where $s\in\Z_{\geq 0}$.

First, we consider the pure cases.

Let $s_0=-1$.  For a vector $\ket{\bs \lambda,\bs \mu}_0^+\in\mc F_{\Lambda_0}$, we have defined sequences of non-negative integers $a=(a_0,\dots,a_{N-1})$, $b=(b_0,\dots,b_{N-1}) \in \Z^{|\hat I|}_{\geq 0}=\Z_{\geq 0}^N$, see \eqref{a and b}. Note that $a_i$ is the number of (non-zero) rows of $\bs \lambda$ which end with a box of color $i$. In particular, we keep the periodicity conventions: $a_{i+N}=a_i$ and $b_{i+N}=b_i$.  Next, we append $\bs\mu$ with a number of zero parts so that the number of parts of $\bs\mu$ is equal to the number of non-zero parts of $\bs\lambda$. By definition, such empty parts of $\bs\mu$ always end with a box of color $0$. Then, $b_i$ is the number of columns (number of parts of $\bs \mu$) that end with a box of color $i$. In particular, we have $|a|=\sum_i a_i=|b|=\sum_i b_i$.

For example, for Figure \ref{s-tab} we have $a=(1,0,1,2,2)$, $b=(4,0,0,1,1)$. 

We sum up the contribution of vectors with the same $a$ and $b$ to the character and obtain the following lemma. 

For $i,j\in\Z$, $i\leq j$, set
$$
z_{ij}=\prod_{s=i}^jz_s.
$$
We have $z_{i,i+N-1}=p$.

\begin{lem}
The character of the module $\mc F_{\Lambda_0}$ is given by

\begin{align}\label{ch F0}
    \chi(\mc F_{\Lambda_0})=\sum_{\substack{a,b\in\Z_{\geq 0}^N\\ |a|= |b|}}p^{-b_0}\prod\limits_{j=0}^{N-1}z_{0,j}^{a_j}\,z_{j,N-1}^{b_j}\,\dfrac{  p^{ a_j(a_j-1)(1+s_{j+1})/4+b_j(b_j-1)(1+s_{j})/4} }{(p)_{a_j}(p)_{b_j}}.
\end{align}

\end{lem}
\begin{proof} The sum over the vectors with the same $a$ and $b$ has the form $m_{a,b}/\prod_{i=0}^{N-1}(p)_{a_i}(p)_{b_i}$, where $m_{a,b}$ is the lowest configuration for the given $a,b$. The monomial $m_{a,b}$ is easy to compute and it depends on whether same rows ending at a box of color $i$ are allowed (that is $s_{i+1}=-1$) and whether same columns ending at a box of color $i$ are allowed (that is $s_i=-1$). 
\end{proof}

Similarly, we obtain analogues formulas for other pure cases.

\begin{lem}\label{lemma ch Fk}
The character of the module $\mc F_{s\Lambda_k}$ is given by
\begin{align}\label{ch Fk}
    \chi(\mc F_{s\Lambda_k})=\sum_{\substack{a,b\in\Z_{\geq 0}^N\\ |a|= |b|}}p^{-b_k}\prod\limits_{j=k}^{N+k-1}z_{k,j}^{a_j}\,z_{j,N+k-1}^{b_j}\,\dfrac{  p^{ a_j(a_j-1)(1+ss_{j+1})/4+b_j(b_j-1)(1+ss_{j})/4} }{(p)_{a_j}(p)_{b_j}}.
\end{align}

\qed
\end{lem}

To give a feeling for the formulas, we write the first few terms in the principal specialization: $z_i=q$, $p=q^N$. The principal specialization counts the number of boxes ignoring the colors. We denote the principal specialization of $\chi(V)$ by $\chi_q(V)$.

Consider the example in Section \ref{sec: pictures}: $m=3$, $n=2$, and standard parity $s_1=s_2=s_3=1$, $s_4=s_0=-1$.
Then, we have 
\begin{align*}
\chi_q(\mc F_{\Lambda_0})=\chi_q(\mc F_{\Lambda_3})&=1+q+2q^2+4q^3+6q^4+10q^5+15q^6+22q^7+33q^8+48q^9+70q^{10}+\cdots \,,\\
\chi_q(\mc F_{\Lambda_1})= \chi_q(\mc F_{\Lambda_2})&= 1 + q + 2 q^2 + 3 q^3 + 5 q^4 + 8 q^5 + 12 q^6 + 19 q^7 + 28 q^8 + 41 q^9 + 60 q^{10}+\cdots\,,\\
\chi_q(\mc F_{\Lambda_4})&= 1 + q + 3 q^2 + 6 q^3 + 10 q^4 + 17 q^5 + 27 q^6 + 42 q^7 + 63 q^8 + 
 94 q^9 + 139 q^{10}+\cdots\,.
\end{align*}

Now, we turn to the non-pure cases.

The module $\mc F_{s_{k}( r\Lambda_{k}-(r+1)\Lambda_{k-1})}$ is constructed by taking the tensor product of $\mc F_{-s_{k}(\Lambda_{k-1})}$ with $r$ extra factors $V(u)$, see Theorem \ref{non pure 1 0,r} and Figure \ref{nu mu la V}.

\begin{lem}
The character $\chi(k,r)$ of the module $\mc F_{s_{k+1}( r\Lambda_{k+1}-(r+1)\Lambda_{k})}$ is given by
\begin{align}\label{ch Fkr}
    \chi(k,r)=\sum_{\substack{a,b\in\Z_{\geq 0}^N \\ |a|= |b|+r}}p^{-b_k}z_k^{-r}\prod\limits_{j=k}^{N+k-1}z_{k,j}^{a_j}\,z_{j,N+k-1}^{b_j}\,\dfrac{ p^{ a_j(a_j-1)(1-s_{k+1}s_{j+1})/4+b_j(b_j-1)(1-s_{k+1}s_{j})/4} }{ (p)_{a_j}(p)_{b_j}}.
\end{align}
\qed
\end{lem}

Finally, we give the formula corresponding to Theorem \ref{non pure 2 0,r}  and Figure \ref{nu mu la +}.

\begin{lem} 
The character $\chi(k,r)$ of the module $\mc F_{s_k(r\Lambda_{k-1}-(r+1)\Lambda_k)}$ is given by
\begin{align}
    \chi(k,r)=\sum_{\substack{a,b\in\Z_{\geq 0}^N \\ |a|= |b|-r}}p^{-b_k}\prod\limits_{j=k}^{N+k-1}z_{k,j}^{a_j}\,z_{j,N+k-1}^{b_j}\,\dfrac{  p^{ a_j(a_j-1)(1-s_ks_{j+1})/4+b_j(b_j-1)(1-s_ks_{j})/4} }{(p)_{a_j}(p)_{b_j}}.
\end{align} \qed
\end{lem}

\subsubsection{Comparison of characters}
We compare the formulas for characters of Fock spaces to the formulas of \cite{KW} for integrable level one $\widehat{\gl}_{m|n}$ modules given in the standard parity. 

Note that we normalized our formulas in such a way that the contribution of the highest weight vector is 1. In \cite{KW}, the contribution of the highest weight vector of weight $\Lambda$ is given by $e^\Lambda$. In our discussion below, we ignore this trivial discrepancy. 

Let $s_i=1$, $i=1,\dots,m$, and $s_i=-1$, $i=m+1,\dots,N$.
The character \eqref{ch F0} can be written as
\begin{align*}
    \chi(\mc F_{\Lambda_0})=\sum_{|a|= |b|}\dfrac{ \prod\limits_{j=0}^{N-1}z_j^{\sum_{i=j}^{N-1} a_i+\sum_{i=1}^{j} b_i} p^{\sum_{i=0}^{m-1} a_i(a_i-1)/2+\sum_{i=1}^{m}b_i(b_i-1)/2} }{\prod\limits_{i=0}^{N-1}(p)_{a_i}(p)_{b_i}}.
\end{align*}
And making the change of variables $z_i\mapsto e^{\epsilon_{i+1}-\epsilon_{i}}p^{\delta_{i,0}}$, we get
\begin{align*}
    \chi(\mc F_{\Lambda_0})&=\sum_{|a|= |b|}\dfrac{ \prod\limits_{i=0}^{N-1}e^{\epsilon_i(a_{i-1}-b_i)}  p^{\sum_{i=0}^{N-1}a_{i}+\sum_{i=0}^{m-1}\left( (a_i^2+b_{i+1}^2)-(a_i+b_{i+1})/2\right)}}{\prod\limits_{i=0}^{N-1}(p)_{a_i}(p)_{b_i}}\\
    &=\sum_{|a|= |b|}\dfrac{ \prod\limits_{i=0}^{N-1}e^{\epsilon_i(a_{i-1}-b_i)}  p^{\sum_{i=0}^{m-1} (a_i^2+b_{i+1}^2)+\sum_{i=m}^{N-1}(a_{i}+b_{i+1})/2}}{\prod\limits_{i=0}^{N-1}(p)_{a_i}(p)_{b_i}},
\end{align*}
which (after multiplication by $e^{\Lambda_0}$) coincides with the formula (3.14) in \cite{KW} with $s=0$, giving the character of the integrable $\widehat{\gl}_{m|n}$ module of highest level $\Lambda_0$  (where $a_i$ and $p$ correspond to $a_{i+1}$ and $q$ in \cite{KW}). 

Similarly, the character $\chi(\mc F_{(r+1)\Lambda_m-r\Lambda_{m+1}})$, $r\in\Z_{\geq 0}$ coincides with (3.14) in \cite{KW} with $s=m+r$. Also $\chi(\mc F_{(r+1)\Lambda_0-r\Lambda_{-1}})$, $r\in\Z_{\geq 0}$ coincides with   (3.14) in \cite{KW} with $s=-r$.

\medskip

 To compare the characters $\chi(\mc F_{\Lambda_k})$, $k=0,\dots,m,$ we first rewrite it as follows.

\begin{lem}
The character \eqref{ch Fk}  of the module $\mc F_{\Lambda_k}$ can be also written
\begin{align}\label{ch Fl}
    \chi(\mc F_{\Lambda_k})=\sum_{\substack{a,b\in\Z_{\geq 0}^N \\ |a|= |b|+k}}\dfrac{\prod\limits_{j=0}^{N-1}z_j^{\sum_{i=j}^{N-1} a_i+\sum_{i=1}^{j} b_i} p^{\sum_{i=0}^{N-2} a_i(a_i-1)(s_{i+1}+1)/4+\sum_{i=1}^{N-1}b_i(b_i-1)(s_{i}+1)/4} }{\prod\limits_{j=0}^{k-1}z_j^{k-j}\prod\limits_{i=0}^{N-1}(p)_{a_i}(p)_{b_i}}.
\end{align}
\end{lem}
\begin{proof}
 Recall that the basis of $\chi(\mc F_{\Lambda_k})$, $k=0,\dots,m$, is parameterized by pairs $(\bs\la, \bs\mu)$ of partitions with $\ell(\bs \la)\geq \ell(\bs\mu)$, see Section \ref{sec non pure fock 1}. 
 
 Now we treat it differently.  Namely, we move all boxes in $k$ diagonals below the blue line from $\bs \mu$ to $\bs \lambda$ creating new partitions $\bs{\bar\la}$ and $\bs{\bar\mu}$. As the result,  $\bs{\bar\la}$ has $k$ more parts than and $\bs{\bar\mu}$.
 
 An example is shown in Figure \ref{la mu bar}. Initially, we have
 $\bs \lambda=(4,2,2,1,1,1)$ and $\bs \mu=(4,4,3,3,2,1)$ according to the dashed blue broken line. Now we have $\bs{\bar\la}=(7,5,5,4,4,4,3)$ and $\bs{\bar\mu}=(1,1,0,0)$ according to the solid broken blue line.

\begin{figure}[H]
\centering
\begin{tikzpicture}
\node at (0pt,7pt) {$\ket{\bs\lambda,\bs\mu}$:};
\tgyoung(79pt,0cm,3401,234,1234,0123,40123,:;40123,:::;012)
\draw[-,line width=0.5mm,green!50!black] (79pt,-26pt)-- ++(52pt,0pt);
\draw[-,line width=0.5mm,blue] (40pt,13pt)-- ++(0pt,-13pt)-- ++(13pt,0pt)-- ++(0pt,-13pt)-- ++(13pt,0pt)-- ++(0pt,-13pt)-- ++(13pt,0pt)-- ++(0pt,-13pt)-- ++(13pt,0pt)-- ++(0pt,-13pt)-- ++(13pt,0pt)-- ++(0pt,-13pt)-- ++(13pt,0pt)-- ++(0pt,-13pt)-- ++(13pt,0pt);
\draw[-,line width=0.5mm,blue,dashed] (79pt,13pt)-- ++(0pt,-13pt)-- ++(13pt,0pt)-- ++(0pt,-13pt)-- ++(13pt,0pt)-- ++(0pt,-13pt)-- ++(13pt,0pt)-- ++(0pt,-13pt)-- ++(13pt,0pt)-- ++(0pt,-13pt)-- ++(13pt,0pt)-- ++(0pt,-13pt)-- ++(13pt,0pt)-- ++(0pt,-13pt);
{\Yfillcolour{gray}
\Yfillopacity{.5}
\tgyoung(40pt,0pt,_3,:_2,::_1)}
\end{tikzpicture}
\caption{Two ways to see a diagram for $\ket{\bs\lambda,\bs\mu}\in \mathcal{F}_{\Lambda_3}(u)$.}
\label{la mu bar}
\end{figure}
Formally, partitions $(\bs{\bar{\la}}, \bs{\bar{\mu}})$ are defined as follows:
\begin{itemize}
    \item For $i=1,\dots,k$, let  $\bar{\la}_i:=\la_{i}+\#\{i-k\leq j<i\ |\ \mu_j\geq i-j\}+k+1-i$;
    \item For $i>k$, let  $\bar{\la}_i:=\la_{i}+\#\{i-k\leq j<i\ |\ \mu_j\geq i-j\}$;
    \item  For $i>0$, let $\bar{\mu}_i:=\max(\mu_{i}-k,0)$.
\end{itemize}
 Here, if $i>\ell(\bs\la)$, we set $\la_i=0$ and if $i>\ell(\bs\mu)$, we set $\mu_i=0$.

Clearly $\bs{\bar{\la}}$ has at least $k$ non-zero more parts than $\bs{\bar{\mu}}$. Thus a basis of $\mc F_{\Lambda_k}$ is labeled by pairs $(\bs{\bar{\la}}, \bs{\bar{\mu}})$, where $\bs{\bar{\la}}$ is an $\s^-_+$ partition,  $\bs{\bar{\mu}}$ is an $\s^-_-$ partition and $\ell(\bs{\bar{\la}})\geq \ell(\bs{\bar{\mu}})+k$.

In fact, this parameterization corresponds to the inclusion
$$
\mc F_{\Lambda_k}(u)\subset V(u)\otimes V(uq_2)\otimes \dots \otimes V(uq_2^{k-1})\otimes \mc F_{\Lambda_0}(uq_2^k).
$$

We append $\bs{\bar{\mu}}$ with zero parts, so that $\bs{\bar{\la}}$ has exactly $k$ extra parts compared to $\bs{\bar{\mu}}$, and define sequences $a,b$ as before. Then, we have $|a|-|b|=k$, and we compute the character in the same way as in Lemma \ref{lemma ch Fk}. Note the new factor in the denominator corresponding to the contribution of $(\bs{\bar{\emptyset}},\bs{\bar{\emptyset}})=((k,k-1,k-2,\dots,1),\emptyset)$.

\end{proof}

The formula \eqref{ch Fl} coincides with formula (3.14) in \cite{KW} with $s=k$.

\medskip

One can extend the horizontal subalgebra $U_q^{hor}\slh_{m|n}\in\Es$ to a subalgebra $U_q^{hor}\widehat{\gl}_{m|n}\in\Es$. Then, in all these cases, the Fock space is an irreducible module with respect to the horizontal $U_q\widehat{\gl}_{m|n}$ subalgebra.

\subsection{Characters of MacMahon modules}
We discuss the characters of MacMahon modules. 

The computation of the characters of general MacMahon modules is an interesting problem that can be studied by a combination of combinatorial and representation theoretical methods, see, for example, \cite{BFM}.
Here, we only discuss characters of pure MacMahon modules, where the answer is simple.

Recall that the parity of a node $i \in \hat{I}$ is given by $|i|=(1-s_is_{i+1})/2$.

For integers $i,j$, $i\leq j$, set $$\sigma(i,j)=(-1)^{\sum_{k=i}^{j} |k|}.$$ 
Note that $\sigma(i+N,j)=\sigma(i,j+N)=\sigma(i,j)$.

The following theorem enumerating plane $\s$-partitions is a generalization of the MacMahon formula on plane partitions.

\begin{thm}\label{Macmahon char thm}
The character of the module $\mc M_k$ is given by
\begin{align}\label{M char}
    \chi(\mc M_k)=\prod_{i\leq k} \prod_{j\geq k} (1-{\sigma(i,j)}z_{i,j})^{-\sigma(i,j)}.
\end{align}
\end{thm}
\begin{proof} For simplicity of notation, we discuss $\chi({\mc M_0})$ with $s_0=-1$.

We start by constructing a bijection between the set of plane $\s$-partitions and the set of pairs of reverse semistandard Young super-tableaux of the same shape following the well-known algorithm in the even case, see  \cite[Chapter 7.20]{FS}.

Let $\tau$ be a partition. A tableau $T$ is an assignment from each box $(i,j)$ in the Young diagram of $\tau$ to a positive integer $T_{i,j}$.

A reverse semistandard Young super-tableau of parity $\s$ is a tableau $T$ with the restrictions:
\begin{enumerate}
    \item $T_{i,j}\geq T_{i,j+1}$, with equality allowed only if $s_{T_{i,j}}=1$;
    \item $T_{i,j}\geq T_{i+1,j}$, with equality allowed only if $s_{T_{i,j}}=-1$.
\end{enumerate}

Recall that, given a parity sequence $\s$, the parity sequence $\s'$ is defined by $s_i'=s_{1-i}$, for all $i\in \hat{I}$.

A pair $(T^{(1)}, T^{(2)})$ of reverse semistandard Young super-tableau is admissible if $T^{(1)}$ is of parity $\s$ and $T^{(2)}$ is of parity $\s'$.

To each admissible pair $(T^{(1)},T^{(2)})$, we assign a plane $\s$-partition $(\vec{\bs \lambda},\vec{\bs \mu })$ by the following rule.
$$
\lambda_i=(T^{(1)}_{1i},T^{(1)}_{2i},\dots),
\qquad \mu_i=(T^{(2)}_{1i}-1,T^{(2)}_{2i}-1,\dots).
$$
Clearly, this is a bijection.
Moreover, it preserves grading if we make the following assignments.  We set the degree of a box $(i,j)$ in the first partition to $\prod_{s=0}^{T_{ij}-1}z_s$.
We set the degree of a box $(i,j)$ in the second partition to  $\prod_{s=1}^{T_{ij}-1}z_{-s}$. Then, we set the degree of an admissible pair $(T^{(1)},T^{(2)})$ equal to the product of degrees of all boxes.

For example, the plane $\s$-partition in Figure \ref{3d} corresponds to the following pair of reverse super-tableaux, see \eqref{layers ex}:
\begin{figure}[H]
\centering
\begin{tikzpicture}
\node at (0pt,7pt) {$T^{(1)}$:};
\tgyoung(30pt,0cm,5433,5422,4311,42)
\node at (180pt,7pt) {$T^{(2)}$:};
\tgyoung(210pt,0cm,6555,6421,6321,21)
\end{tikzpicture}
\caption{An admissible pair of reverse semistandard super-tableaux}
\label{pair 3d}
\end{figure}
For the first tableau, $1,2,3$ are even and can be repeated in rows, while $4,5$ are odd and can be repeated in columns. For the second tableau, $3,4,5$ are even, and $1,2,6=1+5$ are odd.
The degree of this pair of tableaux is 
$z_0^{17}z_1^{18}z_2^{16}z_3^{14}z_4^{13}$.

Now,  fix a positive integer $L$.
Consider  the set of all reverse super-tableaux of a given shape $\bs\nu$ with $T_{ij}\leq L$ and parity $\s$. Then, this set is in bijection with the set of semistandard (not reverse) Young super-tableaux. The bijection is given by replacing $T_{ij}$ with $L+1-T_{ij}$. The set of semistandard Young super-tableaux parameterizes a basis of the polynomial $\gl_\s$ module $L_{\bs\nu}^{\s}$ with highest weight $\bs\nu$, \cite{BR}. 

Thus, admissible pairs of super-tableaux of a given shape $\bs\nu$ correspond to the $\gl_{\s}\otimes \gl_{\s'}$ module $L_\nu^{\s}\otimes L_\nu^{\s'}$.

Note that the parity and the grading are now counted backwards. Namely, a box with value $a$  in the first tableau has the parity $s_{L+1-a}$ and the grading  $\prod_{s=0}^{L-a}z_s$.  A box with value $a$  in the second tableau has the parity $s_{L+1-a}'$ and degree $\prod_{s=1}^{L-a}z_{-s}$.

By the super-version of $\gl_k\times\gl_l$ duality, see \cite[Theorem 2.1]{S}, \cite[Theorem 3.2]{CW}, see also \cite[Corollaries 10a,b]{BRe} for combinatorial treatment, we have
$$
\mathop{\bigoplus}_\nu L_\nu^{\s}\otimes L_\nu^{\s'}=\C[x_{ij}]_{i,j=1}^L.
$$
Here, the degree  of $x_{ij}$ is $\prod_{s=i-L}^ {L-j}z_s$. The parity of $x_{ij}$ is $\sigma(-i,j)$.

The character of the polynomial ring $\C[x_{ij}]_{i,j=1}^L$ in the $L\to\infty$ limit coincides with the right hand side of \eqref{M char}, and the theorem follows.
\end{proof}

In \eqref{M char}, the parity $\s$ is assumed to be $N$-periodic which correspond to $\Es$ Macmahon modules. Clearly, the same argument shows that the formula is valid for counting plane $\s$-partitions with arbitrary parity sequence $\s$.

One can think that $|i|$ is the parity of simple root $\alpha_i$, $\sigma(i,j)$ is the parity of the positive root $\alpha_i+\alpha_{i+1}+\dots +\alpha_j$, and $\chi_q(\mc M_k)$ is the character of a parabolic Verma module of the Lie superalgebra $\gl_{\infty}^{\s}$ with highest weight $(\dots,0,0,K,0,0,\dots)$, where $K$ is a generic complex number in the $k$-th position.

\medskip

Return to the example in Section \ref{sec: pictures}, that is, to the case of $m=3$, $n=2$, and standard parity $s_1=s_2=s_3=1$, $s_4=s_0=-1$.

Formula \eqref{M char} in the principal specialization takes the form
\begin{align*}
\chi_q(\mc M_0)= \chi_q(\mc M_3)= \prod_{k=0}^\infty \dfrac{(1+q^{5k+1})^{2k+1}(1+q^{5k+2})^{4k+2}(1+q^{5k+3})^{4k+2}(1+q^{5k+4})^{2k+1}}{(1-q^{5k+1})^{3k}(1-q^{5k+2})^{k}(1-q^{5k+3})^{k+1}(1-q^{5k+4})^{3k+3}(1-q^{5k+5
})^{5k+5}}\\ =1+q+2 q^2+5 q^3+8 q^4+16 q^5+29 q^6+50 q^7+88 q^8+150 q^9+254 q^{10}+\cdots,
\end{align*}

\begin{align*}
\chi_q(\mc M_1)= \chi_q(\mc M_2)=  \prod_{k=0}^\infty \dfrac{(1+q^{5k+1})^{2k}(1+q^{5k+2})^{4k+1}(1+q^{5k+3})^{4k+3}(1+q^{5k+4})^{2k+2}}{(1-q^{5k+1})^{3k+1}(1-q^{5k+2})^{k+1}(1-q^{5k+3})^{k}(1-q^{5k+4})^{3k+2}(1-q^{5k+5
})^{5k+5}}\\ =1 + q + 3 q^2 + 6 q^3 + 12 q^4 + 23 q^5 + 42 q^6 + 77 q^7 + 136 q^8 + 
 238 q^9 + 410 q^{10} +\cdots ,
\end{align*}

\begin{align*}
\chi_q(\mc M_4)= \prod_{k=0}^\infty \dfrac{(1+q^{5k+1})^{2k}(1+q^{5k+2})^{4k+2}(1+q^{5k+3})^{4k+2}(1+q^{5k+4})^{2k+2}}{(1-q^{5k+1})^{3k+1}(1-q^{5k+2})^{k}(1-q^{5k+3})^{k+1}(1-q^{5k+4})^{3k+2}(1-q^{5k+5
})^{5k+5}}\\ =1 + q + 3 q^2 + 6 q^3 + 11 q^4 + 22 q^5 + 40 q^6 + 72 q^7 + 127 q^8 + 
 221 q^9 + 379 q^{10} +\cdots \ .
\end{align*}

\begin{figure}[H]
\centering
\begin{tikzpicture}
\foreach \x in {-6,...,8} 
        {
        \draw (\x,-.5) node{\tiny \x};
        }   
        
\foreach \y in {0,...,3} 
    \foreach \x in {-6,...,8}
        {
        \draw (\x,\y)+(0,0.74*\y) circle (0.15cm);
        \draw (\x,\y)+(.5, 0.87+0.74*\y) circle (0.15cm);
        }

\foreach \x in {-1,...,1} 
        {
        \draw[shift={(5*\x,0)}, rotate=45] (-0.15,0)--(0.15,0);
		\draw[shift={(5*\x,0)}, rotate=-45] (-0.15,0)--(0.15,0);
		\draw[shift={(5*\x+3,0)}, rotate=45] (-0.15,0)--(0.15,0);
		\draw[shift={(5*\x+3,0)}, rotate=-45] (-0.15,0)--(0.15,0);
        }

\foreach \t in {0,...,1} 
    {
    \foreach \x in {0,...,1} 
    \foreach \s in {0,...,3}
        {
        \draw[shift={(-2.5+\s+5*\x,0.87+5.22*\t)}, rotate=45] (-0.15,0)--(0.15,0);
		\draw[shift={(-2.5+\s+5*\x,0.87+5.22*\t)}, rotate=-45] (-0.15,0)--(0.15,0);
        }       
        \draw[shift={(7.5,0.87+5.22*\t)}, rotate=45] (-0.15,0)--(0.15,0);
        \draw[shift={(7.5,0.87+5.22*\t)}, rotate=-45] (-0.15,0)--(0.15,0);
        \draw[shift={(8.5,0.87+5.22*\t)}, rotate=45] (-0.15,0)--(0.15,0);
        \draw[shift={(8.5,0.87+5.22*\t)}, rotate=-45] (-0.15,0)--(0.15,0);
        }
        \draw[shift={(-5.5,0.87)}, rotate=45] (-0.15,0)--(0.15,0);
        \draw[shift={(-5.5,0.87)}, rotate=-45] (-0.15,0)--(0.15,0);
        \draw[shift={(-4.5,0.87)}, rotate=45] (-0.15,0)--(0.15,0);
         \draw[shift={(-4.5,0.87)}, rotate=-45] (-0.15,0)--(0.15,0);
        \draw[shift={(-5.5,0.87+5.22)}, rotate=45] (-0.15,0)--(0.15,0);
        \draw[shift={(-5.5,0.87+5.22)}, rotate=-45] (-0.15,0)--(0.15,0);
        \draw[shift={(-4.5,0.87+5.22)}, rotate=45] (-0.15,0)--(0.15,0);
        \draw[shift={(-4.5,0.87+5.22)}, rotate=-45] (-0.15,0)--(0.15,0);
		
\foreach \t in {0,...,1}
    {
    \foreach \x in {-1,...,1} 
        \foreach \s in {0,...,3}
            {
            \draw[shift={(\s+5*\x,1.74+3.48*\t)}, rotate=45] (-0.15,0)--(0.15,0);
		    \draw[shift={(\s+5*\x,1.74+3.48*\t)}, rotate=-45] (-0.15,0)--(0.15,0);
            }
           
        }

\foreach \s in {0,...,1}        
    {\foreach \x in {0,...,1} 
        {
        \draw[shift={(-2.5+5*\x,2.61+1.74*\s)}, rotate=45] (-0.15,0)--(0.15,0);
		\draw[shift={(-2.5+5*\x,2.61+1.74*\s)}, rotate=-45] (-0.15,0)--(0.15,0);
		\draw[shift={(+0.5+5*\x,2.61+1.74*\s)}, rotate=45] (-0.15,0)--(0.15,0);
		\draw[shift={(0.5+5*\x,2.61+1.74*\s)}, rotate=-45] (-0.15,0)--(0.15,0);
        }       
        \draw[shift={(7.5,2.61+1.74*\s)}, rotate=45] (-0.15,0)--(0.15,0);
        \draw[shift={(7.5,2.61+1.74*\s)}, rotate=-45] (-0.15,0)--(0.15,0);  
        \draw[shift={(+0.5-5,2.61+1.74*\s)}, rotate=45] (-0.15,0)--(0.15,0);
		\draw[shift={(0.5-5,2.61+1.74*\s)}, rotate=-45] (-0.15,0)--(0.15,0);
        }

\draw (-4.1,6.26)--(0,-0.87)--(4.1,6.26); 
\draw[blue] (-3.1,6.26)--(1,-0.87)--(5.1,6.26); 
\draw[red] (-2.1,6.26)--(2,-0.87)--(6.1,6.26); 
\draw[green] (-1.1,6.26)--(3,-0.87)--(7.1,6.26); 
\draw[orange] (-0.1,6.26)--(4,-0.87)--(8.1,6.26);
\draw[dashed] (-6.1,3.91)--(8.7,3.91);

\end{tikzpicture}
\caption{Roots of $\gl_\infty^\s$.}
\label{root pic}
\end{figure}

The structure of the formula can be seen from Figure \ref{root pic}, where we show the parities of all $\gl_{\infty}^\s$ roots. The bottom row corresponds to simple roots, the next row to roots which are sums of two simple roots, etc. Note the 5-periodicity along the rows and along the colored lines.
Each colored V line shows which roots contribute to the character. The black V line corresponds to $\mc M_0,$ blue to $\mc M_1$, red to $\mc M_2$, green to $\mc M_3$, and orange to $\mc M_4$. The odd roots (shown by
\begin{tikzpicture}[baseline=-3pt]
\draw (0,0pt) circle (0.15cm);
\draw[rotate=45] (-0.15,0)--(0.15,0);
\draw[rotate=-45] (-0.15,0)--(0.15,0); 
\end{tikzpicture}) give a factor in the numerator, while even roots (shown by \begin{tikzpicture}[baseline=-3pt]
\draw (0,0pt) circle (0.15cm);
\end{tikzpicture}) give a factor in the denominator.

Then, we use the periodicity. For example, the contribution of the row of height $5k+3$ to the character of $\mc M_0$ has $k$ periods, all containing four odd roots and one even root plus three more roots, which have the parities as the roots in the third row inside the black $V$. The total is $4k+2$ on the top and $k$ and $k+1$ on the bottom. 

\medskip

Using the same method as in Theorem \ref{Macmahon char thm}, we can count the plane $\s$-partitions restricted to a ``box'' of infinite height.

Let the color of the box $(0,0,0)$ be $k$. Consider plane $\s$-partitions $(\vec{\bs\la},\vec{\bs\mu})$ corresponding to any fixed parity $\s$.

For $L_1,L_2\in \Z_{\geq 0}$, we say that a plane $\s$-partition $(({\bs \lambda}^{(1)},\bs \mu^{(1)}),\dots, ({\bs \lambda}^{(k)},\bs \mu^{(k)}))$ is $L_1\times L_2$ {\it restricted} if $\lambda^{(1)}_1\leq L_1$, $\mu^{(1)}_1\leq L_2$.

Recall that degree of a plane $\s$-partition is $\prod_{i=0}^{N-1} z_i^{c_i}$, where $c_i$ is the number of boxes of color $i$.

Let $\chi_{L_1,L_2}=\sum \deg(\vec{\bs\la},\vec{\bs\mu})$, where the sum is over all $L_1\times L_2$ restricted plane $\s$-partitions, be the generating function for $L_1\times L_2$ restricted plane $\s$-partitions.

\begin{thm}\label{Macmahon char thm 2}
The generating function $\chi_{L_1,L_2}$ is given by
\begin{align*}
  \chi_{L_1,L_2}=\prod_{k-L_1\leq i\leq k} \prod_{k \leq j\leq k+L_2} (1-{\sigma(i,j)}z_{i,j})^{-\sigma(i,j)}.
\end{align*}
\end{thm}

\bigskip

{\bf Acknowledgments.}
This work was partially supported by the Simons Foundation grants \#353831 and \#709444.

\end{document}